\documentclass[leqno]{article}

\usepackage[frenchb,english]{babel}
\usepackage[utf8]{inputenc}
\usepackage{amsmath}
\usepackage{amssymb}
\usepackage{amsfonts}
\usepackage{enumerate}
\usepackage{vmargin}
\usepackage[all]{xy}
\usepackage{mathrsfs}
\usepackage{mathtools}
\usepackage{esint}
\usepackage{lmodern}
\usepackage{slashed}
\usepackage{esint}
\usepackage{ulem}
\usepackage{stmaryrd} 
\usepackage[colorlinks=true,linkcolor=blue,pagebackref=true]{hyperref}%
\setmarginsrb{3cm}{3cm}{3.5cm}{3cm}{0cm}{0cm}{1.5cm}{3cm}
\usepackage{comment}
\usepackage{tikz}
\usetikzlibrary{patterns}

\newcommand{\C}{\mathrm{C}}
\newcommand{\D}{\ensuremath{\mathcal{D}}}

\newcommand{\B}{\mathrm{B}} 
\let\H\relax 
\newcommand{\H}{\mathrm{H}}
\newcommand{\K}{\mathrm{K}}

\let\L\relax 
\newcommand{\L}{\mathrm{L}}

\newcommand{\scr}{\mathscr}

\newcommand{\even}{\mathrm{even}} 
\newcommand{\odd}{\mathrm{odd}} 
\newcommand{\dR}{\mathrm{dR}} 

\newcommand{\M}{\mathrm{M}} 

\newcommand{\bi}{\mathrm{bi}} 

\newcommand{\dist}{\mathrm{dist}} 

\newcommand{\app}{\mathrm{app}}

\newcommand{\V}{\mathrm{V}}

\let\div\relax
\newcommand{\div}{\mathrm{div}} 
\let\cal\relax
\newcommand{\cal}{\mathcal}
\newcommand{\Z}{\ensuremath{\mathbb{Z}}}
\newcommand{\R}{\ensuremath{\mathbb{R}}}

\newcommand{\W}{\mathrm{W}}

\newcommand{\Id}{\mathrm{Id}} 
\newcommand{\la}{\langle}  
\newcommand{\ra}{\rangle} 

\newcommand{\HdR}{\mathrm{HdR}} 

\newcommand{\Hom}{\mathrm{Hom}} 

\renewcommand{\leq}{\ensuremath{\leqslant}}
\renewcommand{\geq}{\ensuremath{\geqslant}}
\newcommand{\qed}{\hfill \vrule height6pt  width6pt depth0pt}
\newcommand{\bnorm}[1]{ \big\| #1  \big\|}

\newcommand{\norm}[1]{\left\Vert#1\right\Vert}

\newcommand{\xra}{\xrightarrow}
\newcommand{\co}{\colon}

\newcommand{\ot}{\otimes}
\newcommand{\ovl}{\overline}

\newcommand{\sign}{\mathrm{sign}} 

\newcommand{\Fred}{\mathrm{Fred}}
\newcommand{\vol}{\mathrm{vol}} 

\let\i\relax 
\newcommand{\i}{\mathrm{i}} 

\newcommand{\loc}{\mathrm{loc}}
\newcommand{\ov}{\overset}

\newcommand{\epsi}{\varepsilon}
\renewcommand{\d}{\mathop{}\mathopen{}\mathrm{d}} 
\newcommand{\e}{\mathrm{e}} 
\renewcommand{\d}{\mathop{}\mathopen{}\mathrm{d}}

\let\div\relax
\newcommand{\div}{\mathrm{div}}

\DeclareMathOperator{\Lip}{\mathrm{Lip}} 
\DeclareMathOperator{\supp}{supp} 
\DeclareMathOperator{\Index}{Index} 
\DeclareMathOperator{\coker}{Coker} 
\DeclareMathOperator{\sgn}{\mathrm{sgn}} 
\let\ker\relax 
\DeclareMathOperator{\ker}{Ker} 
\DeclareMathOperator{\Ran}{Ran} 
\DeclareMathOperator{\diag}{diag} 
\DeclareMathOperator{\dom}{dom} 
\DeclareMathOperator*{\esssup}{esssup}

\DeclareMathOperator{\diam}{diam} 
\DeclareMathOperator{\rank}{rank} 
\DeclareMathOperator{\grad}{grad}

\DeclareMathOperator{\Curv}{Curv} 
\newtheorem{thm}{Theorem}[section]

\newtheorem{prop}[thm]{Proposition}

\newtheorem{cor}[thm]{Corollary}
\newtheorem{lemma}[thm]{Lemma}

\newtheorem{remark}[thm]{Remark}
\newtheorem{example}[thm]{Example}

\newenvironment{proof}[1][]{\noindent {\it Proof #1} : }{\hbox{~}\qed
\smallskip
}

\usepackage{tocloft}
\numberwithin{equation}{section}
\usepackage[nottoc,notlot,notlof]{tocbibind}

\let\OLDthebibliography\thebibliography
\renewcommand\thebibliography[1]{
  \OLDthebibliography{#1}
  \setlength{\parskip}{0pt}
  \setlength{\itemsep}{0pt plus 0.3ex}
}

\newcommand\reallywidehat[1]{\arraycolsep=0pt\relax%
\begin{array}{c}
\stretchto{
  \scaleto{
    \scalerel*[\widthof{\ensuremath{#1}}]{\kern-.5pt\bigwedge\kern-.5pt}
    {\rule[-\textheight/2]{1ex}{\textheight}} 
  }{\textheight} %
}{0.5ex}\\           
#1\\                 
\rule{-1ex}{0ex}
\end{array}
}

\begin{document}
\selectlanguage{english}
\title{\bfseries{The $\L^p$-index of the Hodge--Dirac operator on compact Riemannian manifolds}}
\date{}
\author{\bfseries{C\'edric Arhancet}}
\maketitle


\begin{abstract}
We investigate the spectral and index-theoretic properties of the closure $D_p$ of the Hodge--Dirac operator $D = \mathrm{d} + \mathrm{d}^*$ acting on the Banach space $\mathrm{L}^p(\Omega^\bullet(M))$ of differential forms over a compact Riemannian manifold $M$. Relying on the compactness of $M$, we establish that this operator is bisectorial and admits a bounded $\mathrm{H}^\infty$ functional calculus, without curvature assumptions. This result enables us to prove that the triple $(\mathrm{C}(M), \mathrm{L}^p(\Omega^\bullet(M)),D_p)$ constitutes a compact Banach spectral triple. We then investigate consistent pairings between the Banach K-homology and the K-theory of the algebra $\mathrm{C}(M)$, identifying the resulting Fredholm indices with classical topological invariants, and hence showing that they are independent of $p$. We recover the classical Euler characteristic and the Hirzebruch signature as $\mathrm{L}^p$-indices, demonstrating the effectiveness of Banach noncommutative geometry for geometric analysis, beyond the Hilbertian setting.
\end{abstract}


\makeatletter
 \renewcommand{\@makefntext}[1]{#1}
 \makeatother

{
  \hypersetup{linkcolor=blue}
 \tableofcontents
}

\section{Introduction}
\label{sec:Introduction}

Index theory for elliptic operators on compact manifolds is a cornerstone of modern geometry, bridging the topology of a manifold and the analytic properties of elliptic operators acting on it. A famous example of this interplay is the Atiyah-Singer index theorem \cite{AtS63}. Classically, this theory is formulated within the framework of Hilbert spaces, relying heavily on the structure of $\L^2$-spaces of sections and the spectral theory of selfadjoint operators. We refer to the book \cite{BlB13} and to \cite{Fre21} for comprehensive accounts. Connes' noncommutative geometry \cite{Con94} distilled the essential ingredients of this interplay into the notion of a spectral triple $(\cal{A}, \cal{H}, D)$, where the algebra $\cal{A}$ encodes the topology of the space and the unbounded selfadjoint operator $D$ acting on the Hilbert space $\cal{H}$ carries the analytic and metric information. For an exhaustive exposition, the reader is also referred to the books \cite{GVF01}, \cite{EcI18}, \cite{Kha13}, \cite{CGRS14} and to the surveys \cite{Con00}, \cite{CoM08}, \cite{HiR06} and references therein.

However, the restriction to the Hilbertian setting $p=2$ can be viewed as an artificial limitation from the perspective of modern harmonic analysis, described in the books \cite{HvNVW16}, \cite{HvNVW18} and \cite{HvNVW23}. Many geometric phenomena and analytic properties are naturally expressed in $\L^p$-spaces for $p \neq 2$. In recent years, there has been a growing interest in extending the tools of noncommutative geometry to the Banach space setting, motivated by the Baum-Connes conjecture \cite{Laf02}, new index theorems \cite{Arh26a}, hidden noncommutative geometry \cite{ArK22,JMP18}, $\L^p$-operator algebras \cite{DFP26} and the study of semigroups of operators on $\L^p$-spaces \cite{Arh24a,Arh24b,Arh26b}.

In this paper, we investigate the spectral and index-theoretic properties of a fundamental geometric operator, the Hodge--Dirac operator $D \co \Omega^\bullet(M) \to \Omega^\bullet(M)$ defined by 
\begin{equation}
\label{def-Hodge-Dirac}
D \ov{\mathrm{def}}{=} \d+\d^*, 
\end{equation}
on the space $\Omega^\bullet(M)$ of complex smooth differential forms on a smooth compact Riemannian manifold $M$ (without boundary), where $\d$ is the exterior derivative. Note that its square is the Hodge-de Rham Laplacian
\begin{equation}
\label{Hodge-deRham-Laplacian}
\Delta_{\HdR} 
\ov{\mathrm{def}}{=} 
D^2 
= \d \d^* + \d^* \d.
\end{equation}
More precisely, if $1 < p < \infty$ we consider the closure $D_p$ of the operator $D$, which is a closed unbounded operator on the space $\L^p(\Omega^\bullet(M))$.

It should be possible to establish that the Fredholm index of an elliptic operator on a compact manifold is independent of $p \in (1, \infty)$, by generalizing standard arguments. 
Nevertheless, establishing the structural framework that explains this stability within noncommutative geometry is far from trivial. It requires checking that the unbounded operator generates a compact Banach spectral triple with analytic properties, specifically the bisectoriality and the boundedness of the $\H^\infty$ functional calculus of the Hodge--Dirac operator $D_p$. The latter functional calculus is a substitute for the functional calculus provided by the spectral theory of selfadjoint operators, which is available for free in the Hilbertian setting. This is no longer the case when $p \neq 2$. In the non-compact setting, the $\L^p$-theory of the Hodge--Dirac operator is subtle, for instance, van Neerven and Versendaal proved in \cite{NeV17} the boundedness of the $\H^\infty$ functional calculus under an assumption of positive curvature. We show here that in the compact setting, the situation is considerably more favorable. However, even in this case, the semigroup of operators generated by the closure $\Delta_{\HdR,p}$ of the Hodge-de Rham Laplacian $\Delta_{\HdR}$ is not necessarily contractive on the space $\L^p(\Omega^\bullet(M))$ as observed in \cite[Corollary 3.4 p.~354]{Str86}. However, it is worth noting that it is stated in \cite[Theorem 2 p.~224]{Kor91} that this operator generates a holomorphic semigroup of operators acting on $\L^p(\Omega^\bullet(M))$. Exploiting the compactness of the manifold $M$, we prove that the Hodge--Dirac operator $D_p$ is bisectorial and admits a bounded $\H^\infty$ functional calculus on the space $\L^p(\Omega^\bullet(M))$ \textit{without any curvature assumption}. 



Our strategy is inspired by \cite{NeV17}, where curvature assumptions are used to obtain the $\L^p$-boundedness of the Riesz transforms on $k$-forms and the bounded $\H^\infty$ functional calculus of the Hodge-de Rham Laplacian. In the compact setting, we bypass these curvature requirements. On the one hand, we establish the boundedness of the Riesz transforms by global pseudo-differential calculus. On the other hand, we derive Gaussian bounds for the heat kernel of the Hodge-de Rham Laplacian on $k$-forms from the corresponding estimates for the scalar Laplace-Beltrami operator $\Delta$ by a domination argument. Since the closure $\Delta_{\HdR,2}$ of the unbounded operator $\Delta_{\HdR}$ is a positive selfadjoint operator on the Hilbert space $\L^2(\Omega^\bullet(M))$, it is a sectorial operator admitting a bounded $\H^\infty$ functional calculus. Combining the case $p=2$ and the Gaussian bounds, an extrapolation argument of Duong-Robinson \cite{DuR96} (and its vector-valued extension \cite{Hal05}) gives the boundedness of the $\H^\infty$ functional calculus of the sectorial operator $\Delta_{\HdR,p}$ on the space $\L^p(\Omega^\bullet(M))$ for all $1<p<\infty$. Standard arguments as in \cite{NeV17} then yield bisectoriality of the Hodge--Dirac operator $D_p$ and the boundedness of its $\H^\infty$ functional calculus.

The main objective of this paper is to provide a rigorous construction of the Banach spectral triple $(\C(M), \L^p(\Omega^\bullet(M)), D_p)$ and to demonstrate its effectiveness by recovering classical topological invariants. More precisely, we use the Banach compact spectral triple associated with the Hodge--Dirac operator $D_p$ to define an index pairing between a Banach K-homology class and the K-theory of the algebra $\C(M)$ of continuous functions on the manifold $M$. In concrete terms, the grading by form degree yields a Fredholm operator
\[
D_{p,+} \co \W^{1,p}(\Omega^{\even}(M)) \to \L^p(\Omega^{\odd}(M)),
\]
and we prove that its index does not depend on $p$. Moreover, we identify in Corollary \ref{cor-Euler-operator} this index with the Euler characteristic, which is a classical topological invariant, i.e.,
\[
\Index D_{p,+} 
= \chi(M).
\]
We also obtain in Theorem \ref{thm-signature} for the signature operator 
$$
\cal{D}_{p,+} \co \W^{1,p}(\Omega^+(M)) \to \L^p(\Omega^-(M))
$$ 
a similar formula
$$
\Index \cal{D}_{p,+}
= \sign(M)
$$
on an oriented compact Riemannian manifold $M$ of dimension divisible by four.

Our approach combines techniques from spectral theory on Banach spaces with the formalism of Banach K-homology. This work demonstrates that the machinery of Banach spectral triples is not merely an abstract generalization but an effective tool capable of handling classical geometric operators with the same precision as the Hilbertian theory. If $d=\dim M$, we also show that these Banach spectral triples are $d^+$-summable, by a direct argument relying on Sobolev embeddings. 

There is a substantial literature on $\L^p$-Hodge theory and $\L^p$-analysis of operators on manifolds, where geometric assumptions are typically required to control heat kernels, Riesz transforms, and Hodge decompositions. See for example \cite{AMR08}, \cite{BDG23}, \cite{Cha07}, \cite{ChL25}, \cite{CTW23}, \cite{CoD99}, \cite{CCH06}, \cite{CoD03}, \cite{CoL94}, \cite{DeK23}, \cite{ISS99}, \cite{Li08}, \cite{Li09}, \cite{Li10}, \cite{Li11}, \cite{NeV17}, \cite{Sco95}, \cite{Str83} and \cite{Str86}. Our work fits within this line of research, although we crucially exploit compactness as the main global input. Finally, we refer to \cite{Arh26b} for a study of Dolbeault--Dirac operators on compact K\"ahler manifolds in the same spirit.

\paragraph{Structure of the paper.}
After recalling the basic analytic setup for functional calculus of sectorial and bisectorial operators in Section \ref{sec-preliminaries}, we give some background on differential geometry in Section \ref{sec-Back-geometry}. Section~\ref{sec-Hinfty-calculus} is the analytical heart of the paper. We establish bisectoriality and the boundedness of the $\H^\infty$ functional calculus for the Hodge--Dirac operator on compact manifolds, using Gaussian estimates and domination arguments. In Section~\ref{sec-Banach-spectral-triples}, we introduce the Banach spectral triple structure.  
We also show that this compact Banach spectral triple is $d^+$-summable. Section~\ref{Section-Index-Euler-operator} and Section~\ref{Section-Index-signature-operator} are devoted to the explicit computation of the $\L^p$-index of the Euler and signature operators. Finally, in Section \ref{Appendix}, we describe the structure of unbounded closed operators anticommuting with a bounded symmetry.

\section{Preliminaries on operator theory and functional calculus}
\label{sec-preliminaries}

\subsection{Sectorial operators}

For any angle $\theta \in (0,\pi)$, we introduce the open sector symmetric around the positive real half-axis with opening angle $2\theta$
\begin{equation}
\label{def-sigma-omega}
\Sigma_{\theta} 
\ov{\mathrm{def}}{=} \big\{ z \in \mathbb{C} \backslash \{ 0 \} : \: | \arg z | < \theta \big\}.
\end{equation}
See Figure \ref{figure-sector}. It is useful to put $\Sigma_0 \ov{\mathrm{def}}{=} (0,\infty)$.

\begin{figure}[ht]
\centering
\includegraphics[scale=0.4]{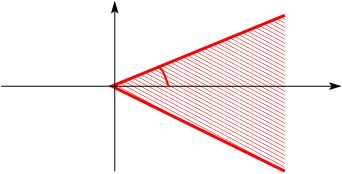}
\begin{picture}(0,0)
\put(-70,38){$\theta$}
\put(-45,20){$\Sigma_\theta$}
\end{picture}
\caption{open sector $\Sigma_\theta$ of angle $2\theta$}
\label{figure-sector}
\end{figure}

We refer to the books \cite{Haa06} and \cite{HvNVW18} for background on sectorial operators and their $\H^\infty$ functional calculus, introduced in the seminal papers \cite{CDMY96} and \cite{McI86}. Let $A \co \dom A \subset X \to X$ be a closed densely defined linear operator acting on a Banach space $X$. We say\footnote{\thefootnote. We caution the reader that this definition may differ across the literature.} that $A$ is a $\theta$-sectorial operator for some angle $\theta \in (0,\pi)$ if its spectrum $\sigma(A)$ is a subset of the closed sector $\ovl{\Sigma_\theta}$ and if the set $\big\{zR(z,A) : z \in \mathbb{C} \backslash \overline{\Sigma_\theta} \big\}$ is bounded in the algebra $\B(X)$ of bounded operators acting on $X$, where $R(z,A) \ov{\mathrm{def}}{=} (z\,\Id-A)^{-1}$ is the resolvent operator. 
The operator $A$ is said to be sectorial if it is a $\theta$-sectorial operator for some $\theta \in (0,\pi)$. In this case, we can introduce the angle of sectoriality
\begin{equation}
\label{equ-sectorial-operator}
\omega_{\sec}(A) 
\ov{\mathrm{def}}{=} \inf\{ \theta \in (0,\pi) : A \textrm{ is $\theta$-sectorial} \}.
\end{equation}
According to \cite[Example 10.1.2 p.~362]{HvNVW18}, if $-A$ is the generator of a strongly continuous bounded semigroup of bounded operators then the operator $A$ is sectorial with $\omega_{\sec}(A) \leq \frac{\pi}{2}$. Moreover, by \cite[Example 10.1.3 p.~362]{HvNVW18} and \cite[Proposition 3.4.4 p.~79]{Haa06}, the operator $A$ is sectorial with $\omega_{\sec}(A) < \frac{\pi}{2}$ if and only if $-A$ generates a bounded holomorphic (equivalently, bounded analytic) strongly continuous semigroup, that is, there exist an angle $\theta \in (0,\frac{\pi}{2})$ and a bounded holomorphic extension $\Sigma_\theta \to \B(X)$, $z \mapsto T_z$.

\begin{example} \normalfont
\label{Example-analytic-positive-selfadjoint}
According to \cite[Example 3.7.5 p.~150]{ABHN11}, if $A$ is a positive selfadjoint operator with dense domain acting on a Hilbert space, then $(\e^{-tA})_{t \geq 0}$ is a bounded holomorphic strongly continuous semigroup.
\end{example}

For any angle $\theta \in (0,\pi)$, we consider the algebra $\H^{\infty}(\Sigma_\theta)$\label{algebra-Hinfty} of all bounded analytic functions $f \co \Sigma_\theta \to \mathbb{C}$, equipped with the supremum norm 
\begin{equation}
\label{norm-Hinfty}
\norm{f}_{\H^{\infty}(\Sigma_\theta)}
\ov{\mathrm{def}}{=} \sup\bigl\{\vert f(z)\vert \, :\, z\in \Sigma_\theta\bigr\}.
\end{equation}
Let $\H^{\infty}_{0}(\Sigma_\theta)$\label{algebra-Hinfty0} be the subalgebra of bounded analytic functions $f \co \Sigma_\theta \to \mathbb{C}$ for which there exist $s,C>0$ such that 
\begin{equation*}
\vert f(z)\vert
\leq C\min\{|z|^s,|z|^{-s}\}, \quad z \in \Sigma_\theta,
\end{equation*} 
as discussed in \cite[Section 2.2]{Haa06}.

Let $A$ be a sectorial operator acting on a Banach space $X$. Consider some angle $\theta \in (\omega_{\sec}(A), \pi)$ and some function $f \in \H^\infty_0(\Sigma_\theta)$. Following \cite[p.~30]{Haa06}, \cite[p.~5]{LM99} (see also \cite[p.~369]{HvNVW18}), for any angle $\nu \in (\omega_{\sec}(A),\theta)$ 
 we introduce the operator
\begin{equation}
\label{2CauchySec}
f(A)
\ov{\mathrm{def}}{=} \frac{1}{2\pi \i}\int_{\partial\Sigma_\nu} f(z) R(z,A) \d z,
\end{equation}
acting on $X$, with a Cauchy integral, where the boundary $\partial\Sigma_\nu$ is parametrized by
\begin{equation*}
\partial\Sigma_\nu(t)
\ov{\mathrm{def}}{=} \begin{cases}
-t \e^{\i\nu} &\text{if } t\in (-\infty,0]\\
t \e^{-\i\nu} &\text{if } t\in [0,\infty)\\
\end{cases}.
\end{equation*}
The sectoriality condition ensures that this integral is absolutely convergent and defines a bounded operator on the Banach space $X$. Using Cauchy's theorem, it is possible to show that this definition does not depend on the choice of the angle $\nu$. The resulting map $\H^\infty_0(\Sigma_\theta) \to \B(X)$, $f \mapsto f(A)$ is an algebra homomorphism.

Following \cite[Definition 2.6 p.~6]{LM99} (see also \cite[p.~114]{Haa06}), we say that the operator $A$ admits a bounded $\H^\infty(\Sigma_\theta)$ functional calculus if the latter homomorphism is bounded, i.e., if there exists a constant $C \geq 0$ such that 
\begin{equation}
\label{Def-functional-calculus}
\norm{f(A)}_{X \to X} 
\leq C\norm{f}_{\H^\infty(\Sigma_\theta)},\quad f \in \H^\infty_0(\Sigma_\theta).
\end{equation}
In this context, we can introduce the $\H^\infty$-angle
\begin{equation*}
\omega_{\H^\infty}(A) 
\ov{\mathrm{def}}{=} \inf\{\theta \in (\omega_{\sec}(A),\pi) : A \text{ admits a bounded $\H^\infty(\Sigma_\theta)$ functional calculus} \}.
\end{equation*}
Suppose that the Banach space $X$ is reflexive. If the operator $A$ is injective and admits a bounded $\H^\infty(\Sigma_\theta)$ functional calculus, then the previous homomorphism naturally extends to a bounded homomorphism $f \mapsto f(A)$ from the algebra $\H^\infty(\Sigma_\theta)$ into the algebra $\B(X)$ of bounded linear operators on $X$. 

\begin{example} \normalfont
\label{positive-selfadjoint}
By \cite[Proposition 10.2.23 p.~388]{HvNVW18}, a positive selfadjoint operator $A$ with dense domain on a complex Hilbert space $H$ is sectorial and admits a bounded $\H^\infty(\Sigma_\theta)$ functional calculus for any angle $\theta >0$, i.e., $\omega_{\H^\infty}(A)=0$.
\end{example}

\subsection{Bisectorial operators}
We refer to \cite{Ege15} and to the books \cite{HvNVW18} and \cite{HvNVW23} for more information on bisectorial operators, which can be seen as generalizations of unbounded selfadjoint operators in the framework of Banach spaces. For any angle $\theta \in (0,\frac{\pi}{2})$, we consider the open bisector $\Sigma_\theta^\bi \ov{\mathrm{def}}{=} (-\Sigma_\theta) \cup \Sigma_\theta$ where the sector $\Sigma_{\theta}$ is defined in \eqref{def-sigma-omega} and $\Sigma_0^\bi \ov{\mathrm{def}}{=} (-\infty,0) \cup (0,\infty)$. Following \cite[Definition 10.6.1 p.~447]{HvNVW18},
 we say that a closed densely defined operator $D$ on a Banach space $X$ is $\theta$-bisectorial for some angle $\theta \in (0,\frac{\pi}{2})$ if its spectrum $\sigma(D)$ is a subset of the closed bisector $\ovl{\Sigma^\bi_{\theta}}$ and if the subset $\big\{z R(z,D) : z \not\in \ovl{\Sigma_{\theta}^\bi} \big\}$ is bounded in the space $\B(X)$ of bounded operators acting on $X$, where $R(z,D) \ov{\mathrm{def}}{=} (z\, \Id -D)^{-1}$ denotes the resolvent operator. See Figure \ref{fig:bisectorial-spectrum}. The infimum of all $\theta \in (0,\frac{\pi}{2})$ such that $D$ is $\theta$-bisectorial is called the angle of bisectoriality of $D$ and denoted by $\omega_{\bi}(D)$.

\begin{figure}[ht]
\centering
\begin{tikzpicture}[scale=0.55]
\clip (-4.5,-2) rectangle (4.5,2);

\fill[pattern=north east lines, pattern color=blue!40, opacity=0.4]
  (0,0) -- (5,0) -- (5,5) -- cycle;
\fill[pattern=north east lines, pattern color=blue!40, opacity=0.4]
  (0,0) -- (5,0) -- (5,-5) -- cycle;

\fill[pattern=north east lines, pattern color=blue!40, opacity=0.4]
  (0,0) -- (-5,0) -- (-5,5) -- cycle;
\fill[pattern=north east lines, pattern color=blue!40, opacity=0.4]
  (0,0) -- (-5,0) -- (-5,-5) -- cycle;

\filldraw[fill=blue!20, draw=black, domain=0:2*pi, samples=200]
  plot ({4*sqrt(2)*cos(\x r)/(1+sin(\x r)^2)},
        {1.5*sqrt(2)*cos(\x r)*sin(\x r)/(1+sin(\x r)^2)}) -- cycle;

\draw (1,0) arc (0:45:1);
\node at (1,0.8) [right] {$\theta$};

\draw (-4.5,0) -- (4.5,0);
\draw (0,-2) -- (0,2);
\draw (-2,2) -- (2,-2);
\draw (-2,-2) -- (2,2);

\node[right] at (-4.5,0.25) {$\sigma(D)$};
\node[right] at (2.5,1.3) {$\ovl{\Sigma^\bi_{\theta}}$};
\end{tikzpicture}
\caption{spectrum $\sigma(D)$ of a bisectorial operator $D$}
\label{fig:bisectorial-spectrum}
\end{figure}

By \cite[Example 3.4.15 p.~163]{Ege15}, any selfadjoint operator $D$ on a complex Hilbert space $H$ is bisectorial with $\omega_{\bi}(D)=0$. It is worth noting that if $D$ is a $\theta$-bisectorial operator on a Banach space $X$ then by \cite[Proposition 10.6.2 (2) p.~448]{HvNVW18} its square $D^2$ is $2\theta$-sectorial and we have
\begin{equation}
\label{Bisec-Ran-Ker}
\ovl{\Ran D^2}
=\ovl{\Ran D}
\quad \text{and} \quad
\ker D^2
=\ker D.
\end{equation}
If the Banach space $X$ is reflexive, we have by \cite[Proposition 3.2.2 (iv)]{Ege15} a topological decomposition
\begin{equation}
\label{decomposition-de-X}
X
=\ker D \oplus \ovl{\Ran D}.
\end{equation}

\paragraph{Functional calculus} Consider a bisectorial operator $D$ on a Banach space $X$ of angle $\omega_\bi(D)$. For any angle $\theta \in (\omega_\bi(D), \frac{\pi}{2})$ and any function $f$ belonging to the space 
$$
\H^{\infty}_0(\Sigma_\theta^\bi) 
\ov{\mathrm{def}}{=}  \left\{ f \in \H^\infty(\Sigma_\theta^\bi) :\: \exists C,s > 0 \: \forall \: z \in \Sigma_\theta^\bi : \: |f(z)| \leq C \min\{|z|^s, |z|^{-s} \}  \right\},
$$
we can define a bounded operator $f(D)$ acting on the space $X$ by integrating on the boundary $\partial \Sigma^\bi_{\nu}$ of the bisector $\Sigma^\bi_{\nu}$ for some angle $\nu \in (\omega_\bi(D),\theta)$ using a Cauchy integral
\begin{equation}
\label{def-f(D)}
f(D)
\ov{\mathrm{def}}{=} \frac{1}{2\pi \i}\int_{\partial \Sigma^\bi_{\nu}} f(z)R(z,D) \d z.
\end{equation}
The integration contour is oriented counterclockwise, so that the interiors of the two sectors $\Sigma_\nu$ and $-\Sigma_\nu$ are always to its left. The integral in \eqref{def-f(D)} converges absolutely thanks to the decay of the function $f$ and is independent of the particular choice of the angle $\nu$ by Cauchy's integral theorem. We refer to \cite[Section 3.2.1]{Ege15} and \cite[Theorem 10.7.10 p.~449]{HvNVW18} for more details. 

The operator $D$ is said to admit a bounded $\H^\infty(\Sigma_\theta^\bi)$ functional calculus if there exists a constant $C \geq 0$ such that 
\begin{equation}
\label{funct-cal-bisector}
\bnorm{f(D)}_{X \to X} 
\leq C \norm{f}_{\H^\infty(\Sigma_\theta^\bi)}, \quad f \in \H^{\infty}_0(\Sigma_\theta^\bi).
\end{equation}
Suppose that the Banach space $X$ is reflexive. If the operator $D$ is injective and admits a bounded $\H^\infty(\Sigma_\theta^\bi)$ functional calculus, then the previous homomorphism naturally extends to a bounded homomorphism $f \mapsto f(D)$ from the algebra $\H^\infty(\Sigma_\theta^\bi)$ into the algebra $\B(X)$ of bounded linear operators on $X$.

\begin{example} \normalfont
\label{ex-signe}
Suppose that $X$ is reflexive. Using the function $\sgn\ov{\mathrm{def}}{=} 1_{\Sigma_\theta}-1_{-\Sigma_\theta}$ and the injective part $D|_{\ovl{\Ran D}}$, we can define, following \cite[p.~498]{HvNVW23}, the bounded operator $\sgn D \co \ovl{\Ran D} \to \ovl{\Ran D}$. Using the decomposition \eqref{decomposition-de-X}, we extend it on $X$ by putting $\sgn D|_{\ker D}=0$. This operator plays a role analogous to that of the Hilbert transform in classical harmonic analysis. This operator will play an important role below, since it is used to produce a Banach Fredholm module and hence a $\K$-homology class, see Proposition \ref{prop-triple-to-Fredholm}.
\end{example}

The following result is \cite[Lemma 2.1]{Arh26b}, which will be used in the proof of Proposition \ref{prop:closure-Hodge-square-dirac}.

\begin{lemma}
\label{lemma-extension-surjective-injective}
Let $A$ and $B$ be unbounded operators on a Banach space $X$ such that $A \subset B$. Suppose that there exists $\lambda \in \mathbb{C}$ such that $\lambda\Id_X+A \co \dom A \to X$ is surjective and $\lambda\Id_X+B \co \dom B \to X$ is injective. Then $A=B$.
\end{lemma}

\section{Background on differential geometry}
\label{sec-Back-geometry}

\paragraph{Differential forms on Riemannian manifolds} 
Let $M$ be a smooth compact manifold of dimension $d$ endowed with a Riemannian metric $g$. For each $x \in M$ we have a canonical inner product on the cotangent space $\mathrm{T}_x^*M$. For any $k \in \{1,\ldots,d\}$, we define an inner product on 
$\Lambda^k \mathrm{T}_x^*M$ by
\[
\la u_1 \wedge \cdots \wedge u_k, v_1 \wedge \cdots \wedge v_k 
\ra_{\Lambda^k \mathrm{T}_x^*M}
\ov{\mathrm{def}}{=}
\det \big[\la u_i, v_j \ra_{\mathrm{T}_x^*M}\big]_{1 \leq i,j \leq k},
\]
where $u_1,\ldots,u_k,v_1,\ldots,v_k \in \mathrm{T}_x^*M$, and we extend this definition by bilinearity to arbitrary elements of $\Lambda^k \mathrm{T}_x^*M$. 
 We extend this inner product to the complexification $
\Lambda^k_{\mathbb{C}}\mathrm{T}_x^*M \ov{\mathrm{def}}{=} \Lambda^k\mathrm{T}_x^*M \ot_{\R} \mathbb{C}$ as a Hermitian inner product, linear in the first variable and conjugate-linear in the second variable. If $z,w \in \mathbb C$, we set $
\la z,w\ra_{\Lambda_{\mathbb C}^0\mathrm{T}_x^*M}
\ov{\mathrm{def}}{=}
z\ovl{w}$. 
We denote by $|\cdot|_{\Lambda^k_{\mathbb{C}}\mathrm{T}_x^*M}$ the associated norm. Finally, these Hermitian inner products induce a Hermitian inner product on the complex exterior algebra
$
\Lambda_{\mathbb{C}}\mathrm{T}_x^*M
\ov{\mathrm{def}}{=}
\bigoplus_{k=0}^d\Lambda_{\mathbb{C}}^k\mathrm{T}_x^*M,
$ 
and we denote by $|\cdot|_{\Lambda_{\mathbb{C}}\mathrm{T}_x^*M}$ the associated norm. We also set $
\Lambda_{\mathbb C}^k\mathrm{T}^*M
\ov{\mathrm{def}}{=}
\Lambda^k\mathrm{T}^*M\ot_{\R}\mathbb C$ and $
\Lambda_{\mathbb C}\mathrm{T}^*M
\ov{\mathrm{def}}{=}
\bigoplus_{k=0}^d\Lambda_{\mathbb C}^k\mathrm{T}^*M$.

For each integer $k \in \{0,\ldots,d\}$, we denote by $\Omega^k(M) \ov{\mathrm{def}}{=} \Gamma^\infty(\Lambda^k_{\mathbb{C}}\mathrm{T}^*M)$ the space of smooth \textit{complex} $k$-forms and by $
\Omega^\bullet(M)
\ov{\mathrm{def}}{=}
\oplus_{k=0}^d \Omega^k(M)$ 
the space of smooth differential forms of all degrees. Recall that by \cite[pp.~168-169]{RuS13} the exterior differential $\d \co \Omega^\bullet(M) \to \Omega^\bullet(M)$ is a graded derivation of degree 1, i.e., we have the Leibniz rule
\begin{equation}
\label{Leibniz-manifold-d}
\d(\eta \wedge \omega)
=\d \eta \wedge \omega + (-1)^{\deg \eta}\eta \wedge\d\omega, \quad \eta,\omega \in \Omega^\bullet(M).
\end{equation}
For any smooth vector field $X$ on the manifold $M$, following \cite[Definition 6.4.7 p.~429]{AMR88}, if $k \geq 0$ we denote by $i_X \co \Omega^{k+1}(M) \to \Omega^{k}(M)$ the interior product with $X$ (or insertion operator), defined by $(i_X\omega)(X_1,\ldots,X_k) \ov{\mathrm{def}}{=} \omega(X,X_1,\ldots,X_k)$ and $i_X\omega \ov{\mathrm{def}}{=} 0$ if $\omega \in \Omega^0(M)$. If $\omega,\eta \in \Omega^k(M)$, we can consider the function $\la \omega,\eta \ra_{\Omega^{k}(M)} \co M \to \mathbb{C}$ defined by $\la \omega,\eta\ra_{\Omega^{k}(M)}(x) \ov{\mathrm{def}}{=} \la \omega(x),\eta(x)\ra_{\Lambda^k_{\mathbb{C}} \mathrm{T}_x^*M}$.
 
If $\mathrm{T}_{\mathbb C}M
\ov{\mathrm{def}}{=}
\mathrm{T}M\ot_{\R}\mathbb C$ and $\mathrm{T}_{\mathbb{C}}^*M \ov{\mathrm{def}}{=}\mathrm{T}^*M \ot_{\R} \mathbb{C}$, we extend the musical isomorphisms complex linearly to $\sharp\co \mathrm{T}_{\mathbb{C}}^*M \to \mathrm{T}_{\mathbb{C}}M$, and the interior product complex linearly to complex vector fields. We denote by $\alpha\mapsto\ovl{\alpha}$ the canonical conjugation on the complexification $\mathrm{T}_{\mathbb{C}}^*M$ and on its exterior powers.

We will also use the adjointness relation between exterior multiplication and interior multiplication. If $\eta \in \Omega^{k-1}(M)$, $\omega \in \Omega^k(M)$ and $\alpha \in \Omega^1(M)$, then
\begin{equation}
\label{eq:wedge-interior-adjoint}
\la \omega,\alpha\wedge\eta \ra_{\Omega^k(M)}
=
\la i_{(\ovl{\alpha})^\sharp} \omega,\eta \ra_{\Omega^{k-1}(M)}.
\end{equation}
In particular, e.g., \cite[p.~327]{PRG14}, if $\alpha$ is real-valued then we have
\[
\la \omega,\alpha\wedge\eta \ra_{\Omega^k(M)}
=
\la i_{\alpha^\sharp}\omega,\eta \ra_{\Omega^{k-1}(M)}.
\]
The Riemannian gradient of a smooth function $f \co M \to \mathbb{C}$ is the unique smooth section $\nabla f$ of the complexified tangent bundle $\mathrm{T}_{\mathbb{C}}M$ such that $g_{\mathbb{C}}(\nabla f,X)
=X(f)$ 
for any smooth complex vector field $X$ on $M$, where $g_{\mathbb{C}}$ denotes the complex-bilinear extension of $g$. See \cite[p.~343]{Lee13} for the particular case of real-valued functions. In other words, we have
\begin{equation}
\label{sharp}
(\d f)^\sharp
=\nabla f.
\end{equation}
We denote by $\dist(\cdot,\cdot)$ the Riemannian distance on $M$.

\paragraph{$\L^p$-spaces on vector bundles}
Let $E$ be a smooth Hermitian vector bundle of finite rank over a smooth compact manifold $M$ equipped with a Lebesgue measure $\mu$. Suppose that $1 \leq p \leq \infty$. For a measurable section $f \co M \to E$, following \cite[p.~481]{Nic21} and \cite[Chapter I]{Gun17} we define
\begin{equation}
\label{Lp-norm-vector-bundle}
\norm{f}_{\L^p(M,E)} 
\ov{\mathrm{def}}{=} 
\begin{cases}
\left( \int_M \norm{f(x)}_{E_x}^p \d\mu(x) \right)^{\frac{1}{p}}, & 1 \leq p < \infty,\\
\esssup_{x \in M} \norm{f(x)}_{E_x}, & p=\infty,
\end{cases}
\end{equation}
and we denote by $\L^p(M,E)$ the corresponding Banach space of measurable sections modulo equality $\mu$-almost everywhere.

If $M$ is a compact smooth Riemannian manifold, we denote by $\mu_g$ the Riemannian measure on $M$ defined in \cite[Theorem 3.11 p.~59]{Gri09}. We set $\L^p(\Omega^\bullet(M)) \ov{\mathrm{def}}{=} \L^p(M,\Lambda_{\mathbb{C}} \mathrm{T}^*M )$, where the norm on each fiber is given by the pointwise $\ell^2$-sum of degrees. Thus, if $\omega=\sum_{k=0}^d \omega_k$ with $\omega_k \in \Omega^k(M)$, then
\begin{equation}
\label{norm-Lp-Df}
\norm{\omega}_{\L^p(\Omega^\bullet(M))}
=
\bigg(\int_M \bigg(\sum_{k=0}^d |\omega_k(x)|_{\Lambda^k_{\mathbb{C}} \mathrm{T}_x^*M}^2\bigg)^{\frac{p}{2}} \d\mu_g(x)\bigg)^{\frac{1}{p}},
\quad 1 \leq p < \infty.
\end{equation}
If $p=2$ this norm is induced by an inner product $\la \cdot, \cdot \ra_{\L^2(\Omega^\bullet(M))}$. We will use the formal adjoint (or codifferential) $\d^* \co \Omega^\bullet(M) \to \Omega^\bullet(M)$ of $\d$ defined by the formula
\begin{equation}
\label{def-codiff}
\la \d \alpha, \beta \ra_{\L^2(\Omega^\bullet(M))}
=\la \alpha, \d^*\beta \ra_{\L^2(\Omega^\bullet(M))}, \quad \alpha,\beta \in \Omega^\bullet(M).
\end{equation}

\paragraph{Hodge star operator on oriented Riemannian manifolds} Now, we suppose that $(M,g)$ is a (connected) smooth compact \textit{oriented} Riemannian manifold of dimension $d$. We denote by $\vol_g \in \Omega^d(M)$ the Riemannian volume form, see \cite[Proposition 15.29 p.~389]{Lee13}. According to \cite[(20.8.6.3)]{Dieu74}, for any integrable scalar function $f$, we have $\int_M f \d \mu_g=\int_M f\, \vol_g$.
Following \cite[p.~412]{Lee09} and \cite[p.~457]{AMR88}, we denote by $* \co \Lambda^k\mathrm T^*M \to \Lambda^{d-k}\mathrm{T}^*M$ the Hodge star operator. 
We denote by the same symbol its complex-linear extension $* \co \Lambda_{\mathbb{C}}^k\mathrm{T}^*M \to \Lambda_{\mathbb{C}}^{d-k}\mathrm{T}^*M$. It induces a complex-linear map $* \co \Omega^k(M) \to \Omega^{d-k}(M)$. With our convention for the Hermitian inner product, $*\ovl{\beta}$ is the unique differential form in $\Omega^{d-k}(M)$ satisfying
\begin{equation}
\label{eq:def-Hodge-star}
\alpha\wedge *\ovl{\beta}
=
\la\alpha,\beta\ra_{\Omega^k(M)}\vol_g,
\quad
\alpha,\beta\in\Omega^k(M).
\end{equation}
By \cite[Exercise 4.3.8 p.~211]{Nab11}, 
we have
\begin{equation}
\label{*-et-Mf}
*(f\omega)
=f*\omega, \quad f \in \C^\infty(M), \omega \in \Omega^k(M).
\end{equation}
Moreover, on the space $\Omega^k(M)$ of $k$-forms the Hodge star operator satisfies by \cite[p.~299]{Dieu74}
\begin{equation}
\label{eq:star-square}
*^2
=(-1)^{k(d-k)} \Id_{\Omega^k(M)}.
\end{equation}
Recall that
\begin{equation}
\label{Hodge-norm}
|\omega(x)|_{\Lambda_{\mathbb C}^k\mathrm{T}_x^*M}
=
|*\omega(x)|_{\Lambda_{\mathbb C}^{d-k}\mathrm{T}_x^*M},
\quad
x\in M,\quad \omega\in\Omega^k(M).
\end{equation}
Finally, recall that by \cite[p.~135]{Jos17} and \cite[Definition 6.5.21 p.~457]{AMR88} the codifferential $\d^* \co \Omega^{k}(M) \to \Omega^{k-1}(M)$ satisfies 
\begin{equation}
\label{d-star-Omega-k}
\d^*\alpha 
\ov{\mathrm{def}}{=} (-1)^{d(k+1)+1} *(\d* \alpha), \quad \alpha \in \Omega^{k}(M),
\end{equation}
where $k \geq 1$, and $\d^*(\Omega^0(M)) \ov{\mathrm{def}}{=} 0$.

\section{$\mathrm{H}^\infty$ functional calculus of the Hodge--Dirac operator on a compact manifold}
\label{sec-Hinfty-calculus}

The Hodge--Dirac operator $D$ defined in \eqref{def-Hodge-Dirac} is closable on the space $\L^p(\Omega^\bullet(M))$ by \cite[Lemma 4.2 p.~3129]{NeV17} or alternatively by combining \cite[Theorem 5.28 p.~168]{Kat76} with the symmetry \cite[Theorem 13.6 p.~318]{BlB13} of $D$. We denote by $D_p \co \dom D_p \subset \L^p(\Omega^\bullet(M)) \to \L^p(\Omega^\bullet(M))$ its closure. In this section, we prove that $D_p$ is bisectorial and admits a bounded $\mathrm{H}^\infty$ functional calculus. Our strategy is inspired by that of \cite{NeV17}. This latter approach uses curvature assumptions for proving the boundedness of Riesz transforms on $\L^p$-spaces of $k$-forms and the boundedness of the $\mathrm{H}^\infty$ functional calculus of the Hodge-de Rham Laplacian. On the one hand, we use instead global pseudo-differential calculus for establishing the boundedness of Riesz transforms. On the other hand, we prove Gaussian estimates on the Hodge-de Rham Laplacian on $k$-forms using the ones of the Laplace-Beltrami operator $\Delta$ and a domination argument. Since the Hodge-de Rham Laplacian $\Delta_{\HdR,2}$ is a positive selfadjoint operator on the Hilbert space $\L^2(\Omega^\bullet(M))$, it admits by Example \ref{positive-selfadjoint} a bounded $\H^\infty$ functional calculus at the level $p=2$. Gaussian estimates allow us to use the extrapolation result of Duong and Robinson \cite{DuR96} and Haller-Dintelmann \cite{Hal05} for obtaining the boundedness of the $\H^\infty$ functional calculus on the space $\L^p(\Omega^\bullet(M))$ for all $1 < p < \infty$. A standard adaptation of the argument in \cite{NeV17} then yields  the bisectoriality and the boundedness of the functional calculus of the Hodge--Dirac operator $D_p$.

\subsection{Heat kernels estimates}

The following domination result for the Hodge-de Rham Laplacian is \cite[Theorem 2.2 p.~495]{Li10}. This reference does not provide a proof and the reference cited in \cite{Li10} before this result does not contain this result. However, it can be found in \cite{Ros88} and \cite{DoL82}. 
Here, $\W_k$ is the Weitzenb\"ock curvature on $k$-forms.

\begin{prop}
\label{prop-domination}
Let $M$ be a smooth compact Riemannian manifold of dimension $d$. Let $k \in \{0,\ldots,d\}$. Suppose that $\W_k \geq -a$ for some $a \in \R$. For any differential form $\omega \in \Omega^k(M)$, we have
\begin{equation}
\label{domination}
|\e^{-t\Delta_{\HdR}}\omega|
\leq \e^{a t} \e^{t\Delta} |\omega|, \quad t > 0,
\end{equation}
where $\Delta\ov{\mathrm{def}}{=} \div \grad$ is the Laplace-Beltrami operator, acting as a negative operator on the space of functions.
\end{prop}

For any integer $k \in \{0,\dots,d\}$, let $\Delta_{\HdR,k}$ be the Hodge-de Rham Laplacian on $k$-forms. 
We need the following result on kernel domination.

\begin{lemma}
\label{lemma-kernel-domination}
Let $M$ be a smooth compact Riemannian manifold of dimension $d$. Let $k \in \{0,\ldots,d\}$. Assume that for some $a \in \R$ we have for any $\omega \in \Omega^k(M)$ the estimate
\begin{equation}
\label{domin-bis}
\bigl|\e^{-t\Delta_{\HdR,k}}\omega\bigr|
\leq \e^{at} \e^{t\Delta}(|\omega|),\quad t > 0.
\end{equation}
Let $p_t^{k}(x,y)$ and $p_t(x,y)$ be the smooth integral kernels of the operators $\e^{-t\Delta_{\HdR,k}}$ and $\e^{t\Delta}$. Then
\[
\norm{p_t^{k}(x,y)}
\leq \e^{at} p_t(x,y),\quad t>0, x,y \in M.
\]
\end{lemma}

\begin{proof}
Since $\e^{-t\Delta_{\HdR,k}}$ has kernel $p_t^{k}$, for every $\omega \in \Omega^k(M)$ we have
\begin{align*}
\MoveEqLeft
\bigl|(\e^{-t\Delta_{\HdR,k}}\omega)(x)\bigr|         
=\left|\int_M p_t^{k}(x,y)\omega(y)\d\mu_g(y) \right| 
\leq \int_M \norm{p_t^{k}(x,y)} |\omega(y)| \d\mu_g(y).
\end{align*}
Similarly, we have
\[
(\e^{t\Delta}|\omega|)(x)
=\int_M p_t(x,y) |\omega(y)| \d\mu_g(y),
\]
with $p_t(x,y) \geq 0$. Fix $y \in M$, $t > 0$ and a unit vector $v \in \Lambda_{\mathbb{C}}^k \mathrm{T}^*_y M$. With \cite[(16.12.11)]{Dieu72}, we can choose a smooth differential form $\sigma \in \Omega^k(M)$ over $M$ with $\sigma(y)=v$. Let $(\rho_n)$ be a sequence of non-negative smooth functions supported in geodesic balls $B(y,\frac{1}{n})$ such that $\int_M \rho_n(z) \d \mu_g(z) = 1$ (an approximate identity). Define the sequence of test forms $\omega_n \ov{\mathrm{def}}{=} \rho_n\sigma$. Applying the semigroup domination assumption \eqref{domin-bis} to $\omega_n$, we obtain for any integer $n$
\begin{align*}
\MoveEqLeft
\left| \int_M p_t^{k}(x,z) \omega_n(z) \d\mu_g(z) \right| 
=\bigl|(\e^{-t\Delta_{\HdR,k}}\omega_n)(x)\bigr| \\
&\ov{\eqref{domin-bis}}{\leq} \e^{at} \e^{t\Delta}(|\omega_n|)         
=\e^{at} \int_M p_t(x,z) |\omega_n(z)| \d\mu_g(z).
\end{align*}
Letting $n \to \infty$, the properties of the approximate identity and the continuity of the kernels yield:
\[
|p_t^{k}(x,y) v| 
\leq \e^{at} p_t(x,y) |v| 
= \e^{at} p_t(x,y).
\]
Taking the supremum over unit vectors $v$ gives $\norm{p_t^{k}(x,y)} \leq \e^{at}p_t(x,y)$ almost everywhere. By smoothness of the kernels, the inequality holds pointwise for all $x,y \in M$.
\end{proof}

Recall that if $E$ is a smooth Hermitian vector bundle of finite rank over a compact Riemannian manifold $M$, we can consider the tensor bundle $E \boxtimes E^*$ over $M \times M$ with fiber
\[
(E \boxtimes E^*)_{(x,y)} 
= E_x \ot E_y^* 
= \Hom(E_y,E_x), \quad x,y \in M.
\]
We refer to \cite[p.~147]{BaC17}, \cite[(23.4.4)]{Dieu88} and \cite[p.~74]{BGV04} for more information. If $K \co M \times M \to E \boxtimes E^*$ is a continuous section then the function $M \times M  \to [0,\infty)$, $(x,y) \mapsto \norm{K(x,y)}_{\Hom(E_y,E_x)}$ is continuous. Moreover, for any $f \in \L^1(M,E)$ and any $x \in M$, we can define
\begin{equation}
\label{integral-operator}
(T_K f)(x) 
\ov{\mathrm{def}}{=} \int_M K(x,y) f(y) \d\mu_g(y),
\end{equation}
where the integral is a Bochner integral in the finite-dimensional space $E_x$. 
We need the following variant of the Dunford-Pettis theorem. 

\begin{prop}
\label{prop:kernel-norm-equality-bundles}
Let $(M,g)$ be a smooth compact Riemannian manifold. Consider a smooth Hermitian vector bundle $E$ of finite rank over $M$ and a continuous section $K \co M \times M \to E \boxtimes E^*$. Then $T_K$ extends uniquely to a bounded operator $T_K \co \L^1(M,E) \to \L^\infty(M,E)$, and its operator norm is given by
\[
\norm{T_K}_{\L^1(M,E) \to \L^\infty(M,E)}
=
\sup_{(x,y) \in M \times M} \norm{K(x,y)}_{\Hom(E_y,E_x)}.
\]
\end{prop}

\begin{proof}
Let $
C \ov{\mathrm{def}}{=} \sup_{(x,y)\in M\times M} \norm{K(x,y)}_{\Hom(E_y,E_x)}$. Consider a continuous section $f \co M \to E$. For any $x \in M$, we have
\begin{align*}
\MoveEqLeft
\norm{(T_K f)(x)}_{E_x}
\ov{\eqref{integral-operator}}{=}
\norm{\int_M K(x,y) f(y) \d\mu_g(y)}_{E_x} 
\leq
\int_M \norm{K(x,y) f(y)}_{E_x} \d\mu_g(y) \\
&\leq
\int_M \norm{K(x,y)}_{\Hom(E_y,E_x)} \norm{f(y)}_{E_y} \d\mu_g(y) 
\leq C \int_M \norm{f(y)}_{E_y} \d\mu_g(y)
\ov{\eqref{Lp-norm-vector-bundle}}{=} C \norm{f}_{\L^1(M,E)}.
\end{align*}
Taking the essential supremum over $x\in M$ yields $
\norm{T_K f}_{\L^\infty(M,E)} \leq C \norm{f}_{\L^1(M,E)}$. By density of the subspace $\Gamma^\infty(M,E)$ of smooth sections in the space $\L^1(M,E)$ and completeness of $\L^\infty(M,E)$, the operator extends uniquely to a bounded map $T_K \co \L^1(M,E) \to \L^\infty(M,E)$ with $\norm{T_K}_{\L^1(M,E) \to \L^\infty(M,E)} \leq C$. 

Since $M \times M$ is compact and since the function $(x,y)\mapsto \norm{K(x,y)}$ is continuous, there exist $(x_0,y_0) \in M \times M$ such that $
\norm{K(x_0,y_0)}_{\Hom(E_{y_0},E_{x_0})} = C$. Choose $u_0\in E_{y_0}$ with $\norm{u_0}_{E_{y_0}}=1$ such that
\begin{equation}
\label{inter-38}
\norm{K(x_0,y_0)u_0}_{E_{x_0}} 
= \norm{K(x_0,y_0)}_{\Hom(E_{y_0},E_{x_0})} 
= C.
\end{equation}
Fix $\epsi > 0$. There exist an open neighborhood $U$ of $y_0$ and a smooth trivialization $\Phi \co E|_U \to U \times \mathbb{C}^r$ which is fiberwise unitary. In particular, writing $\Phi_y \co E_y \to \mathbb{C}^r$ for the restriction to the fiber, we have
\[
\norm{\Phi_y(\xi)}_{\mathbb{C}^r}=\norm{\xi}_{E_y},
\quad y\in U, \xi \in E_y.
\]
Let $e_1\in\mathbb{C}^r$ be the first vector of the canonical basis and assume, after composing $\Phi$ with a fixed unitary matrix if necessary, that $
\Phi_{y_0}(u_0)=e_1$. Define a smooth local section $\tilde u \co U \to E|_U$, $y \mapsto \Phi_y^{-1}(e_1)$. Then $\tilde u(y_0)=u_0$ and $\norm{\tilde u(y)}_{E_y}=1$ for all $y\in U$. Set $u\ov{\mathrm{def}}{=}\tilde u$, so $u$ is a smooth section on $U$ with $\norm{u(y)}_{E_y}=1$ for all $y\in U$. Since $K \co M\times M\to E\boxtimes E^*$ is continuous, the map $
y\mapsto K(x_0,y)u(y)\in E_{x_0}$ 
is continuous on $U$. Hence, shrinking $U$ if necessary, we may assume that
\begin{equation}
\label{inter-37}
\norm{K(x_0,y)u(y)-K(x_0,y_0)u_0}_{E_{x_0}} 
\leq \epsi,
\quad y \in U.
\end{equation}
Define a section $f \in \L^1(M,E)$ by
\begin{equation}
\label{inter-36}
f(y) \ov{\mathrm{def}}{=}
\begin{cases}
\mu(U)^{-1} u(y), & y\in U, \\
0, & y\notin U.
\end{cases}
\end{equation}
Then $\norm{f}_{\L^1(M,E)} \ov{\eqref{Lp-norm-vector-bundle}}{=} 1$. Note that $T_K(f)$ is continuous since $K$ is continuous.   
Moreover, we have
\begin{equation}
\label{inter-34}
(T_K f)(x_0)
\ov{\eqref{integral-operator}}{=}
\int_M K(x_0,y) f(y) \d\mu_g(y)
\ov{\eqref{inter-36}}{=}
\frac{1}{\mu(U)} \int_U K(x_0,y)u(y) \d\mu_g(y).
\end{equation}
We deduce that
\begin{align*}
\norm{(T_K f)(x_0) - K(x_0,y_0)u_0}_{E_{x_0}}
&\ov{\eqref{inter-34}}{=}
\norm{\frac{1}{\mu(U)} \int_U \bigl(K(x_0,y)u(y) - K(x_0,y_0)u_0\bigr) \d\mu_g(y)}_{E_{x_0}} \\
&\leq
\frac{1}{\mu(U)} \int_U \norm{K(x_0,y)u(y) - K(x_0,y_0)u_0}_{E_{x_0}} \d\mu_g(y) 
\ov{\eqref{inter-37}}{\leq} \epsi.
\end{align*}
We infer that
\[
\norm{(T_K f)(x_0)}_{E_{x_0}}
\geq
\norm{K(x_0,y_0)u_0}_{E_{x_0}} - \epsi
\ov{\eqref{inter-38}}{=}
C - \epsi.
\]
Consequently, we obtain
\[
\norm{T_K f}_{\L^\infty(M,E)}
\geq \norm{(T_K f)(x_0)}_{E_{x_0}}
\geq C - \epsi.
\]
Since $\epsi > 0$ is arbitrary and $\norm{f}_{\L^1(M,E)}=1$, we get $\norm{T_K}_{\L^1(M,E) \to \L^\infty(M,E)} \geq C$.
\end{proof}

\begin{prop}
\label{prop:Gk-compact}
Let $(M,g)$ be a smooth compact Riemannian manifold of dimension $d$. For any integer $k \in \{0,\dots,d\}$, there exist constants $C,c>0$ such that, for all $x,y\in M$,
\begin{equation}
\label{eq:Gk-compact}
\norm{p_t^k(x,y)}
\leq
\frac{C}{V(x,\sqrt t)}
\exp\Bigl(-c\frac{\dist(x,y)^2}{t}\Bigr), \quad t > 0,
\end{equation}
where $\norm{\cdot}$ is the operator norm on $\Hom_{\mathbb{C}}(\Lambda_{\mathbb{C}}^k\mathrm{T}_y^*M,\Lambda_{\mathbb{C}}^k\mathrm{T}_x^*M)$ and $V(x,r)$ is the volume of the geodesic ball $B(x,r)$.
\end{prop}

\begin{proof}
By \cite[Corollary 5.3.5 p.~142]{Hsu02} the heat kernel $p_t(x,y)$ of the operator $\e^{t\Delta}$ acting on functions satisfies an estimate of the form
$$
p_t(x,y) 
\lesssim \frac{1}{t^{\frac{d}{2}}}, \quad x,y \in M, 0<t<1.
$$
By \cite[Exercise 16.5 p.~424]{Gri09} (or \cite[Theorem 1.1 p.~35]{Gri97}), we deduce the existence of constant $c > 0$ such that
\begin{equation}
\label{eq:G0-compact}
p_t(x,y) 
\lesssim \frac{1}{V(x,\sqrt{t})}\exp\Bigl(-c\frac{\dist(x,y)^2}{t}\Bigr), \quad x,y \in M, 0 <t < 1.
\end{equation}
Since the manifold $M$ is compact, there exists $a_k \geq 0$ such that $\W_k \geq -a_k$. Proposition \ref{prop-domination} yields the semigroup domination
\begin{equation}
\label{eq:domination-compact}
\bigl|\e^{-t\Delta_{\HdR,k}}\omega\bigr|
\ov{\eqref{domination}}{\leq}
\e^{a_k t}\e^{t\Delta}(|\omega|),
\quad t > 0, \omega \in \L^2(\Omega^k(M)).
\end{equation}
By Lemma \ref{lemma-kernel-domination}, this implies the pointwise kernel domination
\begin{equation}
\label{eq:kernel-domination}
\norm{p_t^k(x,y)}
\leq
\e^{a_k t}\,p_t(x,y),
\quad x,y \in M,t > 0.
\end{equation}
Fix $t\in(0,1]$. Using $\e^{a_k t}\leq \e^{a_k}$ and inserting \eqref{eq:G0-compact} into \eqref{eq:kernel-domination}
we obtain \eqref{eq:Gk-compact} for all $t \in (0,1]$ (with possibly different constants).

Let $P \ov{\mathrm{def}}{=} P_2 \co \L^2(\Omega^k(M)) \to \L^2(\Omega^k(M))$ be the orthogonal projection onto the finite-dimensional space $\cal{H}^k(M)$ of harmonic $k$-forms. Then 
$$
\e^{-t\Delta_{\HdR,k}}
=\e^{-t\Delta_{\HdR,k}}P+\e^{-t\Delta_{\HdR,k}}(\Id-P).
$$ 
According to \cite[p.~126 and Corollary 5.20 p.~128]{Voi07} and \cite[(24.48.4) p.~319]{Dieu82}, $\Delta_{\HdR,k}$ is an elliptic pseudo-differential operator of order 2 with a principal symbol of bijective type. Consequently, by \cite[(23.35.2)]{Dieu88} its spectrum is discrete, entirely composed of eigenvalues. Hence, on the subspace $(\Id-P)\L^2(\Omega^k(M))$ the spectrum of $\Delta_{\HdR,k}$ is contained in $[\lambda_k,\infty)$ for some $\lambda_k > 0$. 
For any $t \geq 2$, we obtain by spectral theory
\begin{align}
\MoveEqLeft
\label{inter98}
\bnorm{\e^{-(t-2)\Delta_{\HdR,k}}(\Id-P)}_{\L^2(\Omega^k(M)) \to \L^2(\Omega^k(M))}
\leq \e^{-\lambda_k(t-2)} 
\lesssim \e^{-\lambda_k t}.        
\end{align}
Now, for any $x \in M$ and any form $\omega \in \Omega^{k}(M)$, we have
\begin{align*}
\MoveEqLeft
\norm{\e^{-\Delta_{\HdR,k}}\omega(x)}_{\Lambda^{k}\mathrm{T}_x^*M}
\ov{\eqref{integral-operator}}{=} \norm{\int_M p_1^{k}(x,y)\omega(y) \d\mu_g(y) }_{\Lambda^{k}\mathrm{T}_x^*M} \\
&\leq \int_M \norm{p_1^{k}(x,y)} \norm{\omega(y)}_{_{\Lambda^{k}\mathrm{T}_y^*M }}
\d\mu_g(y) 
\ov{\eqref{Lp-norm-vector-bundle}}{\leq} \left( \int_M \norm{p_1^{k}(x,y)}^2 \d\mu_g(y) \right)^{\frac{1}{2}} \norm{\omega}_{\L^2(\Omega^{k}(M))}.
\end{align*} 
Since $p_1^{k}$ is smooth on the compact manifold $M \times M$, we have
$$
\sup_{x \in M} \bigg( \int_M \norm{p_1^{k}(x,y)}^2 \d\mu_g(y)
\bigg)^{\frac{1}{2}} <\infty.
$$ 
Consequently $\e^{-\Delta_{\HdR,k}}$ induces a bounded operator from $\L^2(\Omega^{k}(M))$ into $\L^\infty(\Omega^{k}(M))$. By duality, using the symmetry of $\Delta_{\HdR,k}$ provided by \cite[Proposition 3 (i) p.~540]{GMS98}, we have a bounded operator $\e^{-\Delta_{\HdR,k}} \co \L^1(\Omega^{k}(M)) \to \L^2(\Omega^{k}(M))$.  
For any $t \geq 2$, we obtain
\begin{align*}
\MoveEqLeft
\norm{\e^{-t\Delta_{\HdR,k}}(\Id-P)}_{\L^1(\Omega^k(M)) \to \L^\infty(\Omega^k(M))} \\
&\leq \norm{\e^{-\Delta_{\HdR,k}}}_{\L^2 \to \L^\infty} \bnorm{\e^{-(t-2)\Delta_{\HdR,k}}(\Id-P)}_{\L^2 \to \L^2} \norm{\e^{-\Delta_{\HdR,k}}}_{\L^1 \to \L^2} 
\ov{\eqref{inter98}}{\lesssim}  \e^{-\lambda_k t}. 
\end{align*}
Moreover, the map $P=\e^{-t\Delta_{\HdR,k}}P \co \L^1(\Omega^k(M)) \to \L^\infty(\Omega^k(M))$ is clearly bounded (according to \cite[(24.48.7) p.~321]{Dieu82}, it is an integral operator which is smoothing). We deduce that the operator $\e^{-t\Delta_{\HdR,k}} \co \L^1(\Omega^k(M)) \to \L^\infty(\Omega^k(M))$ is bounded with norm $\leq C$ for any $t \geq 2$. By Proposition \ref{prop:kernel-norm-equality-bundles}, we deduce that $\norm{p_t^k(x,y)} \leq C$ for any $t \geq 2$ and any $x,y \in M$. Finally, for $t \geq 2$ one has $V(x,\sqrt t) \simeq \mu_g(M)$ and 
$$
\exp(-c\diam(M)^2) 
\leq \exp\Bigl(-c\frac{\dist(x,y)^2}{t}\Bigr).
$$ 
On the compact set $[1,2] \times M \times M$, the heat kernel $(t,x,y) \mapsto p_t^{k}(x,y)$ is smooth, hence bounded.  
So the same estimate also holds for $1 \leq t \leq 2$, after increasing $C$.
\end{proof}

\subsection{Some properties of the closure of the Hodge-de Rham Laplacian}

We need the following standard fact.

\begin{lemma}
\label{lem:generator-heat-semigroup-closure-HdR}
Let $(M,g)$ be a smooth compact Riemannian manifold. Suppose that $1 < p < \infty$. Let $(T_t)_{t \geq 0}$ be the consistent extension to the Banach space $\L^p(\Omega^\bullet(M))$ of the heat semigroup $(\e^{-t\Delta_{\HdR,2}})_{t \geq 0}$ on $\L^2(\Omega^\bullet(M))$, and let $-A_p$ denote its generator. Then $
A_p=\Delta_{\HdR,p}$.
\end{lemma}

\begin{proof}
By definition, the subspace $\Omega^\bullet(M)$ is a core for $\Delta_{\HdR,p}$. We claim that $\Omega^\bullet(M)$ is also a core for $A_p$. First, the subspace $\Omega^\bullet(M)$ is dense in $\L^p(\Omega^\bullet(M))$. Moreover, it is invariant under the semigroup $(T_t)_{t \geq 0}$ since the heat semigroup preserves smooth differential forms. 
If $\omega\in\Omega^\bullet(M)$, using the strong continuity of $(T_t)_{t \geq 0}$ on $\L^p(\Omega^\bullet(M))$ and \cite[Lemma 1.3 p.~50]{EnN00} \cite[Proposition 1.2.2 p.~16]{ABHN11}, we see that
\begin{align*}
\MoveEqLeft
\frac{T_t\omega-\omega}{t}
=\frac{\e^{-t\Delta_{\HdR,2}}\omega-\omega}{t}
=-\frac1t\int_0^t \e^{-s\Delta_{\HdR,2}}\Delta_{\HdR}\omega \d s \\
&=-\frac1t\int_0^t T_s\Delta_{\HdR}\omega \d s
\xra[t \to 0]{}
-\Delta_{\HdR}\omega.         
\end{align*}
in $\L^p(\Omega^\bullet(M))$. By definition of $A_p$, we have $\Omega^\bullet(M) \subset \dom A_p$ and $A_p\omega=\Delta_{\HdR}\omega$ for any $\omega \in \Omega^\bullet(M)$. Therefore $\Omega^\bullet(M)$ is a dense subspace of $\L^p(\Omega^\bullet(M))$, contained in $\dom A_p$ and invariant under $(T_t)_{t \geq 0}$. By \cite[Proposition G.2.4 p.~526]{HvNVW18}, it is a core for $A_p$. Thus the two closed operators $A_p$ and $\Delta_{\HdR,p}$ coincide on the common core $\Omega^\bullet(M)$. We conclude that $A_p=\Delta_{\HdR,p}$.
\end{proof}

Now, we prove that we can identify the square $D_p^2$ with the closure $\Delta_{\HdR,p}$ of $\Delta_{\HdR}$. We first need an intermediary result. We refer to \cite{GrS13} and to \cite[p.~483]{Nic21} for more information on Sobolev spaces on Riemannian manifolds.

\begin{lemma}
\label{lemma-I-plus-Kodaira-isomorphism-W2p}
Let $(M,g)$ be a smooth compact Riemannian manifold.  Suppose that $1 < p < \infty$. Then
\begin{equation}
\label{eq-domain-Kodaira-W2p}
\dom \Delta_{\HdR,p}
=
\W^{2,p}_\nabla(\Omega^{\bullet}(M)).
\end{equation}
%
Moreover, the operator
\begin{equation}
\label{eq-I-plus-Kodaira-closure-isomorphism}
\Id+\Delta_{\HdR,p}
\co
\W^{2,p}_\nabla(\Omega^{\bullet}(M))
\to
\L^p(\Omega^{\bullet}(M))
\end{equation}
is an isomorphism of Banach spaces. Finally, $-1 \in \rho(\Delta_{\HdR,p})$ and the graph norm of the operator $\Delta_{\HdR,p}$ is equivalent to the Sobolev norm, i.e.,
\begin{equation}
\label{eq-Kodaira-W2p-graph-norm}
\norm{u}_{\W^{2,p}_\nabla(\Omega^{\bullet}(M))}
\approx
\norm{u}_{\L^p(\Omega^{\bullet}(M))}
+
\norm{\Delta_{\HdR,p}u}_{\L^p(\Omega^{\bullet}(M))},
\quad u \in \dom \Delta_{\HdR,p}.
\end{equation}
\end{lemma}

\begin{proof}
Note that the operator $\Id+\Delta_{\HdR}$ is an elliptic differential operator of order $2$. Using \cite[Proposition 5.1]{Arh26b}, we obtain a well-defined bounded operator $\Id+\Delta_{\HdR}
\co
\W^{2,p}_\nabla(\Omega^{\bullet}(M))
\to
\L^p(\Omega^{\bullet}(M))$. We first prove that $\Id+\Delta_{\HdR}$ is injective. Consider some $u \in \W^{2,p}_\nabla(\Omega^{\bullet}(M))$ such that $(\Id+\Delta_{\HdR})u=0$. By elliptic regularity \cite[(23.22.8)]{Dieu88}, 
$u$ is smooth. Taking the $\L^2$ inner product with $u$, we get
\begin{align*}
0
&=
\big\la(\Id+\Delta_{\HdR})u,u
\big\ra_{\L^2(\Omega^{\bullet}(M))}
=
\norm{u}_{\L^2(\Omega^{\bullet}(M))}^2
+
\big\la\Delta_{\HdR}u,u
\big\ra_{\L^2(\Omega^{\bullet}(M))}
\\
&\ov{\eqref{Hodge-deRham-Laplacian}}{=}
\norm{u}_{\L^2(\Omega^{\bullet}(M))}^2
+
\norm{\d u}_{\L^2(\Omega^{\bullet}(M))}^2
+
\norm{\d^*u}_{\L^2(\Omega^{\bullet}(M))}^2.
\end{align*}
Consequently, $u=0$. Hence the operator $\Id+\Delta_{\HdR}$ is injective.

Now, we prove that $\Id+\Delta_{\HdR}$ has a closed range. By \cite[(23.30.6)]{Dieu88} and \cite[Theorem 8.6 p.~207]{Gru09}, there exists a pseudo-differential operator $Q \co \Omega^{\bullet}(M) \to \Omega^{\bullet}(M)$ of order $-2$ and a smoothing operator $R \co \Omega^{\bullet}(M) \to \Omega^{\bullet}(M)$ such that
\begin{equation}
\label{eq:parametrix-elliptic-bis}
\Id 
= Q(\Id+\Delta_{\HdR}) + R.
\end{equation}
Hence for any $u \in \W^{2,p}_\nabla(\Omega^{\bullet}(M))$ we have using \cite[Proposition 5.1]{Arh26b} and following \cite[p.~193]{LaM89}.
\begin{align}
\MoveEqLeft
\label{eq:elliptic-estimate-P}
\norm{u}_{\W^{2,p}_\nabla(\Omega^{\bullet}(M))}
\ov{\eqref{eq:parametrix-elliptic-bis}}{=} \norm{Q (\Id+\Delta_{\HdR}) u+Ru}_{\W^{2,p}_\nabla(\Omega^{\bullet}(M))} \\
&\lesssim \norm{(\Id+\Delta_{\HdR}) u}_{\L^p}+\norm{Ru}_{\W^{2,p}_\nabla(\Omega^{\bullet}(M))}
\lesssim \norm{(\Id+\Delta_{\HdR}) u}_{\L^p}+\norm{u}_{\L^p}.  \nonumber       
\end{align}
We claim that there exists a constant \(C>0\) such that
\begin{equation}
\label{eq-coercive-estimate-I-plus-Kodaira}
\norm{u}_{\W^{2,p}_\nabla(\Omega^{\bullet}(M))}
\leq
C
\norm{(\Id+\Delta_{\HdR})u}_{\L^p(\Omega^{\bullet}(M))},
\quad
u\in
\W^{2,p}_\nabla(\Omega^{\bullet}(M)).
\end{equation}
Suppose that it is not the case. There exists a sequence $(u_j)$ in $\W^{2,p}_\nabla(\Omega^{\bullet}(M))$ such that
\[
\norm{u_j}_{\W^{2,p}_\nabla(\Omega^{\bullet}(M))}=1
\quad \text{and} \quad
\norm{(\Id+\Delta_{\HdR})u_j}_{\L^p(\Omega^{\bullet}(M))} \to 0.
\]
By compactness \cite[Theorem 10.2.36 p.~483]{Nic21} of the embedding $\W^{2,p}_\nabla(\Omega^{\bullet}(M)) \hookrightarrow \L^p(\Omega^{\bullet}(M))$, after passing to a subsequence, we may assume that $(u_j)$ converges in $\L^p(\Omega^{\bullet}(M))$ to some $u$. Using \eqref{eq:elliptic-estimate-P} applied to $u_j-u_k$, we infer that the sequence $(u_j)$ is Cauchy in $\W^{2,p}_\nabla(\Omega^{\bullet}(M))$. Since the space $\W^{2,p}_\nabla(\Omega^{\bullet}(M))$ is complete, there exists $v \in \W^{2,p}_\nabla(\Omega^{\bullet}(M))$ such that $
u_j \to v
\quad \text{in } \W^{2,p}_\nabla(\Omega^{\bullet}(M))$. By the continuous embedding $
\W^{2,p}_\nabla(\Omega^{\bullet}(M))
\hookrightarrow
\L^p(\Omega^{\bullet}(M))$, we also have $u_j \to v$ in $\L^p(\Omega^{\bullet}(M))$. On the other hand, by construction of the subsequence,  $u_j \to u$ in $\L^p(\Omega^{\bullet}(M))$. By uniqueness of the limit in $\L^p(\Omega^{\bullet}(M))$, we obtain $u=v$. Thus $u_j \to u$ in $\W^{2,p}_\nabla(\Omega^{\bullet}(M))$. Since $(\Id+\Delta_{\HdR})u_j \to 0$ in $\L^p(\Omega^{\bullet}(M))$, we get $(\Id+\Delta_{\HdR})u=0$. By injectivity, $u=0$, contradicting
\[
\norm{u}_{\W^{2,p}_\nabla(\Omega^{\bullet}(M))}
=
\lim_j \norm{u_j}_{\W^{2,p}_\nabla(\Omega^{\bullet}(M))}
=
1.
\]
This proves \eqref{eq-coercive-estimate-I-plus-Kodaira}. Consequently, by \cite[Theorem 2.5 p.~70]{AbA02}, the subspace $\Ran (\Id+\Delta_{\HdR})$ is closed in the space $\L^p(\Omega^{\bullet}(M))$.

It remains to prove that the subspace $\Ran (\Id+\Delta_{\HdR})$ is dense. Let $g \in \L^{p^*}(\Omega^{\bullet}(M))$ annihilate $\Ran (\Id+\Delta_{\HdR})$, where $p^*$ is the conjugate exponent of $p$. Then
\[
\big\la g,(\Id+\Delta_{\HdR})u \big\ra_{\L^{p^*},\L^p}
=0,
\quad
u \in \W^{2,p}_\nabla(\Omega^{\bullet}(M)).
\]
In particular, this holds for every smooth form $u$. Since $\Id+\Delta_{\HdR}$ is formally selfadjoint,  we have, by definition of the distribution $(\Id+\Delta_{\HdR})g$,
\[
\big\la (\Id+\Delta_{\HdR})g,u\big\ra_{\cal{D}',\cal{D}}
=
\big\la g,(\Id+\Delta_{\HdR})u\big\ra_{\cal{D}',\cal{D}}
=
\big\la g, (\Id+\Delta_{\HdR})u \big\ra_{\L^{p^*},\L^p}
=0
\]
for every $u\in\Omega^{\bullet}(M)$. 
Hence $
(\Id+\Delta_{\HdR})g=0$ distributionally. By elliptic regularity of $\Id+\Delta_{\HdR}$ \cite[(23.22.8)]{Dieu88}, $g$ is smooth. Taking the $\L^2$ inner product with $g$, the same positivity argument as previously gives $g=0$. Thus the annihilator of $\Ran (\Id+\Delta_{\HdR})$ in $\L^{p^*}(\Omega^{\bullet}(M))$ is trivial. Hence the subspace $\Ran (\Id+\Delta_{\HdR})$ is dense in $\L^p(\Omega^{\bullet}(M))$.

Since $\Ran (\Id+\Delta_{\HdR})$ is both closed and dense in $\L^p(\Omega^{\bullet}(M))$, it is equal to $\L^p(\Omega^{\bullet}(M))$. Consequently $\Id+\Delta_{\HdR}$ is bijective. Since it is a bounded bijection between Banach spaces, its inverse is bounded by \cite[Corollary 1.6.6 p.~44]{Meg98}.

Let $\Delta_{\HdR,p}$ be the closure of the unbounded operator $\Delta_{\HdR}$. We first prove the inclusion $
\W^{2,p}_\nabla(\Omega^{\bullet}(M))
\subset
\dom\Delta_{\HdR,p}$. Let \(u\in\W^{2,p}_\nabla(\Omega^{\bullet}(M))\). Choose a sequence \((u_j)\) in \(\Omega^{\bullet}(M)\) such that $u_j\to u$ in \(\W^{2,p}_\nabla(\Omega^{\bullet}(M))\). By \cite[Proposition 5.1]{Arh26b}, we have $
\Delta_{\HdR}u_j
\to
\Delta_{\HdR}u$ in \(\L^p(\Omega^{\bullet}(M))\). Hence $u \in \dom \Delta_{\HdR,p}$ and $
\Delta_{\HdR,p}u
=
\Delta_{\HdR}u$.

Conversely, let $u \in \dom\Delta_{\HdR,p}$. By the definition of the closure, there exists a sequence \((u_j)\) in $\Omega^{\bullet}(M)$ such that $u_j \to u$ and $\Delta_{\HdR}u_j
\to
\Delta_{\HdR,p}u$ in $\L^p(\Omega^{\bullet}(M))$. Consequently, we have 
$(\Id+\Delta_{\HdR})u_j
\to
u+\Delta_{\HdR,p}u$ in \(\L^p(\Omega^{\bullet}(M))\). Applying the bounded operator $S  \ov{\mathrm{def}}{=}
(\Id+\Delta_{\HdR})^{-1}
\co
\L^p(\Omega^{\bullet}(M))
\to
\W^{2,p}_\nabla(\Omega^{\bullet}(M))$, we obtain
\[
u_j
=
S(\Id+\Delta_{\HdR})u_j
\to
S(u+\Delta_{\HdR,p}u)
\]
in $\W^{2,p}_\nabla(\Omega^{\bullet}(M))$. Since the embedding of $\W^{2,p}_\nabla(\Omega^{\bullet}(M))$ into $\L^p(\Omega^{\bullet}(M))$ is continuous and $u_j \to u$ in $\L^p(\Omega^{\bullet}(M))$, the uniqueness of the limit gives $
u
=S(u+\Delta_{\HdR,p}u)$. Thus $
u\in\W^{2,p}_\nabla(\Omega^{\bullet}(M))$. This proves \eqref{eq-domain-Kodaira-W2p}.

It also shows that $\Id+\Delta_{\HdR,p}$ coincides, under the identification of its domain with the Sobolev space $\W^{2,p}_\nabla(\Omega^{\bullet}(M))$, with the isomorphism $\Id+\Delta_{\HdR}$ constructed previously. Hence \eqref{eq-I-plus-Kodaira-closure-isomorphism} holds and $-1 \in \rho(\Delta_{\HdR,p})$.

Finally, \cite[Proposition 5.1]{Arh26b} gives
\[
\norm{u}_{\L^p(\Omega^{\bullet}(M))}
+
\norm{\Delta_{\HdR,p}u}_{\L^p(\Omega^{\bullet}(M))}
\lesssim
\norm{u}_{\W^{2,p}_\nabla(\Omega^{\bullet}(M))},
\]
whereas \eqref{eq-coercive-estimate-I-plus-Kodaira} gives
\begin{align*}
\norm{u}_{\W^{2,p}_\nabla(\Omega^{\bullet}(M))}
&\ov{\eqref{eq-coercive-estimate-I-plus-Kodaira}}{\lesssim}
\norm{(\Id+\Delta_{\HdR,p})u}_{\L^p(\Omega^{\bullet}(M))}
\leq
\norm{u}_{\L^p(\Omega^{\bullet}(M))}
+
\norm{\Delta_{\HdR,p}u}_{\L^p(\Omega^{\bullet}(M))}.
\end{align*}
This proves \eqref{eq-Kodaira-W2p-graph-norm}.
\end{proof}

\begin{prop}
\label{prop:closure-Hodge-square-dirac}
Let $(M,g)$ be a smooth compact Riemannian manifold. Suppose that $1<p<\infty$. Then
\begin{equation}
\label{eq:Hodge-square-Dirac}
\Delta_{\HdR,p}
=D_p^2.
\end{equation}
\end{prop}

\begin{proof}
Since $D$ is a first-order differential operator, \cite[Proposition 5.1]{Arh26b} gives a bounded operator
$
D\co \W^{1,p}_\nabla(\Omega^\bullet(M))
\to
\L^p(\Omega^\bullet(M))$. We claim that
\begin{equation}
\label{eq-W1p-contained-domain-D}
\W^{1,p}_\nabla(\Omega^\bullet(M))
\subset\dom D_p,
\qquad
D_pu=Du,
\quad
u \in \W^{1,p}_\nabla(\Omega^\bullet(M)).
\end{equation}
Indeed, let $u\in\W^{1,p}_\nabla(\Omega^\bullet(M))$. By density of $\Omega^\bullet(M)$ in $\W^{1,p}_\nabla(\Omega^\bullet(M))$, there exists a sequence $(u_n)$ in $\Omega^\bullet(M)$ such that $u_n \to u$ in $\W^{1,p}_\nabla(\Omega^\bullet(M))$. In particular, we have $
u_n\to u
\quad\text{in }\L^p(\Omega^\bullet(M))$. Moreover, by the boundedness of $
D \co \W^{1,p}_\nabla(\Omega^\bullet(M))
\to\L^p(\Omega^\bullet(M))$, we have $Du_n \to Du$ in $\L^p(\Omega^\bullet(M))$. So $u \in \dom D_p$ and $D_pu=Du$. The claim is proved.

Let $u \in \dom\Delta_{\HdR,p}$. By Lemma~\ref{lemma-I-plus-Kodaira-isomorphism-W2p}, we have $u \in \W^{2,p}_\nabla(\Omega^{\bullet}(M))$. Since $D$ has order one, \cite[Proposition 5.1]{Arh26b} gives $Du \in \W^{1,p}_\nabla(\Omega^{\bullet}(M))$. Hence, by \eqref{eq-W1p-contained-domain-D}, we have $u \in \dom D_{p}^2$ and $
D_{p}^2u
=D^2u
\ov{\eqref{Hodge-deRham-Laplacian}}{=} \Delta_{\HdR}u
=\Delta_{\HdR,p}u$. Consequently, we obtain
\begin{equation}
\label{eq-inclusion-Kodaira-square-D}
\Delta_{\HdR,p}
\subset D_{p}^2.
\end{equation}
We next prove that $\Id+ D_{p}^2$ is injective. Consider some $u \in \dom D_{p}^2$ satisfying $
(\Id+ D_{p}^2)u=0$. Since $D_p$ is the closure of $D$, for every $v \in \dom D_p$ we have
\[
Dv=D_pv
\]
in the sense of distributions. Indeed, if $(v_n)$ is a sequence in $\Omega^\bullet(M)$ such that $v_n\to v$ and $Dv_n\to D_pv$ in $\L^p(\Omega^\bullet(M))$, then passing to the limit in the sense of distributions \cite[(23.28.8)]{Dieu88} gives $Dv=D_pv$. Applying this observation first to $u$ and then to $D_pu\in\dom D_p$, we obtain
\[
D^2u=D_p^2u
\]
in the sense of distributions. Since $D^2=\Delta_{\HdR}$ on smooth forms, and hence as differential operators on distributions, we infer that $(\Id+\Delta_{\HdR})u=0$ in the sense of distributions. The differential operator $\Id+\Delta_{\HdR}$ is elliptic. By elliptic regularity, it follows that $
u \in \Omega^\bullet(M)$. In particular, $u\in\L^2(\Omega^\bullet(M))$. Taking the $\L^2$ inner product with $u$, we obtain
\[
0
=
\big\la(\Id+\Delta_{\HdR}) u,u \big\ra_{\L^2(\Omega^\bullet(M))}
=\norm{u}_{\L^2(\Omega^\bullet(M))}^2
+\norm{\d u}_{\L^2(\Omega^\bullet(M))}^2
+\norm{\d^*u}_{\L^2(\Omega^\bullet(M))}^2.
\]
Thus $u=0$, proving the injectivity of $\Id+ D_{p}^2$. By Lemma~\ref{lemma-I-plus-Kodaira-isomorphism-W2p}, the operator $\Id+\Delta_{\HdR,p}$ is surjective. We can therefore apply Lemma~\ref{lemma-extension-surjective-injective} to the inclusion \eqref{eq-inclusion-Kodaira-square-D}, with $A=\Delta_{\HdR,p}$, $B= D_{p}^2$ and $\lambda=1$. This gives \eqref{eq:Hodge-square-Dirac}.
\end{proof}

\subsection{Boundedness of the functional calculi}

We shall use the following result \cite[Proposition 9.1]{Arh26b}, which extends to operators acting on $\L^p$-spaces of vector bundles a result of Duong and Robinson \cite[Theorem 3.4 p.~108]{DuR96} and the vector-valued version due to Haller-Dintelmann \cite{Hal05}. Indeed, we can use \cite[Remark I.17 p.~13]{Gun17}, which says that each space $\L^p(M,E)$ is isometrically isomorphic to a Bochner space with values in a finite-dimensional space.

\begin{prop}
\label{prop:Duong-Robinson-vector-bundle}
Let $E$ be a Hermitian complex vector bundle of finite rank over a smooth compact Riemannian manifold $M$. Let $A$ be a sectorial operator on $\L^2(M,E)$ such that $-A$ generates a bounded holomorphic semigroup $
(S_z)_{z \in \Sigma_\theta}$ on $\L^2(M,E)$ for some angle $\theta \in (0,\frac{\pi}{2})$. Assume that, for every $t>0$, the operator $S_t$ admits an integral kernel $
K_t(x,y) \in \Hom(E_y,E_x)$ in the sense that
\[
(S_t\omega)(x)
=
\int_M K_t(x,y)\omega(y) \d\mu_g(y),
\quad
\omega \in \L^2(M,E),
\]
and that there exist constants $c,C>0$ and $m>0$ such that
\begin{equation}
\label{eq:kernel-domination-vector-bundle}
\norm{K_t(x,y)}_{E_y \to E_x}
\leq
\frac{c}{V(x,t^{\frac{1}{m}})}
\exp\left(-C\frac{\dist(x,y)^m}{t}\right),
\quad x,y \in M, t>0.
\end{equation}
Finally, assume that the operator $A$ admits a bounded $\H^\infty(\Sigma_\mu)$ functional calculus on the Hilbert space $\L^2(M,E)$ for some angle $\mu \in \left(\frac{\pi}{2}-\theta,\frac{\pi}{2}\right)$. Then, for every $p \in (1,\infty)$, the semigroup $(S_t)_{t>0}$ extends consistently to a bounded holomorphic semigroup on $\L^p(M,E)$. If $-A_p$ denotes its generator on $\L^p(M,E)$, then $A_p$ admits a bounded $\H^\infty(\Sigma_{\theta'})$ functional calculus for some angle $\theta' \in (0,\frac{\pi}{2})$.
\end{prop}

Now, we obtain a fundamental property of functional calculus of the Hodge-de Rham Laplacian.

\begin{thm}
\label{Th-funct-HdR}
Let $M$ be a  smooth compact Riemannian manifold. Suppose that $1 < p < \infty$. The closure $\Delta_{\HdR,p}$ of the Hodge-de Rham Laplacian $\Delta_{\HdR}$ is sectorial on the Banach space $\L^p(\Omega^\bullet(M))$ and admits a bounded $\H^\infty(\Sigma_\theta)$ functional calculus for some angle $\theta \in (0,\frac{\pi}{2})$.
\end{thm}

\begin{proof}
By Proposition \ref{prop:Gk-compact}, the semigroup generated by the operator $-\Delta_{\HdR}$ admits ``Gaussian estimates''. So the assumption \eqref{eq:kernel-domination-vector-bundle} of Proposition \ref{prop:Duong-Robinson-vector-bundle} is satisfied. Note that $\Delta_{\HdR,2}=D_2^2$ is a positive selfadjoint operator on the Hilbert space $\L^2(\Omega^\bullet(M))$. So by \cite[Proposition 10.2.23 p.~388]{HvNVW18} it admits a bounded $\H^\infty(\Sigma_\mu)$ functional calculus for any angle $\mu >0$. By combining Lemma \ref{lem:generator-heat-semigroup-closure-HdR} and Proposition \ref{prop:Duong-Robinson-vector-bundle}, we conclude that the operator $\Delta_{\HdR,p}$ admits a bounded $\H^\infty(\Sigma_\theta)$ functional calculus for some angle $\theta \in (0,\frac{\pi}{2})$ on the Banach space $\L^p(\Omega^\bullet(M))$.
\end{proof}

\begin{prop}
\label{prop-bounded-Riesz-transforms}
Let $(M,g)$ be a smooth compact Riemannian manifold. Suppose that $1 < p < \infty$. The Riesz transforms $\d\Delta_{\HdR}^{-\frac{1}{2}}$ and $\d^*\Delta_{\HdR}^{-\frac{1}{2}}$ are bounded on the space $\L^p(\Omega^\bullet(M))$.
\end{prop}

\begin{proof}
Note that by \cite[(23.29.10)]{Dieu88} the exterior derivative $\d$ is a differential operator of order $1$ and that by \cite[p.~126 and Corollary 5.20 p.~128]{Voi07} (see also \cite[pp.~318-319]{Dieu82} for the orientable case) the Hodge-de Rham Laplacian $\Delta_{\HdR}$ is an elliptic differential operator of order $2$. Let $P \co \Omega^\bullet(M) \to \cal{H}(M)$ be the projection on the subspace $\cal{H}(M) \ov{\mathrm{def}}{=} \oplus_{k=0}^d \cal{H}^k(M)$ of harmonic forms, which is smoothing, according to \cite[(24.48.7) p.~321]{Dieu82}. In particular by \cite[Theorem 1 p.~124]{Bur68} \cite{See67} \cite{Shu01}, we see that $\Delta_{\HdR}^{-\frac{1}{2}}(\Id-P)$ is a pseudo-differential operator of order $-1$ (defined as 0 on the subspace of harmonic forms). Hence by the composition rule \cite[17.13.5 p.~298]{Dieu72}, $\d\Delta_{\HdR}^{-\frac{1}{2}}(\Id-P)$ is a pseudo-differential operator of order $0$, which we symbolically denote by $\d\Delta_{\HdR}^{-\frac{1}{2}}$. Therefore, by a classical result \cite[Proposition 5.1]{Arh26b}, it extends to a bounded operator on the Banach space $
\L^p(\Omega^\bullet(M))
=
\ker \Delta_{\HdR,p}
\oplus
\ovl{\Ran \Delta_{\HdR,p}}$, 
where the decomposition follows from \cite[Proposition 2.1.1 (h) p.~21]{Haa06}. Note that $\d\Delta_{\HdR}^{-\frac{1}{2}}$ is understood as acting by zero on the Banach subspace $
\ker \Delta_{\HdR,p}
=
\ker \Delta_{\HdR,p}^{\frac12}$, 
where the latter equality is provided by \cite[Proposition 3.1.1 (d) p.~61]{Haa06}. On the subspace $\Ran \Delta_{\HdR,p}$, the Riesz transform is initially defined in the usual way, and the preceding boundedness allows it to extend uniquely by continuity to $\ovl{\Ran \Delta_{\HdR,p}}$. A similar reasoning shows that the Riesz transform $\d^*\Delta_{\HdR}^{-\frac{1}{2}}$ is also bounded on the space $\L^p(\Omega^\bullet(M))$.
\end{proof}

\begin{cor}
\label{cor-Hodge-Dirac-functional-calculus}
Let $(M,g)$ be a smooth compact Riemannian manifold. Suppose that $1 < p < \infty$. The Hodge--Dirac operator $D_p$ is bisectorial and admits a bounded $\H^\infty(\Sigma^\bi_\theta)$ functional calculus for some angle $\theta \in (0,\frac{\pi}{2})$ on the Banach space $\L^p(\Omega^\bullet(M))$.
\end{cor}

\begin{proof}
Combining Proposition \ref{prop-bounded-Riesz-transforms} with Theorem \ref{Th-funct-HdR} and a similar argument to that of \cite{NeV17}, we conclude that the Hodge–Dirac operator $D_p$ is bisectorial on the Banach space $\L^p(\Omega^\bullet(M))$ and admits a bounded $\H^\infty(\Sigma^\bi_\theta)$ functional calculus for some angle $\theta \in (0,\frac{\pi}{2})$. An alternative way to obtain this conclusion is to combine Proposition \ref{prop-bounded-Riesz-transforms} with Theorem \ref{Th-funct-HdR} and to replace the argument from \cite{NeV17} with the abstract Hodge--Dirac machinery developed in \cite{Arh26b}. Indeed, we can apply \cite[Theorem 3.22]{Arh26b} and \cite[Theorem 3.25]{Arh26b} with $
X \ov{\mathrm{def}}{=} \L^p(\Omega^\even(M))$, $Y \ov{\mathrm{def}}{=} \L^p(\Omega^\odd(M))$ the restrictions $A \ov{\mathrm{def}}{=} \Delta_{\HdR,p}^\even$ and $\widetilde{A} \ov{\mathrm{def}}{=}\Delta_{\HdR,p}^\odd$ of the closed Hodge--de Rham Laplacian $\Delta_{\HdR,p}$ to even and odd forms. Moreover, we use $\partial \ov{\mathrm{def}}{=} D_{p,+}$ and $\partial^\dagger \ov{\mathrm{def}}{=} D_{p,-}$, $
T_t \ov{\mathrm{def}}{=} \e^{-tA}$ and $\widetilde T_t \ov{\mathrm{def}}{=} \e^{-t\widetilde A}$ for any $t > 0$. 

Observe that on smooth even forms, the Hodge--de Rham Laplacian $\Delta_\HdR$ commutes with both $\d$ and $\d^*$ by \cite[(2.7.16) p.~166]{RuS17}. So, it is easy to check that the condition $\Curv_{\partial,\H^\infty}(0)$ of \cite[Definition 3.8]{Arh26b} is satisfied. The regularization property of \cite[Definition~3.15]{Arh26b} is proved by the same argument as in the proof of \cite[Corollary 5.23]{Arh26b}. Consider the bounded finite-rank projection $P_p^\odd \co \L^p(\Omega^\odd(M)) \to \L^p(\Omega^\odd(M))$ onto the space $\cal{H}^\odd(M)$ of odd harmonic forms, and set $Q \ov{\mathrm{def}}{=} \Id_{\L^p(\Omega^\odd(M))} - P_p^\odd$. As in the proof of \cite[Lemma 5.22]{Arh26b}, we can prove that $\Ran Q = \Ran D_{p,+}$, $\ker Q \subset \dom D_{p,-}$ and $D_{p,-}|_{\ker Q}=0$. 
\end{proof}

\section{Banach spectral triples from Hodge--Dirac operators on $\L^p(\Omega^\bullet(M))$}
\label{sec-Banach-spectral-triples}

\subsection{Domain of the Hodge--Dirac operator}
\label{sec-domain}

In \cite{NeV17}, the domain of the Hodge--Dirac operator $D_p$ on the space $\L^p(\Omega^\bullet(M))$ is not specified. In this section, we address and complete this missing part. Let $(M,g)$ be a smooth compact Riemannian manifold. Following essentially the definitions of \cite[p.~68]{ISS99} and \cite[Definition 3.2 p.~2079]{Sco95} (where the manifold $M$ is compact), we consider the norm
\begin{equation}
\label{eq:def-W1p-D-norm-again}
\norm{\omega}_{\scr{L}^{1,p}(\Omega^\bullet(M))}
\ov{\mathrm{def}}{=}
\norm{\omega}_{\L^p(\Omega^\bullet(M))}
+\norm{\d\omega}_{\L^p(\Omega^\bullet(M))}
+\norm{\d^*\omega}_{\L^p(\Omega^\bullet(M))}
\end{equation}
and we introduce the completion $\scr{L}^{1,p}(\Omega^\bullet(M))$ of the space $\Omega^{\bullet}(M)$ equipped with this norm.

The following result describes precisely the domain of the closure of the Hodge--Dirac operator.

\begin{prop}
\label{prop:graph-norm-Hodge-Dirac-Lp}
Let $(M,g)$ be a smooth compact Riemannian manifold. Suppose that $1 < p < \infty$. 
%
Then we have
\begin{equation}
\label{eq:domain-D-graph-precise}
\dom D_p
=
\scr{L}^{1,p}(\Omega^\bullet(M))
\end{equation}
and the graph norm of $D_p$ is equivalent to the Sobolev norm \eqref{eq:def-W1p-D-norm-again}, that is 
\begin{equation}
\label{eq:graph-norm-D-W1p-equivalence}
\norm{\omega}_{\scr{L}^{1,p}(\Omega^\bullet(M))}
\approx_p \norm{\omega}_{\L^p(\Omega^\bullet(M))}+\norm{D_p\omega}_{\L^p(\Omega^\bullet(M))}, \quad \omega \in \scr{L}^{1,p}(\Omega^\bullet(M)).
\end{equation}
\end{prop}

\begin{proof}
By Corollary \ref{cor-Hodge-Dirac-functional-calculus}, the operator $D_p$ is bisectorial on $\L^p(\Omega^\bullet(M))$ and admits a bounded $\H^\infty(\Sigma_\theta^\bi)$ functional calculus for some angle $\theta \in (0,\frac{\pi}{2})$. So the abstract theory \cite[Theorem 15.5.2 p.~501]{HvNVW23} of bisectorial operators entails that the operator $\Delta_{\HdR,p} \ov{\mathrm{def}}{=} D_p^2$ is sectorial, the equality of domains
$\dom D_p
=\dom\Delta_{\HdR,p}^{\frac12}$ ,
and the norm equivalence
\begin{equation}
\label{eq:equiv-D-Delta-half-proof}
\norm{D_p\omega}_{\L^p(\Omega^\bullet(M))}
\approx
\bnorm{\Delta_{\HdR,p}^{\frac12}\omega}_{\L^p(\Omega^\bullet(M))},
\quad
\omega \in \dom D_p.
\end{equation}
For any $\omega \in \Omega^\bullet(M)$, using the triangular inequality we obtain 
\begin{align*}
\MoveEqLeft
\norm{\omega}_{\L^p(\Omega^\bullet(M))} + \norm{D\omega}_{\L^p(\Omega^\bullet(M))}
\leq \norm{\omega}_{\L^p(\Omega^\bullet(M))}
+ \norm{\d\omega+\d^*\omega}_{\L^p(\Omega^\bullet(M))} \\
&\leq \norm{\omega}_{\L^p(\Omega^\bullet(M))}
+\norm{\d\omega}_{\L^p(\Omega^\bullet(M))}
+\norm{\d^*\omega}_{\L^p(\Omega^\bullet(M))} 
\ov{\eqref{eq:def-W1p-D-norm-again}}{=} \norm{\omega}_{\scr{L}^{1,p}(\Omega^\bullet(M))}.
\end{align*}
By Proposition \ref{prop-bounded-Riesz-transforms}, the reduced Riesz transforms, initially defined on $\Delta_{\HdR,p}^{\frac{1}{2}}(\Omega^{\bullet}(M))$ by
\begin{equation}
\label{Dol-Riesz}
R_{\d}\bigl(\Delta_{\HdR,p}^{\frac{1}{2}}\omega\bigr)
\ov{\mathrm{def}}{=}
\d\omega,
\quad \text{and} \quad
R_{\d^*}\bigl(\Delta_{\HdR,p}^{\frac{1}{2}}\omega\bigr)
\ov{\mathrm{def}}{=}
\d^*\omega, \quad \omega \in \Omega^{\bullet}(M),
\end{equation}
extend to bounded operators
$R_{\d},R_{\d^*}
\co
\ovl{\Ran\Delta_{\HdR,p}^{\frac{1}{2}}}
\to
\L^p(\Omega^{\bullet}(M))$. For any form $\omega \in \Omega^\bullet(M)$, we infer that
\begin{align}
\MoveEqLeft
\label{eq:d-dstar-estimate-Riesz}
\norm{\d\omega}_{\L^p(\Omega^\bullet(M))}
+\norm{\d^*\omega}_{\L^p(\Omega^\bullet(M))}
\ov{\eqref{Dol-Riesz}}{=} \bnorm{R_{\d}\bigl(\Delta_{\HdR,p}^{\frac12}\omega\bigr)}_{\L^p(\Omega^\bullet(M))}
+\bnorm{R_{\d^*}\bigl(\Delta_{\HdR,p}^{\frac12}\omega\bigr)}_{\L^p(\Omega^\bullet(M))}  \\
&\leq 
\bigl(\norm{R_{\d}}_{}+\norm{R_{\d^*}}_{}\bigr) \bnorm{\Delta_{\HdR}^{\frac12}\omega}_{\L^p(\Omega^\bullet(M))}
\ov{\eqref{eq:equiv-D-Delta-half-proof}}{\lesssim} \norm{D\omega}_{\L^p(\Omega^\bullet(M))}.\nonumber
\end{align}
We deduce that
\begin{align*}
\norm{\omega}_{\scr{L}^{1,p}(\Omega^\bullet(M))}
&\ov{\eqref{eq:def-W1p-D-norm-again}}{=} \norm{\omega}_{\L^p(\Omega^\bullet(M))}
+\norm{\d\omega}_{\L^p(\Omega^\bullet(M))}
+\norm{\d^*\omega}_{\L^p(\Omega^\bullet(M))}\\
&\ov{\eqref{eq:d-dstar-estimate-Riesz}}{\lesssim}
\norm{\omega}_{\L^p(\Omega^\bullet(M))}+\norm{D\omega}_{\L^p(\Omega^\bullet(M))}.
\end{align*}
So we have an equivalence $\norm{\omega}_{\scr{L}^{1,p}(\Omega^\bullet(M))}\approx \norm{\omega}_{\L^p(\Omega^\bullet(M))}+\norm{D\omega}_{\L^p(\Omega^\bullet(M))}$ for $\omega \in \Omega^\bullet(M)$. Let $\norm{\cdot}_{D,p}$ denote the graph norm of $D$ on $\Omega^\bullet(M)$, that is
\[
\norm{\omega}_{D,p}
\ov{\mathrm{def}}{=}
\norm{\omega}_{\L^p(\Omega^\bullet(M))}
+\norm{D\omega}_{\L^p(\Omega^\bullet(M))},
\quad
\omega \in \Omega^\bullet(M).
\]
By the previous equivalence, the graph norm $\norm{\cdot}_{D,p}$ and the norm $\norm{\cdot}_{\scr{L}^{1,p}(\Omega^\bullet(M))}$ are equivalent on the linear subspace $\Omega^\bullet(M)\subset\L^p(\Omega^\bullet(M))$. Therefore they induce the same Cauchy sequences and the same notion of completion. More precisely, we obtain
\begin{equation}
\label{eq:domain-D-equal-closure-W1p}
\dom D_p
=
\ovl{\Omega^\bullet(M)}^{\norm{\cdot}_{D,p}}
=
\ovl{\Omega^\bullet(M)}^{\norm{\cdot}_{\scr{L}^{1,p}(\Omega^\bullet(M))}}.
\end{equation}
By definition of the space $\scr{L}^{1,p}(\Omega^\bullet(M))$, the last completion in \eqref{eq:domain-D-equal-closure-W1p} is precisely $\scr{L}^{1,p}(\Omega^\bullet(M))$. We obtain $
\dom D_p
=\scr{L}^{1,p}(\Omega^\bullet(M))$. Finally, since equivalent norms on a normed space induce equivalent norms on their completions, we have the equivalence \eqref{eq:graph-norm-D-W1p-equivalence}.
\end{proof}

\begin{remark} \normalfont
In \cite[Theorem 3 p.~74]{CoL94}, Coulhon and Ledoux constructed a complete non-compact
Riemannian manifold $M$ on which the Riesz transform $\d(-\Delta)^{-\frac{1}{2}}$ is not bounded on the Banach space $\L^p(M)$ for  all sufficiently large $p$ (or sufficiently small). Later, Coulhon and Duong proved in  \cite[Section 5 p.~1164-1166]{CoD99} that, on the connected sum $M=\R^n\#\R^n$ with $n \geq 2$, the Riesz transform $\d(-\Delta)^{-\frac{1}{2}}$ is not bounded on $\L^p(M)$ for any $p > n$.  
Furthermore, in \cite[Proposition 6.2 p.~1747]{CoD03} they exhibited a complete non-compact Riemannian manifold $M$ with bounded geometry and polynomial volume growth and some $p_0 \in (1,2)$ such that the reverse Riesz estimate $\bnorm{\Delta^{\frac{1}{2}}f}_{\L^p(M)} \lesssim_p \norm{\d f}_{\L^p(\Omega^1(M)}$, where $f \in \C_0^\infty(M)$, is false for any $p \in (1,p_0)$. We refer also to \cite{CCH06} and \cite{Loh94} for related results.
\end{remark}

If $M$ is a smooth compact Riemannian manifold, by \cite[Theorem 6.1 p.~68]{ISS99} (see also \cite[Corollary 4.12 p.~2085]{Sco95} for a related result), we have the equivalence
\begin{equation}
\label{equivalences-Sobolev-W1}
\norm{\omega}_{\W^{1,p}(\Omega^\bullet(M))}
\approx \norm{\omega}_{\scr{L}^{1,p}(\Omega^\bullet(M))},
\end{equation}
where $\W^{1,p}(\Omega^\bullet(M))$ is the Sobolev space defined in the usual way by means of a finite atlas on $M$. More precisely, since the manifold $M$ is compact, there exists a finite atlas $\mathfrak{A}=\{(U_1,\varphi_1),(U_2,\varphi_2),\ldots,(U_m,\varphi_m)\}$ such that each chart $\varphi_i \co U_i \to V_i$ is a diffeomorphism onto an open subset $V_i$ of $\R^d$. Let $(\chi_1,\ldots,\chi_m)$ be a smooth partition of unity subordinate to this atlas. For any differential form $\omega \in \Omega^\bullet(M)$, set $\omega_i \ov{\mathrm{def}}{=} \chi_i \omega$ and consider the pull-back $(\varphi_i^{-1})^* \omega_i$ on $V_i$. This differential form can be written as
\[
(\varphi_i^{-1})^* \omega_i(x)
=
\sum_{k=0}^d \sum_{|I|=k} f_{i,I}(x) \d x_I,
\quad x \in V_i,
\]
where the second sum runs over all multi-indices $I$ of length $k$. We then define
\begin{equation}
\label{Sobolev-charts}
\norm{\omega}_{\W_{\mathfrak{A}}^{1,p}(\Omega^\bullet(M))}
\ov{\mathrm{def}}{=}
\sum_{i=1}^m \sum_{k=0}^d \sum_{|I|=k} \norm{f_{i,I}}_{\W^{1,p}(V_i)}.
\end{equation}
As indicated in \cite[p.~48]{ISS99} (without proof), the Sobolev spaces corresponding to different atlases and subordinate partitions of unity are equivalent. We therefore simply write $\W^{1,p}(\Omega^\bullet(M))$.

\begin{remark} \normalfont
If $M$ is a manifold of bounded geometry, then \cite[Proposition 4.1 p.~54]{Shu92} seems to say that $\dom D_p=\W_p^{1}(\Omega^\bullet(M))$ for some suitable Sobolev space $\W_p^{1}(\Omega^\bullet(M))$.
\end{remark}

Consider some integer $k \in \{0,\ldots,d-1\}$. Following \cite[Definition 3.1 p.~47]{ISS99}, we say that a differential form $\omega \in \L^1(\Omega^k(M))$ admits a distributional exterior derivative if there exists a locally integrable differential form $\eta \in \L^1(\Omega^{k+1}(M))$ 
such that 
$$
\la \eta, \varphi\ra_{\L^1(\Omega^{k+1}(M)),\L^\infty(\Omega^{k+1}(M))}  
= \la \omega,\d^*\varphi \ra_{\L^1(\Omega^{k}(M)),\L^\infty(\Omega^{k}(M))}
$$
for any differential form $\varphi \in \Omega^{k+1}(M)$. In this case, $\eta$ is uniquely determined and we let $
\d \omega
\ov{\mathrm{def}}{=} \eta$ and we say that $\eta$ is the distributional exterior derivative of $\omega$. We have
\begin{equation}
\label{eq:def-distributional-d}
\la \d \omega, \varphi\ra_{\L^1(\Omega^{k+1}(M)),\L^\infty(\Omega^{k+1}(M))} 
= \la \omega,\d^*\varphi \ra_{\L^1(\Omega^{k}(M)),\L^\infty(\Omega^{k}(M))}.
\end{equation}
The notion of distributional exterior coderivative is defined analogously.
Now, we describe the domains of the closures of the differential operators $\d$ and $\d^*$ in terms of distributional derivatives.

\begin{cor}
\label{cor-domain-d-dstar-Dp}
Let $(M,g)$ be a smooth compact Riemannian manifold. Suppose that $1 < p < \infty$. Then the operators $
\d,\d^* \co \Omega^\bullet(M) \subset \L^p(\Omega^\bullet(M)) \to \L^p(\Omega^\bullet(M))$ are closable. Moreover, their closures $\d_p$ and $(\d^*)_p$ satisfy
\[
\dom \d_p
=
\{\omega \in \L^p(\Omega^\bullet(M)) : \d\omega \text{ exists in the distributional sense and belongs to } \L^p(\Omega^\bullet(M))\},
\]
\[
\dom (\d^*)_p
=
\{\omega \in \L^p(\Omega^\bullet(M)) : \d^*\omega \text{ exists in the distributional sense and belongs to } \L^p(\Omega^\bullet(M))\}
\]
and
\begin{equation}
\label{last-of-the-last}
\scr{L}^{1,p}(\Omega^\bullet(M))
=
\dom \d_p \cap \dom (\d^*)_p.
\end{equation}
Moreover, under these identifications, the closed operators $\d_p$ and $(\d^*)_p$ act as the corresponding distributional differential operators. Finally, we have 
\begin{equation}
\label{Dp-decompo-en-deux}
D_p
=\d_p+(\d^*)_p,
\end{equation}
as unbounded operators on $\L^p(\Omega^{\bullet}(M))$ with common domain $\dom D_p=\dom \d_{p} \cap \dom (\d^*)_p$.
\end{cor}

\begin{proof}
By \cite[Theorem 5.28 p.~168]{Kat76}, the operators $
\d,\d^* \co \Omega^\bullet(M) \subset \L^p(\Omega^\bullet(M)) \to \L^p(\Omega^\bullet(M))$ are closable. We denote by $\d_p$ and $(\d^*)_p$ the closures of $\d$ and $\d^*$. By definition of the closure, if $\omega \in \dom \d_p$ then there exists a sequence $(\omega_n)$ in $\Omega^\bullet(M)$ such that
\[
\omega_n \to \omega
\quad \text{and} \quad
\d\omega_n \to \d_p\omega
\]
in $\L^p(\Omega^\bullet(M))$. For any differential form $\varphi$, passing to the limit in the distributional identity
\[
\la \d\omega_n,\varphi \ra_{\L^2(\Omega^\bullet(M))}
\ov{\eqref{def-codiff}}{=} \la \omega_n,\d^*\varphi \ra_{\L^2(\Omega^\bullet(M))},
\]
we obtain $\la \d_p\omega,\varphi \ra_{\L^1(\Omega^\bullet(M)),\L^\infty(\Omega^\bullet(M))}
= \la \omega,\d^*\varphi \ra_{\L^1(\Omega^\bullet(M)),\L^\infty(\Omega^\bullet(M))}$. This proves that $\d\omega$ exists in the distributional sense and equals $\d_p\omega$. Thus
\[
\dom \d_p
\subset
\{\omega \in \L^p(\Omega^\bullet(M)) : \d\omega \text{ exists distributionally and belongs to } \L^p(\Omega^\bullet(M))\}.
\]
Conversely, suppose that $\omega \in \L^p(\Omega^\bullet(M))$ and that $\d\omega$ exists in the distributional sense and belongs to the space $\L^p(\Omega^\bullet(M))$. Since the manifold $M$ is compact, by the density \cite[Corollary 3.6 p.~48]{ISS99} or \cite[Theorem I.19 p.~14]{Gun17} of the subspace $\Omega^\bullet(M)$ there exists a sequence $(\omega_n)$ of elements in $\Omega^\bullet(M)$ such that $\omega_n \to \omega$ and $\d\omega_n \to \d\omega$ in $\L^p(\Omega^\bullet(M))$. So $\omega \in \dom \d_p$ and $\d_p\omega=\d\omega$. Hence $\dom \d_p$ is exactly the distributional domain of $\d$. We can use the same argument for $(\d^*)_p$.

Now, we identify the intersection of these two domains with $\scr{L}^{1,p}(\Omega^\bullet(M))$. Consider some $
u \in \dom \d_p \cap \dom (\d^*)_p$. By the preceding domain identifications, both $\d u$ and
$\d^*u$ exist in the distributional sense and belong to $\L^p(\Omega^{\bullet}(M))$. Now, we use \cite[Theorem I.19 p.~14]{Gun17} with the finite family $\{\d,\d^*\}$ of first-order differential operators. So, there exists a sequence $(u_n)$ in $\Omega^{\bullet}(M)$ such that $u_n \to u$, $\d u_n \to \d u$, $\d^* u_n \to \d^* u$ in $\L^p(\Omega^{\bullet}(M))$. Therefore $u \in \scr{L}^{1,p}(\Omega^\bullet(M))$. The converse inclusion is obvious from the definition of $\scr{L}^{1,p}(\Omega^\bullet(M))$ and the closedness of the operators $\d_p$ and $(\d^*)_p$. Thus $
\scr{L}^{1,p}(\Omega^\bullet(M))
=
\dom \d_p \cap \dom (\d^*)_p$.

Now, we prove the last assertion. By Proposition~\ref{prop:graph-norm-Hodge-Dirac-Lp}, we deduce the equalities $
\dom D_{p}
\ov{\eqref{eq:domain-D-graph-precise}}{=}
\scr{L}^{1,p}(\Omega^\bullet(M))
\ov{\eqref{last-of-the-last}}{=} \dom \d_{p}
\cap
\dom (\d^*)_{p}$. Consider some $u \in \dom D_{p}$. Since $
\dom D_{p}
\ov{\eqref{eq:domain-D-graph-precise}}{=} \scr{L}^{1,p}(\Omega^\bullet(M))$, there exists a sequence $(u_n)$ in $\Omega^{\bullet}(M)$ such that $
u_n \to u$, $
\d u_n \to \d_{p}u$, 
$\d^* u_n \to (\d^*)_{p}u$
in the space $\L^p(\Omega^{\bullet}(M))$. We deduce that
\[
D u_n
\ov{\eqref{def-Hodge-Dirac}}{=} \d u_n+\d^*u_n
\to \d_{p}u+(\d^*)_{p}u
\]
in $\L^p(\Omega^{\bullet}(M))$. Since $D_{p}$ is the closure of the unbounded operator $D$, we conclude that $
D_{p}u=\d_{p}u+(\d^*)_{p}u$. The equality of operators is a consequence of the equality of domains.
\end{proof}

Since the norm \eqref{eq:def-W1p-D-norm-again} dominates the $\L^p$-norm, the identity map on $\Omega^{\bullet}(M)$ extends to a continuous map $\scr{L}^{1,p}(\Omega^\bullet(M)) \hookrightarrow \L^p(\Omega^{\bullet}(M))$. Using the closed operators $\d_p$ and $(\d^*)_p$, it is elementary to prove that it is injective. Hence, we can identify $\scr{L}^{1,p}(\Omega^\bullet(M))$ with its image in $\L^p(\Omega^{\bullet}(M))$. Under this identification, the space coincides with a space studied in \cite[Chapter I, Section 3]{Gun17}.

\begin{remark} \normalfont
The spaces $\dom \d_p$ and $\dom (\d^*)_p$ are considered in \cite[pp.~49-50]{ISS99}.
\end{remark}

\subsection{Commutators of the Hodge--Dirac operators with multiplication operators}

In this section, we describe the commutators of the Hodge--Dirac operator $D_p$ with multiplication operators. We will use the following observation proved in \cite{Arh26a}.

\begin{lemma}
\label{lem:matrix-multiplication-Lp}
Let $(\Omega,\mu)$ be a measure space and let $1 \leq p < \infty$. Let $A \co \Omega \to \M_N$ be a strongly measurable function such that $
\esssup_{x \in \Omega} \norm{A(x)}_{\M_N}
< \infty$. 
Define the multiplication operator
\[
M_A \co \L^p(\Omega,\mathbb{C}^N) \to \L^p(\Omega,\mathbb{C}^N), 
\qquad
(M_A u)(x) \ov{\mathrm{def}}{=} A(x)u(x), \quad u \in \L^p(\Omega,\mathbb{C}^N).
\]
Then $M_A$ is bounded and its operator norm is given by
\begin{equation}
\label{eq:norm-matrix-mult}
\norm{M_A}_{\L^p(\Omega,\mathbb{C}^N)\to \L^p(\Omega,\mathbb{C}^N)}
=
\esssup_{x \in \Omega} \norm{A(x)}_{\M_N}.
\end{equation}
\end{lemma}

Note that the formula \eqref{eq:commutator-D-Mf} is folklore but we are unable to locate a convincing reference. So, we give a complete proof.

\begin{prop}
\label{prop:commutator-Hodge-Dirac}
Let $(M,g)$ be a smooth compact Riemannian manifold of dimension $d$. For any function $f \in \C^\infty(M)$ and any differential form $\omega \in \Omega^\bullet(M)$, we have
\begin{equation}
\label{eq:commutator-D-Mf}
[D,M_f](\omega)
=\d f \wedge \omega - i_{\nabla f}\omega.
\end{equation}
Suppose that $1 < p <\infty$. If $f \in \C^\infty(M)$ then $M_f(\dom D_p) \subset \dom D_p$ and the commutator $[D_p,M_f]$ extends to a bounded operator on the Banach space $\L^p(\Omega^\bullet(M))$ with
\begin{equation}
\label{eq:commutator-complex-bound}
\norm{[D_p,M_f]}_{\L^p(\Omega^\bullet(M))) \to \L^p(\Omega^\bullet(M)))}
\leq
\sqrt{2}\norm{\d f}_{\L^\infty(\Omega^1(M))}.
\end{equation}
If in addition $f$ is real-valued, then
\begin{equation}
\label{eq:commutator-real-exact-norm}
\norm{[D_p,M_f]}_{\L^p(\Omega^\bullet(M))) \to \L^p(\Omega^\bullet(M)))}
=
\norm{\d f}_{\L^\infty(\Omega^1(M))}.
\end{equation}
\end{prop}

\begin{proof}
Let $f \in \C^\infty(M)$. We split the commutator into two parts:
\[
[D,M_f]
\ov{\eqref{def-Hodge-Dirac}}{=} [\d +\d^*,M_f]
=[\d,M_f]+[\d^*,M_f].
\]
For any differential form $\omega \in \Omega^{\bullet}(M)$, we have
\begin{align}
\label{inter-45}
[\d,M_f](\omega)
&=\d(f\omega)-f\d\omega
\ov{\eqref{Leibniz-manifold-d}}{=} \d f \wedge \omega.
\end{align}
Now, we compute $[\d^*,M_f]$. Let $\omega,\eta \in \Omega^\bullet(M)$. Since the $\L^2$ inner product is linear in the first variable and conjugate-linear in the second one, we have
\[
\la \d^*(f\omega),\eta \ra_{\L^2(\Omega^\bullet(M))}
\ov{\eqref{def-codiff}}{=}
\la f\omega, \d\eta \ra_{\L^2(\Omega^\bullet(M))}
=
\la \omega,\ovl{f} \d\eta \ra_{\L^2(\Omega^\bullet(M))}.
\]
By the Leibniz rule, we have
\[
\d(\ovl{f}\eta)
\ov{\eqref{Leibniz-manifold-d}}{=}
\d\ovl{f}\wedge\eta+\ovl{f} \d\eta,
\]
Hence $
\overline{f}\d\eta
=\d(\ovl{f}\eta)-\d\ovl{f}\wedge\eta$. Consequently, we obtain
\begin{align*}
\MoveEqLeft
\la \d^*(f\omega),\eta \ra_{\L^2(\Omega^\bullet(M))}
=
\big\la \omega,\d(\ovl{f}\eta) \big\ra_{\L^2(\Omega^\bullet(M))}
-
\big\la \omega,\d\ovl{f}\wedge\eta \big\ra_{\L^2(\Omega^\bullet(M))} \\
&\ov{\eqref{def-codiff} \eqref{eq:wedge-interior-adjoint}}{=}
\big\la \d^*\omega,\ovl{f}\eta \big\ra_{\L^2(\Omega^\bullet(M))}
-
\big\la i_{(\ovl{\d\ovl{f}})^\sharp}\omega,\eta \big\ra_{\L^2(\Omega^\bullet(M))} \\
&=
\la f \d^*\omega,\eta \ra_{\L^2(\Omega^\bullet(M))}
-
\big\la i_{(\d f)^\sharp}\omega,\eta \big\ra_{\L^2(\Omega^\bullet(M))}
\ov{\eqref{sharp}}{=} \la f \d^*\omega-i_{\nabla f}\omega,\eta \ra_{\L^2(\Omega^\bullet(M))}.
\end{align*}
Since this holds for all $\eta \in \Omega^{\bullet}(M)$, 
we deduce that
\[
[\d^*,M_f](\omega)
= -i_{\nabla f}\omega.
\]
Combining this identity with \eqref{inter-45}, we obtain
\begin{equation}
\label{commutator-proof}
[D,M_f](\omega)
=\d f \wedge \omega - i_{\nabla f}\omega.
\end{equation}
This is the desired formula \eqref{eq:commutator-D-Mf}.

Now, we prove the second part of the proposition. Fix $x \in M$. Using the musical isomorphisms, we can identify $\Lambda^k \mathrm{T}_x^*M$ with the exterior power $\Lambda^k \mathrm{T}_xM$, and we equip $\Lambda^k \mathrm{T}_x^*M$ with the inner product induced by $g_x$.

Following \cite[Section 3.4, (3.63), (3.67)]{VaR19}, for any $v \in \mathrm{T}_xM$ we introduce the left exterior multiplication
$E(v) \co \Lambda \mathrm{T}_xM \to \Lambda \mathrm{T}_xM$ by $v$ defined by
\[
E(v)(A)
\ov{\mathrm{def}}{=} v \wedge A,
\]
and for any covector $\beta \in \mathrm{T}_x^*M$ we consider the left contraction
$I(\beta) \co \Lambda \mathrm{T}_xM \to \Lambda \mathrm{T}_xM $ by $\beta$ defined by
\[
I(\beta)(A)
\ov{\mathrm{def}}{=} \beta \lrcorner A.
\]
If we denote by $\mathrm{T}_xM \to \mathrm{T}_x^*M$, $v \mapsto v^\flat$ the map defined by $v^\flat(w)=g_x(v,w)$, then, under the previous identification, $I(v^\flat)$ corresponds precisely to the interior product $i_v$ acting on differential forms.

In \cite[(3.76)–(3.77) p.~75]{VaR19}, the authors introduce maps $\gamma_+ ,\gamma_- \co \mathrm{T}_x M \to \mathrm{End}(\Lambda \mathrm{T}_xM)$ defined by
\begin{equation}
\label{eq:def-Clifford-mappings}
\gamma_+(v)
\ov{\mathrm{def}}{=} E(v) + I(v^\flat),
\qquad
\gamma_-(v)
\ov{\mathrm{def}}{=} E(v) - I(v^\flat),
\quad v \in \mathrm{T}_xM,
\end{equation}
and prove in \cite[Theorem 3.6 p.~76]{VaR19} that for all $u,v \in \mathrm{T}_xM$ these operators satisfy the Clifford relations
\begin{equation}
\label{eq:Clifford-relations-Gamma}
\gamma_+(v)\gamma_+(u) + \gamma_+(u)\gamma_+(v)
=2 g_x(v,u)\Id,
\qquad
\gamma_-(v)\gamma_-(u) + \gamma_-(u)\gamma_-(v)
=-2 g_x(v,u)\Id.
\end{equation}
In particular, taking $u=v$ and dividing by $2$, we obtain
\begin{equation}
\label{eq:Gamma-square}
\gamma_+(v)^2
=g_x(v,v)\Id,
\qquad
\gamma_-(v)^2
=-g_x(v,v)\Id,
\quad v \in \mathrm{T}_xM.
\end{equation}

For any covector $\alpha \in \mathrm{T}_{\mathbb C,x}^*M$, we define $
C_\alpha\co\Lambda_{\mathbb C}\mathrm{T}_x^*M
\to
\Lambda_{\mathbb C}\mathrm{T}_x^*M$, $
C_\alpha\omega
\ov{\mathrm{def}}{=}
\alpha\wedge\omega-i_{\alpha^\sharp}\omega$. Assume now that $\alpha \in \mathrm{T}_x^*M$ is real and set $v \ov{\mathrm{def}}{=}\alpha^\sharp\in\mathrm{T}_xM$. Then, under the previous identification, the restriction of $C_\alpha$ to the real exterior algebra corresponds exactly to the operator $\gamma_-(v)$. Hence, by complexification of \eqref{eq:Gamma-square}, we have
\begin{equation}
\label{eq:Calpha-square}
C_\alpha^2
=
-|\alpha|_{\Lambda^1\mathrm{T}_x^*M}^2
\Id_{\Lambda_{\mathbb C}\mathrm{T}_x^*M}.
\end{equation}
On the other hand, the adjointness relation between wedge and interior multiplication implies that, for any $\eta,\omega\in\Lambda_{\mathbb C}\mathrm{T}_x^*M$,
\begin{align*}
\la C_\alpha \eta,\omega \ra_x
&=\la \alpha \wedge \eta,\omega \ra_x - \la i_{\alpha^\sharp}\eta,\omega \ra_x
= \la \eta,i_{\alpha^\sharp}\omega \ra_x - \la \eta,\alpha \wedge \omega \ra_x \\
&=-\big\la \eta,(\alpha \wedge \omega - i_{\alpha^\sharp}\omega) \big\ra_x
=-\la \eta,C_\alpha \omega \ra_x.
\end{align*}
Thus $C_\alpha$ is skew-adjoint on the finite-dimensional Hilbert space $\Lambda_{\mathbb{C}} \mathrm{T}_x^*M$, i.e.,
\begin{equation}
\label{eq:Calpha-skew-adjoint}
C_\alpha^*
=-C_\alpha.
\end{equation}
Combining \eqref{eq:Calpha-square} and \eqref{eq:Calpha-skew-adjoint}, we obtain
\[
C_\alpha^* C_\alpha
\ov{\eqref{eq:Calpha-skew-adjoint}}{=} -C_\alpha^2
\ov{\eqref{eq:Calpha-square}}{=} |\alpha|_{\Lambda_{\mathbb{C}}^1 \mathrm{T}_x^*M}^2 \Id_{\Lambda_{\mathbb{C}} \mathrm{T}_x^*M}.
\]
Hence, we infer that
\begin{equation}
\label{eq:Calpha-norm}
\norm{C_\alpha}_{\Lambda_{\mathbb{C}} \mathrm{T}_x^*M \to \Lambda_{\mathbb{C}} \mathrm{T}_x^*M}
=|\alpha|_{\Lambda^1 \mathrm{T}_x^*M}.
\end{equation}

Now, consider some $\alpha \in \mathrm{T}_{\mathbb{C},x}^*M$ and write uniquely $\alpha =a+\i b$ for some $a,b \in \mathrm{T}_x^*M$. By complex linearity, we have $
C_\alpha
=C_a+\i C_b$. Hence
\begin{equation}
\label{eq:Calpha-norm-bis}
\norm{C_\alpha}_{\Lambda_{\mathbb{C}} \mathrm{T}_x^*M \to \Lambda_{\mathbb{C}} \mathrm{T}_x^*M}
\leq
\norm{C_a}+\norm{C_b}
\ov{\eqref{eq:Calpha-norm}}{=}
|a|+|b|
\leq
\sqrt{2}\bigl(|a|^2+|b|^2\bigr)^{\frac12}
=
\sqrt{2}|\alpha|.
\end{equation}

Now, we apply this fibrewise description to the commutator. For each $x \in M$ and any function $f \in \C^\infty(M)$, consider the element $\alpha_x \ov{\mathrm{def}}{=} \d f(x)$ of $\mathrm{T}_{\mathbb{C},x}^*M$. By \eqref{eq:commutator-D-Mf} and the previous identification, for any $\omega \in \Omega^\bullet(M)$ we have
\begin{equation}
\label{inter-A-525}
[D,M_f](\omega)(x)
\ov{\eqref{eq:commutator-D-Mf}}{=} \d f(x) \wedge \omega(x) - i_{(\nabla f(x))}\omega(x)
=C_{\alpha_x}(\omega(x)).
\end{equation}

Using \eqref{eq:Calpha-norm-bis} with $\alpha=\alpha_x$, we obtain
\begin{align}
\MoveEqLeft
\label{ref-88UY}
|[D,M_f](\omega)(x)|_{\Lambda_{\mathbb C}\mathrm{T}_x^*M}
\ov{\eqref{inter-A-525}}{=}
|C_{\alpha_x}(\omega(x))|_{\Lambda_{\mathbb C}\mathrm{T}_x^*M} 
\ov{\eqref{eq:Calpha-norm-bis}}{\leq}
\sqrt{2}|\alpha_x|_{\Lambda_{\mathbb C}^1\mathrm{T}_x^*M}
|\omega(x)|_{\Lambda_{\mathbb C}\mathrm{T}_x^*M} \\
&=
\sqrt{2}|\d f(x)|_{\Lambda_{\mathbb C}^1\mathrm{T}_x^*M}
|\omega(x)|_{\Lambda_{\mathbb C}\mathrm{T}_x^*M} \nonumber
\leq
\sqrt{2}\norm{\d f}_{\L^\infty(\Omega^1(M))}
|\omega(x)|_{\Lambda_{\mathbb C}\mathrm{T}_x^*M}.
\nonumber
\end{align}
Taking the $\L^p$-norm, we deduce that
\begin{align*}
\MoveEqLeft
\norm{[D,M_f](\omega)}_{\L^p(\Omega^\bullet(M))}
\ov{\eqref{norm-Lp-Df}}{=}
\bigg(\int_M |[D,M_f](\omega)(x)|_{\Lambda_{\mathbb{C}} \mathrm{T}_x^*M}^{p} \d \mu_g(x)\bigg)^{\frac{1}{p}} \\
&\ov{\eqref{ref-88UY}}{\leq} \sqrt{2}\norm{\d f}_{\L^\infty(\Omega^1(M))}
\bigg(\int_M |\omega(x)|_{\Lambda_{\mathbb{C}} \mathrm{T}_x^*M}^{p} \d \mu_g(x)\bigg)^{\frac{1}{p}} 
\ov{\eqref{norm-Lp-Df}}{=}
\sqrt{2}\norm{\d f}_{\L^\infty(\Omega^1(M))} \norm{\omega}_{\L^p(\Omega^\bullet(M))}.
\end{align*}
Hence $[D,M_f]$ extends by density to a bounded operator on $\L^p(\Omega^\bullet(M))$ and satisfies
\[
\norm{[D,M_f]}_{\L^p(\Omega^\bullet(M)) \to \L^p(\Omega^\bullet(M))}
\leq \sqrt{2}\norm{\d f}_{\L^\infty(\Omega^1(M))}.
\]
Now, we assume that the function $f$ is real-valued. Then $\d f(x)$ is a real covector for any $x \in M$. Hence, by \eqref{eq:Calpha-norm} and \eqref{inter-A-525}, we have
\[
|[D,M_f](\omega)(x)|_{\Lambda_{\mathbb C}\mathrm{T}_x^*M}
\leq
\norm{\d f}_{\L^\infty(\Omega^1(M))}
|\omega(x)|_{\Lambda_{\mathbb C}\mathrm{T}_x^*M}.
\]
Similarly to the complex case, we obtain
\begin{equation}
\label{eq:global-upper-bound-real}
\norm{[D,M_f]}_{\L^p(\Omega^\bullet(M)) \to \L^p(\Omega^\bullet(M))}
\leq
\norm{\d f}_{\L^\infty(\Omega^1(M))}.
\end{equation}

Fix a local orthonormal frame of the bundle $\Lambda_{\mathbb{C}}\mathrm{T}^* M$ on a coordinate chart $U \subset M$. By \cite[Remark I.17]{Gun17}, this yields an isometric identification of the Banach space $\L^p(\Omega^\bullet(U))$ with the Bochner space $\L^p(U,\mathbb{C}^N)$, where $N \ov{\mathrm{def}}{=} \dim_{\mathbb{C}} \Lambda_{\mathbb{C}} \mathrm{T}_x^*M=2^d$ does not depend on $x$. In this trivialization, the family of endomorphisms $C_{\alpha_x}$ is represented by a strongly measurable matrix-valued function $A \co U \to \M_N$, $x \mapsto C_{\alpha_x}$, and the operator $[D,M_f]$ restricted to $\L^p(\Omega^\bullet(U))$ corresponds exactly to the multiplication operator
\[
M_A \co \L^p(U,\mathbb{C}^N) \to \L^p(U,\mathbb{C}^N),
\qquad
(M_A u)(x) \ov{\mathrm{def}}{=} A(x)u(x).
\]
For any $x \in U$, we have
\[
\norm{A(x)}_{\M_N}
=\norm{C_{\alpha_x}}_{\Lambda_{\mathbb{C}} \mathrm{T}_x^*M \to \Lambda_{\mathbb{C}} \mathrm{T}_x^*M}
\ov{\eqref{eq:Calpha-norm}}{=} |\d f(x)|_{\Lambda^1 \mathrm{T}_x^*M}.
\]
Hence
\[
\esssup_{x \in U} \norm{A(x)}_{\M_N}
=\norm{\d f}_{\L^\infty(\Omega^1(U))}.
\]
Applying Lemma~\ref{lem:matrix-multiplication-Lp} with $\Omega=U$ and this matrix field $A$, we obtain
\begin{equation}
\label{eq:norm-local-commutator}
\norm{[D,M_f]}_{\L^p(\Omega^\bullet(U)) \to \L^p(\Omega^\bullet(U))}
= \norm{M_A}_{\L^p(U,\mathbb{C}^N) \to \L^p(U,\mathbb{C}^N)}
\ov{\eqref{eq:norm-matrix-mult}}{=} \norm{\d f}_{\L^\infty(\Omega^1(U))}.
\end{equation}
We denote by $
T \co \L^p(\Omega^\bullet(U)) \to \L^p(\Omega^\bullet(U))$ the local version of the commutator $[D,M_f]$ acting on $\L^p(\Omega^\bullet(U))$. Then \eqref{eq:norm-local-commutator} tells us that
\begin{equation}
\label{eq:norm-Tj}
\norm{T}_{\L^p(\Omega^\bullet(U)) \to \L^p(\Omega^\bullet(U))}
=\norm{\d f}_{\L^\infty(\Omega^1(U))}.
\end{equation}
Next, consider the isometric embedding
\[
J \co \L^p(\Omega^\bullet(U)) \to \L^p(\Omega^\bullet(M)),
\qquad
(J\omega)(x) \ov{\mathrm{def}}{=}
\begin{cases}
\omega(x)&\text{if }x \in U,\\
0&\text{if }x \in M\setminus U.
\end{cases}
\]
Since the commutator $[D,M_f]$ is given pointwise by the formula \eqref{eq:commutator-D-Mf}, it follows that for any $\omega \in \Omega^\bullet(M)$ supported in $U$ we have
\[
[D,M_f](\omega)(x)
=
\begin{cases}
T(\omega|_{U})(x)&\text{if }x \in U,\\
0&\text{if }x \in M \setminus U
\end{cases}.
\]
In particular, for every $\eta \in \Omega^\bullet(M)$ supported in $U$ we obtain
\[
\norm{[D,M_f](J\eta)}_{\L^p(\Omega^\bullet(M))}
=\norm{T\eta}_{\L^p(\Omega^\bullet(U))},
\qquad
\norm{J\eta}_{\L^p(\Omega^\bullet(M))}
=\norm{\eta}_{\L^p(\Omega^\bullet(U))}.
\]
Hence
\begin{align*}
\MoveEqLeft
\norm{[D,M_f]}_{\L^p(\Omega^\bullet(M)) \to \L^p(\Omega^\bullet(M))}
\geq \sup_{\eta \neq 0} \frac{\norm{[D,M_f](J\eta)}_{\L^p(\Omega^\bullet(M))}}{\norm{J\eta}_{\L^p(\Omega^\bullet(M))}}\\
&= \sup_{\eta \neq 0} \frac{\norm{T\eta}_{\L^p(\Omega^\bullet(U))}}{\norm{\eta}_{\L^p(\Omega^\bullet(U))}}
=\norm{T}_{\L^p(\Omega^\bullet(U)) \to \L^p(\Omega^\bullet(U))}
\ov{\eqref{eq:norm-Tj}}{=} \norm{\d f}_{\L^\infty(\Omega^1(U))}.
\end{align*}
Observing that $\norm{\d f}_{\L^\infty(\Omega^1(M))} =\sup_U \norm{\d f}_{\L^\infty(\Omega^1(U))}$, we obtain the lower bound
\begin{equation}
\label{eq:global-lower-bound}
\norm{[D,M_f]}_{\L^p(\Omega^\bullet(M)) \to \L^p(\Omega^\bullet(M))}
\geq \norm{\d f}_{\L^\infty(\Omega^1(M))}.
\end{equation}

It remains to prove that the multiplication operator $M_f$ preserves the domain of $D_p$. Let $u\in\dom D_p$. Since $D_p$ is the closure of $D$, there exists a sequence $(u_n)$ in $\Omega^\bullet(M)$ such that $
u_n\to u$ and $Du_n\to D_pu$ in $\L^p(\Omega^\bullet(M))$. Let $B_f \co \L^p(\Omega^\bullet(M)) \to \L^p(\Omega^\bullet(M))$ denote the bounded extension of the commutator $[D,M_f]$ constructed above. Since $f$ is smooth on the compact manifold $M$, the multiplication operator $M_f$ is bounded on $\L^p(\Omega^\bullet(M))$. Hence $fu_n\to fu$ in $\L^p(\Omega^\bullet(M))$. Moreover, we have
\[
D(fu_n)
=
fDu_n+[D,M_f]u_n
=
fDu_n+B_fu_n
\to fD_pu+B_fu.
\]
in $\L^p(\Omega^\bullet(M))$. Since $D_p$ is the closure of the operator $D$, we conclude that $ fu\in\dom D_p$ and
\begin{equation}
\label{eq:closed-commutator-Dp-Mf}
D_p(fu)
=
fD_pu+B_fu,
\quad
u\in\dom D_p.
\end{equation}
Consequently, $
M_f(\dom D_p)\subset\dom D_p$ and
\[
[D_p,M_f]u
=
B_fu,
\quad
u \in \dom D_p.
\]
Thus the unbounded commutator $[D_p,M_f]$ defined on $\dom D_p$ extends to the bounded operator $B_f$ on the Banach space $\L^p(\Omega^\bullet(M))$.
\end{proof}

\subsection{Rellich-Kondrachov theorem for $\L^p$-spaces of differential forms}

We need a Rellich-Kondrachov theorem for spaces of differential forms and a quantitative form for estimating the dimension of the spectral triples. Recall that the approximation numbers \cite[Definition I.d.14 p.~69]{Kon86} of a bounded linear operator $T \co X \to Y$ are defined by 
\begin{equation}
\label{def-approximation-numbers}
a_n(T) 
\ov{\mathrm{def}}{=} \inf\{\norm{T-R} : R \co X \to Y,\,  \rank R < n \}, \quad n \geq 1.
\end{equation}
These ``singular numbers'' quantify how well the map can be approximated by finite-dimensional linear operators. We refer to \cite{Kon86}, \cite{Pie80}, \cite{Pie87} and \cite{Pie07} for more information on singular numbers and the associated ideals. We will use the following result \cite[Theorem 6.5 p.~295]{EdE18}. 

\begin{thm}
\label{Th-Edmunds}
Let $\Omega$ be a bounded open set in $\R^d$ such that $\partial \Omega$ is minimally smooth, where $d \geq 1$. Suppose that $1 \leq p <\infty$. Then the $n$th approximation number $a_n(J)$ of the (compact) embedding map $J \co \W^{1,p}(\Omega) \to \L^p(\Omega)$ satisfies
\begin{equation}
\label{estimations-approx-Sobolev}
a_n(J)
=O(\tfrac{1}{n^{ \frac{1}{d}}})
\end{equation}
as $n \to \infty$. 
\end{thm}

We will also use the following observation from \cite{Arh26a}.

\begin{lemma}
\label{lemma:approx-finite-direct-sum}
Suppose that $1 \leq p \leq \infty$. Consider two $\ell^p$-direct sums $
X \ov{\mathrm{def}}{=} \bigoplus_{k=1}^N X_k$ and $Y \ov{\mathrm{def}}{=} \bigoplus_{k=1}^N Y_k$ of Banach spaces and a block diagonal operator $T \ov{\mathrm{def}}{=} \bigoplus_{k=1}^N T_k \co X \to Y$, where each $T_k \co X_k \to Y_k$ is a bounded linear operator. Then for every integer $n \geq 1$ we have
\begin{equation}
\label{eq:an-direct-sum}
a_{Nn}(T)
\leq \max_{1 \leq k \leq N} a_n(T_k).
\end{equation}
\end{lemma}

We will use the next result.

\begin{prop}
\label{prop-equivalence}
Let $(M,g)$ be a smooth compact Riemannian manifold of dimension $d$. Let $(\widetilde U,\widetilde\varphi)$ be a coordinate chart and let $U$ be an open subset of $\widetilde U$ such that $\ovl U\subset \widetilde U$. Set $\varphi\ov{\mathrm{def}}{=}\widetilde\varphi|_U$. For each $k\in\{0,\dots,d\}$, consider on $\Lambda_{\mathbb C}^k\mathrm T^*U$ the Euclidean Hermitian norm $|\cdot|_{\mathrm{eucl}}$ induced by the coordinates and the Hermitian norm $|\cdot|_g$ induced by $g$. Then there exists a constant $C\geq1$ such that
\[
C^{-1}|\eta|_{\mathrm{eucl}}
\leq
|\eta|_g
\leq
C|\eta|_{\mathrm{eucl}},
\quad
x\in U,\quad \eta\in\Lambda_{\mathbb C}^k\mathrm T_x^*M.
\]
\end{prop}

\begin{proof}
The case $k=0$ is immediate. Suppose that $1\leq k\leq d$. For each $x\in U$, let $A_k(x)$ be the unique positive definite selfadjoint operator on $\Lambda_{\mathbb C}^k\mathrm T_x^*M$ such that
\[
\la \eta,\zeta\ra_g
=
\la A_k(x)\eta,\zeta\ra_{\mathrm{eucl}},
\quad
\eta,\zeta\in\Lambda_{\mathbb C}^k\mathrm T_x^*M.
\]
The map $x\mapsto A_k(x)$ is continuous on $\widetilde U$. Since $\ovl U$ is a compact subset of $\widetilde U$, there exist constants $m_k,M_k>0$ such that
\[
m_k\leq\lambda_{\min}(A_k(x))
\leq\lambda_{\max}(A_k(x))
\leq M_k,
\quad x\in U.
\]
Consequently,
\[
m_k|\eta|_{\mathrm{eucl}}^2
\leq
\la A_k(x)\eta,\eta\ra_{\mathrm{eucl}}
=
|\eta|_g^2
\leq
M_k|\eta|_{\mathrm{eucl}}^2.
\]
Taking square roots and setting $C\ov{\mathrm{def}}{=}\max\{m_k^{-\frac12},M_k^{\frac12}\}$ proves the result.
\end{proof}

We also need the following local description of the $\L^p$-norm.

\begin{lemma}
\label{lem:equiv-global-local-Lp}
Let $(M,g)$ be a smooth compact Riemannian manifold of dimension $d$. Let $\mathfrak A=\{(U_1,\varphi_1),\ldots,(U_m,\varphi_m)\}$ be a finite atlas and suppose that, for every $j$, there exists a coordinate chart $(\widetilde U_j,\widetilde\varphi_j)$ such that $\ovl{U_j}\subset\widetilde U_j$ and $\varphi_j=\widetilde\varphi_j|_{U_j}$. Set $V_j\ov{\mathrm{def}}{=}\varphi_j(U_j)$. Suppose that $1\leq p<\infty$. For $j\in\{1,\ldots,m\}$ and a family $(g_{j,I})_{k,I}$ with $g_{j,I}\in\L^p(V_j)$, set
\[
\eta_j
\ov{\mathrm{def}}{=}
\varphi_j^*
\left(
\sum_{k=0}^d\sum_{|I|=k}g_{j,I}\d x_I
\right).
\]
Then we have
\begin{equation}
\label{eq:local-coordinate-Lp-equivalence}
\norm{\eta_j}_{\L^p(\Omega^\bullet(U_j))}
\approx
\left(
\int_{V_j}
\left(
\sum_{k=0}^d\sum_{|I|=k}|g_{j,I}(x)|^2
\right)^{\frac p2}
\d x
\right)^{\frac1p}
\leq
\sum_{k=0}^d\sum_{|I|=k}
\norm{g_{j,I}}_{\L^p(V_j)}.
\end{equation}
Let $(\chi_1,\ldots,\chi_m)$ be a smooth partition of unity subordinate to $\mathfrak A$. For $\omega\in\Omega^\bullet(M)$, write
\[
(\varphi_j^{-1})^*(\chi_j\omega)
=
\sum_{k=0}^d\sum_{|I|=k}f_{j,I}\d x_I.
\]
Then
\begin{equation}
\label{eq:equiv-global-local-Lp-correct}
\norm{\omega}_{\L^p(\Omega^\bullet(M))}
\approx
\left(
\sum_{j=1}^m
\int_{V_j}
\left(
\sum_{k=0}^d\sum_{|I|=k}|f_{j,I}(x)|^2
\right)^{\frac p2}
\d x
\right)^{\frac1p}.
\end{equation}
\end{lemma}

\begin{proof}
Fix $j\in\{1,\ldots,m\}$ and $x\in V_j$, and set $y\ov{\mathrm{def}}{=}\varphi_j^{-1}(x)$. Proposition \ref{prop-equivalence}, applied simultaneously to the finitely many degrees, gives a constant $C_j\geq1$ such that
\[
C_j^{-1}
\left(
\sum_{k=0}^d\sum_{|I|=k}|g_{j,I}(x)|^2
\right)^{\frac12}
\leq
|\eta_j(y)|_{\Lambda_{\mathbb C}\mathrm T_y^*M}
\leq
C_j
\left(
\sum_{k=0}^d\sum_{|I|=k}|g_{j,I}(x)|^2
\right)^{\frac12}.
\]
In the coordinates $\varphi_j$, the Riemannian measure has the form $
(\varphi_j^{-1})^*\mu_g
=
\rho_j(x)\d x$, where $\rho_j$ is smooth and positive. Since $\ovl{U_j}\subset\widetilde U_j$, there exists $c_j\geq1$ such that
\[
c_j^{-1}\leq\rho_j(x)
\leq c_j,
\quad x\in V_j.
\]
A change of variables therefore yields the equivalence in \eqref{eq:local-coordinate-Lp-equivalence}. Its last inequality follows from
\[
\left(
\sum_{k=0}^d\sum_{|I|=k}|g_{j,I}(x)|^2
\right)^{\frac12}
\leq
\sum_{k=0}^d\sum_{|I|=k}|g_{j,I}(x)|
\]
and Minkowski's inequality. Apply the first part to the coefficients of $(\varphi_j^{-1})^*(\chi_j\omega)$. We obtain
\begin{equation}
\label{eq:chartwise-Lp-equivalence}
\norm{\chi_j\omega}_{\L^p(\Omega^\bullet(M))}^p
\approx
\int_{V_j}
\left(
\sum_{k=0}^d\sum_{|I|=k}|f_{j,I}(x)|^2
\right)^{\frac p2}
\d x.
\end{equation}
Since $\sum_{j=1}^m\chi_j=1$ and $0\leq\chi_j\leq1$, we have pointwise
\[
|\omega(y)|_{\Lambda_{\mathbb C}\mathrm T_y^*M}^p
\leq
m^{p-1}
\sum_{j=1}^m
|\chi_j(y)\omega(y)|_{\Lambda_{\mathbb C}\mathrm T_y^*M}^p.
\]
Hence
\[
\norm{\omega}_{\L^p(\Omega^\bullet(M))}^p
\leq
m^{p-1}
\sum_{j=1}^m
\norm{\chi_j\omega}_{\L^p(\Omega^\bullet(M))}^p.
\]
Conversely,
\[
\sum_{j=1}^m
\norm{\chi_j\omega}_{\L^p(\Omega^\bullet(M))}^p
\leq
m\norm{\omega}_{\L^p(\Omega^\bullet(M))}^p.
\]
Combining these two estimates with \eqref{eq:chartwise-Lp-equivalence} proves \eqref{eq:equiv-global-local-Lp-correct}.
\end{proof}

Now, we can prove the following result.

\begin{thm}
\label{thm:Rellich-local-forms-approx}
Let $(M,g)$ be a smooth compact Riemannian manifold of dimension $d \geq 1$. Suppose that $1 \leq p< \infty$. For each integer $k\in\{0,\ldots,d\}$, consider the canonical embedding
\[
J_k \co \W^{1,p}(\Omega^k(M)) \to \L^p(\Omega^k(M)).
\]
Then the following statements hold.
\begin{enumerate}
\item The canonical embedding $
J\ov{\mathrm{def}}{=}\bigoplus_{k=0}^d J_k
\co
\W^{1,p}(\Omega^\bullet(M))
\to
\L^p(\Omega^\bullet(M))$ 
is compact.
\item There exists a constant $C>0$ such that
\begin{equation}
\label{eq:approximation-numbers-Rellich-forms}
a_n(J)
\leq Cn^{-\frac1d},
\quad
n\geq1.
\end{equation}
\end{enumerate}
\end{thm}

\begin{proof}
For each point of $M$, choose a coordinate chart $(\widetilde U,\widetilde\varphi)$ and a Euclidean ball $V$ such that $\ovl V\subset\widetilde\varphi(\widetilde U)$. By compactness, there exist finitely many such charts $(\widetilde U_j,\widetilde\varphi_j)$ and balls $V_j$, $j\in\{1,\ldots,m\}$, such that the open sets $
U_j
\ov{\mathrm{def}}{=}
\widetilde\varphi_j^{-1}(V_j)$ cover $M$. Set $\varphi_j\ov{\mathrm{def}}{=}\widetilde\varphi_j|_{U_j}$ and choose a smooth partition of unity $(\chi_1,\ldots,\chi_m)$ subordinate to $(U_1,\ldots,U_m)$. We use on $\W^{1,p}(\Omega^\bullet(M))$ the equivalent norm \eqref{Sobolev-charts} associated with this atlas and this partition of unity. For $\omega\in\Omega^\bullet(M)$, write
\begin{equation}
\label{eq:coefficients-chi-omega}
(\varphi_j^{-1})^*(\chi_j\omega)
=
\sum_{k=0}^d\sum_{|I|=k}f_{j,I}\d x_I,
\quad j\in\{1,\ldots,m\}.
\end{equation}
Since $\supp \chi_j \subset U_j$, each coefficient $f_{j,I}$ belongs to $\C_c^\infty(V_j)$. Set 
$
N
\ov{\mathrm{def}}{=}
m2^d
$ 
and introduce the finite $\ell^1$-direct sums
\[
\cal X
\ov{\mathrm{def}}{=}
\left(
\bigoplus_{j=1}^m
\bigoplus_{k=0}^d
\bigoplus_{|I|=k}
\W^{1,p}(V_j)
\right)_{\ell^1},
\qquad
\cal Y
\ov{\mathrm{def}}{=}
\left(
\bigoplus_{j=1}^m
\bigoplus_{k=0}^d
\bigoplus_{|I|=k}
\L^p(V_j)
\right)_{\ell^1}.
\]
Define
\[
\cal T\omega
\ov{\mathrm{def}}{=}
(f_{j,I})_{j,k,I}.
\]
By \eqref{Sobolev-charts}, we have
\[
\norm{\cal T\omega}_{\cal X}
=
\sum_{j=1}^m
\sum_{k=0}^d
\sum_{|I|=k}
\norm{f_{j,I}}_{\W^{1,p}(V_j)}
=
\norm{\omega}_{\W_{\mathfrak A}^{1,p}(\Omega^\bullet(M))}
\lesssim
\norm{\omega}_{\W^{1,p}(\Omega^\bullet(M))}.
\]
Thus $\cal T$ extends uniquely to a bounded operator $
\cal T
\co
\W^{1,p}(\Omega^\bullet(M))
\to
\cal X$. For each $j$, let $
J_{V_j}
\co
\W^{1,p}(V_j)
\to
\L^p(V_j)$ be the canonical embedding and define
\[
\cal J
\ov{\mathrm{def}}{=}
\bigoplus_{j=1}^m
\bigoplus_{k=0}^d
\bigoplus_{|I|=k}
J_{V_j}
\co
\cal X
\to
\cal Y.
\]
For each $j$, denote by $
\cal E_j
\co
\L^p(\Omega^\bullet(U_j))
\to
\L^p(\Omega^\bullet(M))$ the extension by zero. If $h=(h_{j,I})_{j,k,I}\in\cal Y$, set
\[
\eta_j(h)
\ov{\mathrm{def}}{=}
\varphi_j^*
\left(
\sum_{k=0}^d
\sum_{|I|=k}
h_{j,I}\d x_I
\right)
\]
and define
\begin{equation}
\label{eq:def-gluing-operator}
\cal Sh
\ov{\mathrm{def}}{=}
\sum_{j=1}^m
\cal E_j\eta_j(h).
\end{equation}
By \eqref{eq:local-coordinate-Lp-equivalence}, we have
\begin{align*}
\MoveEqLeft
\norm{\cal Sh}_{\L^p(\Omega^\bullet(M))}
\leq
\sum_{j=1}^m
\norm{\eta_j(h)}_{\L^p(\Omega^\bullet(U_j))}
\lesssim
\sum_{j=1}^m
\sum_{k=0}^d
\sum_{|I|=k}
\norm{h_{j,I}}_{\L^p(V_j)}
=
\norm{h}_{\cal Y}.
\end{align*}
Consequently $
\cal S
\co
\cal Y
\to
\L^p(\Omega^\bullet(M))$ is bounded. We claim that
\begin{equation}
\label{eq:corrected-Rellich-factorization}
J
=
\cal S\cal J \cal T.
\end{equation}
Indeed, if $\omega\in\Omega^\bullet(M)$, then \eqref{eq:coefficients-chi-omega} gives
\begin{align*}
\cal S\cal J\cal T\omega
=
\sum_{j=1}^m
\cal E_j
\varphi_j^*
\left(
\sum_{k=0}^d
\sum_{|I|=k}
f_{j,I}\d x_I
\right)=
\sum_{j=1}^m
\cal E_j
\bigl((\chi_j\omega)|_{U_j}\bigr)
=
\sum_{j=1}^m
\chi_j\omega
=
\omega.
\end{align*}
The factorization extends to $\W^{1,p}(\Omega^\bullet(M))$ by density of $\Omega^\bullet(M)$ and boundedness of the three operators.

Each $V_j$ is a bounded ball and hence has smooth boundary. Theorem \ref{Th-Edmunds} shows that $J_{V_j}$ is compact and, after increasing the constant to cover the finitely many initial indices, that there exists $C_j>0$ such that
\begin{equation}
\label{eq:local-approximation-number-bound}
a_n(J_{V_j})
\leq
C_jn^{-\frac1d},
\quad n\geq1.
\end{equation}
Thus $\cal J$ is compact, and \eqref{eq:corrected-Rellich-factorization} proves the compactness of $J$. Let $
C_{\mathrm{loc}}
\ov{\mathrm{def}}{=}
\max_{1\leq j\leq m}C_j$. Applying Lemma \ref{lemma:approx-finite-direct-sum} with direct-sum exponent $1$ to the $N$ diagonal blocks of $\cal J$, we obtain
\begin{equation}
\label{eq:approximation-direct-sum-Nn}
a_{Nn}(\cal J)
\leq
C_{\loc}n^{-\frac1d},
\quad n\geq1.
\end{equation}
Let $r\geq N$ and set $
n
\ov{\mathrm{def}}{=}
\left\lfloor\frac rN\right\rfloor$. Then $n \geq 1$, $Nn \leq r$ and $n \geq \frac{r}{2N}$. Since approximation numbers form a decreasing sequence, we have
\[
a_r(\cal J)
\leq
a_{Nn}(\cal J)
\leq
C_{\loc}(2N)^{\frac1d}r^{-\frac1d}.
\]
For $1\leq r<N$, we have
\[
a_r(\cal J)
\leq
\norm{\cal J}
\leq
\norm{\cal J}N^{\frac1d}r^{-\frac1d}.
\]
Consequently, there exists $C_{\cal J}>0$ such that
\[
a_r(\cal J)
\leq
C_{\cal J}r^{-\frac1d},
\quad r\geq1.
\]
Finally, the ideal property of approximation numbers and \eqref{eq:corrected-Rellich-factorization} yield
\[
a_r(J)
\leq
\norm{\cal S}
a_r(\cal J)
\norm{\cal T}
\leq
C_{\cal J}
\norm{\cal S}
\norm{\cal T}
r^{-\frac1d},
\quad r\geq1.
\]
This proves \eqref{eq:approximation-numbers-Rellich-forms}.
\end{proof}

\subsection{Compact Banach spectral triples}

\paragraph{Compact Banach spectral triples}
First, we recall the notion of compact Banach spectral triple introduced in \cite{Arh26a} (see also \cite[Definition 5.10 p.~218]{ArK22} for a slight variant). Consider a triple $(\cal{A},X,D)$ consisting of the following data: a Banach space $X$, a bisectorial operator $D$ on $X$ with dense domain $\dom D \subset X$, and a Banach algebra $\cal{A}$ equipped with a continuous homomorphism $\pi \co \cal{A} \to \B(X)$. Heuristically, a Banach spectral triple plays the role of a ``noncommutative manifold'': the algebra $\cal{A}$ replaces an algebra $\C(\cal{X})$ of continuous functions defined on a topological space $\cal{X}$, while the operator $D$ encodes the underlying first-order differential structure. The norm of the commutator $[D,\pi(a)]$ measures how ``Lipschitz'' a ``noncommutative function'' $a$ is, and the density of the Lipschitz algebra $\Lip_D(\cal A)$ in $\cal{A}$ expresses that this differential structure is rich enough to approximate the whole algebra $\cal{A}$. Finally, compactness of the resolvent of the operator $D$ is the analytic counterpart of compactness in the commutative geometric setting.

We say that $(\cal{A},X,D)$ is a compact Banach spectral triple if
\begin{enumerate}
\item{} $D$ admits a bounded $\H^\infty(\Sigma^\bi_\theta)$ functional calculus on a bisector $\Sigma^\bi_\theta$ for some angle $\theta \in (\omega_{\bi}(D),\frac{\pi}{2})$,
\item{} $D$ has compact resolvent, i.e., there exists $\lambda$ belonging to the resolvent set $\rho(D)\ov{\mathrm{def}}{=} \mathbb{C}\backslash \sigma(D)$ such that the resolvent $R(\lambda,D)$ is compact,
\item{} the Lipschitz algebra
\begin{align}
\label{Lipschitz-algebra-def}
\MoveEqLeft
\Lip_D(\cal{A}) 
\ov{\mathrm{def}}{=} \big\{a \in \cal{A} : \pi(a) \cdot \dom D \subset \dom D 
\text{ and the unbounded operator } \\
&\qquad \qquad  [D,\pi(a)] \co \dom D \subset X \to X \text{ extends to an element of } \B(X)\big\}. \nonumber
\end{align} 
is dense in the algebra $\cal{A}$.
\end{enumerate}
If in addition, there exists an isomorphism $\gamma \co X \to X$, with $\gamma^2=\Id$, 
\begin{equation}
\label{grading-ST}
D\gamma=-\gamma D 
\quad \text{and} \quad 
\gamma \pi(a)=\pi(a)\gamma, \quad a \in \cal{A},
\end{equation}
the spectral triple is said to be even or graded. Otherwise, it is said to be odd or ungraded.

Let $(M,g)$ be a smooth compact Riemannian manifold. For any function $f \in \C(M)$, we introduce the multiplication operator $M_f \co \L^p(\Omega^\bullet(M)) \to \L^p(\Omega^\bullet(M))$, $\omega \mapsto f\omega$. We can consider the representation $\pi \co \C(M) \to \B(\L^p(\Omega^\bullet(M)))$, $f \mapsto M_f$, i.e.,
\begin{equation}
\label{def-de-pi}
\pi(f)\omega 
\ov{\mathrm{def}}{=} f\omega,\quad f \in \C(M), \omega \in \L^p(\Omega^\bullet(M)).
\end{equation}
Now, we can prove the following result. Even in the case $p=2$, note that our approach is different from that of the proof of \cite[Theorem 3.4.1]{SuZ23}. If $0 < q < \infty$ and if $X$ and $Y$ are Banach spaces, we will use the notation
\begin{equation}
\label{def-Sqs}
S^{q,\infty}_\app(X,Y)
\ov{\mathrm{def}}{=} \big\{T \co X \to Y : (a_n(T)) \in \ell^{q,\infty} \big\}
\end{equation}
of \cite[Definition 1.d.18 p.~72]{Kon86} and $S^{q,\infty}_\app(X)$ for $S^{q,\infty}_\app(X,X)$. These spaces are useful for measuring the degree of compactness of operators. Following \cite{Arh26a}, we say that a compact spectral triple $(\cal{A},X,D)$ is $q^+$-summable if the resolvent operator $R(\i,D)$ belongs to $S^{q,\infty}_\app(X)$.

\begin{thm}
\label{thm:compact-Hodge-Dirac}
Let $(M,g)$ be a smooth compact Riemannian manifold of dimension $d \geq 1$. Suppose that $1 < p < \infty$. Then $(\C(M), \L^p(\Omega^\bullet(M)), D_p)$ is a compact Banach spectral triple, which is $d^+$-summable.
\end{thm}

\begin{proof}
The first point is satisfied by Corollary \ref{cor-Hodge-Dirac-functional-calculus}. Now, we prove the second point, i.e., $R(\i,D_{p})$ is a compact operator on the space $\L^p(\Omega^\bullet(M))$. Actually, we will show in addition that $R(\i,D_{p})$ belongs to the space $S_{\app}^{d,\infty}(\L^p(\Omega^\bullet(M)))$, i.e., the spectral triple is $d^+$-summable. 

Since $D_p$ is bisectorial, we have $\i \in \rho(D_{p})$. So the operator $(\i\,\Id-D_p)^{-1} \co \L^p(\Omega^\bullet(M)) \to \W^{1,p}(\Omega^\bullet(M))$ is bounded by definition\footnote{\thefootnote. Consider a closed operator $T$ on a Banach space $X$. We equip $\dom T$ with the graph norm, which is a Banach space since $T$ is closed. For any $\lambda \in \rho(T)$, the bijection $\lambda\,\Id - T \co \dom T \to X$ is continuous. By the bounded inverse theorem, the resolvent operator $(\lambda\,\Id-T)^{-1}$ is continuous from $X$ onto the space $\dom T$.}. We denote by $J \co \W^{1,p}(\Omega^\bullet(M)) \to \L^p(\Omega^\bullet(M))$ the canonical injection, which is compact and belongs to the space $S_{\app}^{d,\infty}(\W^{1,p}(\Omega^\bullet(M)),\L^p(\Omega^\bullet(M)))$ by Theorem \ref{thm:Rellich-local-forms-approx}. Then the resolvent operator $R(\i,D_{p}) \co \L^p(\Omega^\bullet(M)) \to \L^p(\Omega^\bullet(M))$ admits the factorization
\[
\L^p(\Omega^\bullet(M))\xra{(\i\,\Id-D_p)^{-1}} 
\W^{1,p}(\Omega^\bullet(M)) 
\xra{J} \L^p(\Omega^\bullet(M)).
\]
Using the ideal properties of \cite[Theorem 4.8 p.~158]{Kat76} and \cite[Lemma p.~72]{Kon86}, we conclude that the operator $R(\i,D_{p})$ is compact and belongs to the space $S_{\app}^{d,\infty}(\L^p(\Omega^\bullet(M)))$.

Finally, Proposition \ref{prop:commutator-Hodge-Dirac} shows that for every $f\in \C^\infty(M)$, we have $M_f(\dom D_p)\subset\dom D_p$ and that the commutator $[D_p,M_f]\co\dom D_p\to\L^p(\Omega^\bullet(M))$ extends to a bounded operator on $\L^p(\Omega^\bullet(M))$. Hence $\C^\infty(M) \subset \Lip_{D_p}(\C(M))$. Since $\C^\infty(M)$ is dense in $\C(M)$ for the uniform norm, we conclude that $\Lip_{D_p}(\C(M))$ is dense in $\C(M)$. Thus the third condition in the definition of a compact Banach spectral triple is satisfied.
\end{proof}

\section{Index of the Euler operator}
\label{Section-Index-Euler-operator}

In this section, we study the Fredholm index of the even part of the $\L^p$-Hodge--Dirac operator and relate it to the Euler characteristic $\chi(M)$ of the smooth compact Riemannian manifold $M$. We first show that the natural grading by parity of the degree turns the compact Banach spectral triple associated with $D_p$ into an even one. We then combine the $\L^p$-Hodge decomposition and a Banach Fredholm module pairing to prove that the corresponding Fredholm operator has index equal to $\chi(M)$.

\begin{thm}
\label{thm:even-Hodge-Dirac}
Let $(M,g)$ be a smooth compact Riemannian manifold of dimension $d$. Suppose that $1 < p < \infty$. Then the maps $\Omega^k(M) \to \Omega^k(M)$, $\omega \mapsto (-1)^{k}\omega$, where $k \in \{0,\ldots,d\}$, induce an isometric isomorphism $\gamma \co \L^p(\Omega^\bullet(M)) \to \L^p(\Omega^\bullet(M))$ and $(\C(M),\L^p(\Omega^\bullet(M)),D_p,\gamma)$ is an even compact Banach spectral triple.
\end{thm}

\begin{proof}
By Theorem~\ref{thm:compact-Hodge-Dirac}, the triple $(\C(M),\L^p(\Omega^\bullet(M)),D_p)$ is a compact Banach spectral triple. We define a linear map $
\gamma \co \Omega^\bullet(M) \to \Omega^\bullet(M)$ 
by prescribing its action on homogeneous forms:
\begin{equation}
\label{def-gamma-123}
\gamma\omega \ov{\mathrm{def}}{=} (-1)^k \omega,
\quad
\omega \in \Omega^k(M), k = 0,\dots,d,
\end{equation}
and extending linearly to arbitrary elements of $\Omega^\bullet(M)$. It is immediate that $
\gamma^2 = \Id_{\Omega^\bullet(M)}$. If $\omega \in \Omega^\bullet(M)$, we can write $\omega = \sum_{k=0}^d \omega_k$ with $\omega_k \in \Omega^k(M)$. Observe that
\begin{align*}
\norm{\gamma\omega}_{\L^p(\Omega^\bullet(M))}
&\ov{\eqref{norm-Lp-Df}\eqref{def-gamma-123}}{=}
\bigg(\int_M \bigg(\sum_{k=0}^d |(-1)^k\omega_k(x)|_{\Lambda_{\mathbb{C}}^k \mathrm{T}_x^*M}^2\bigg)^{\frac{p}{2}} \d\mu_g(x)\bigg)^{\frac1p}
\\
&=
\bigg(\int_M \bigg(\sum_{k=0}^d |\omega_k(x)|_{\Lambda_{\mathbb{C}}^k \mathrm{T}_x^*M}^2\bigg)^{\frac{p}{2}} \d\mu_g(x)\bigg)^{\frac1p}
\ov{\eqref{norm-Lp-Df}}{=} \norm{\omega}_{\L^p(\Omega^\bullet(M))}.
\end{align*}
Since $\Omega^\bullet(M)$ is dense in the Banach space $\L^p(\Omega^\bullet(M))$, the linear operator $\gamma \co \Omega^\bullet(M) \to \Omega^\bullet(M)$ clearly extends by continuity to an isometric isomorphism $\gamma \co \L^p(\Omega^\bullet(M)) \to \L^p(\Omega^\bullet(M))$ satisfying $\gamma^2=\Id_{\L^p(\Omega^\bullet(M))}$. 

Let $f \in \C^\infty(M)$ and let $\omega \in \Omega^\bullet(M)$. Write $\omega = \sum_{k=0}^d \omega_k$ with $\omega_k \in \Omega^k(M)$. Then $\pi(f)\omega = f\omega = \sum_{k=0}^d f\omega_k$, and each $f\omega_k$ is a $k$-form. Consequently, we have
\[
\gamma\big(\pi(f)\omega\big)
=
\gamma\bigg(\sum_{k=0}^d f\omega_k\bigg)
=
\sum_{k=0}^d (-1)^k f\omega_k
=
f\sum_{k=0}^d (-1)^k \omega_k
=
\pi(f)\gamma\omega.
\]
Hence $\gamma\pi(f) = \pi(f)\gamma$ on $\Omega^\bullet(M)$, and by density and continuity we obtain on $\L^p(\Omega^\bullet(M))$ the equality
\[
\gamma\pi(f) = \pi(f)\gamma,
\quad f \in \C^\infty(M).
\]
Now, let $f \in \C(M)$. There exists a sequence $(f_n)$ in $\C^\infty(M)$ such that $\norm{f_n-f}_{\L^\infty(M)} \to 0$. Since $\pi$ is the representation by multiplication operators, we have
$
\norm{\pi(f_n)-\pi(f)}_{\B(\L^p(\Omega^\bullet(M)))}
= \norm{f_n-f}_{\L^\infty(M)} \to 0$. Hence, using $\gamma\pi(f_n)=\pi(f_n)\gamma$, we obtain
\begin{align*}
\norm{\gamma\pi(f)-\pi(f)\gamma}
\leq
\norm{\gamma(\pi(f)-\pi(f_n))}
+
\norm{(\pi(f_n)-\pi(f))\gamma}
\leq
2\norm{\gamma}\norm{f_n-f}_{\infty} \to 0.
\end{align*}
Consequently, we have $
\gamma\pi(f)=\pi(f)\gamma$ for any $f \in \C(M)$.

Finally, we verify the anticommutation relation with $D$. Recall that $D \ov{\mathrm{def}}{=} \d + \d^*$, where $\d$ increases the degree by one and $\d^*$ decreases the degree by one. Let $\omega_k \in \Omega^k(M)$. Then
\begin{equation}
\label{inter-ABC}
D\omega_k
\ov{\eqref{def-Hodge-Dirac}}{=} \d\omega_k + \d^*\omega_k
=\alpha_{k+1} + \beta_{k-1},
\end{equation}
where $\alpha_{k+1} \ov{\mathrm{def}}{=} \d\omega_k$ belongs to $\Omega^{k+1}(M)$ and $\beta_{k-1} \ov{\mathrm{def}}{=} \d^*\omega_k$ belongs to $\Omega^{k-1}(M)$ (with the convention $\alpha_{d+1} = 0$ if $k=d$ and $\beta_{-1} = 0$ if $k=0$). On the one hand, since the map $\gamma$ acts by $(-1)^{\deg}$ on homogeneous forms, we have
\begin{equation}
\label{inter-UYIIII}
\gamma D\omega_k
\ov{\eqref{inter-ABC}}{=}
\gamma(\alpha_{k+1} + \beta_{k-1})
\ov{\eqref{def-gamma-123}}{=}
(-1)^{k+1}\alpha_{k+1} + (-1)^{k-1}\beta_{k-1}.
\end{equation}
On the other hand, we have
\begin{equation}
\label{inter-DFGT}
D\gamma\omega_k
\ov{\eqref{def-gamma-123}}{=} (-1)^k D\omega_k
\ov{\eqref{inter-ABC}}{=} (-1)^k \alpha_{k+1} + (-1)^k \beta_{k-1}.
\end{equation}
Therefore
\begin{align*}
\MoveEqLeft
(\gamma D + D\gamma)\omega_k
\ov{\eqref{inter-DFGT} \eqref{inter-UYIIII}}{=}
\big((-1)^{k+1} + (-1)^k\big)\alpha_{k+1}
+\big((-1)^{k-1} + (-1)^k\big)\beta_{k-1}
= 0.
\end{align*}
%
We conclude by linearity that $\gamma D = -D \gamma$ on $\Omega^\bullet(M)$. By definition, $D_p$ is the closure of $D$ acting on the subspace $\Omega^\bullet(M)$ of the Banach space $\L^p(\Omega^\bullet(M))$. In particular, $\Omega^\bullet(M)$ is a core for the operator $D_p$. To pass from the anticommutation relation on $\Omega^{\bullet}(M)$ to the domain of the operator $D_{p}$, consider some $u \in \dom D_{p}$. By definition, there exists a sequence $(u_k)$ of elements of $\Omega^{\bullet}(M)$ such that
\[
u_k \to u
\quad\text{and}\quad
D u_k \to D_{p}u
\]
in $\L^p(\Omega^{\bullet}(M))$. Since the map $\gamma$ is bounded on $\L^p(\Omega^{\bullet}(M))$, we have $\gamma u_k \to \gamma u$ in $\L^p(\Omega^{\bullet}(M))$. Moreover, using $\gamma D = - D \gamma$ on $\Omega^{\bullet}(M)$, we obtain
\[
D(\gamma u_k)
=-\gamma(D u_k)
\to
-\gamma(D_{p}u)
\qquad \text{in }\L^p(\Omega^{\bullet}(M)).
\]
Therefore $(\gamma u_k)$ is an approximating sequence in the graph norm for $\gamma u$, which shows that $\gamma u \in \dom D_{p}$ and
\[
D_{p}(\gamma u)
=
-\gamma(D_{p}u).
\]
Hence $\gamma D_{p} + D_{p}\gamma = 0$ on the space $\dom D_{p}$.
\end{proof}

Recall that we have the canonical decomposition $\Omega^\bullet(M)
= \Omega^{\even}(M) \oplus \Omega^{\odd}(M)$ of the space of smooth forms, where
\[
\Omega^{\even}(M) \ov{\mathrm{def}}{=} \bigoplus_{k \text{ even}} \Omega^k(M) 
\quad \text{and} \quad
\Omega^{\odd}(M) \ov{\mathrm{def}}{=} \bigoplus_{k \text{ odd}} \Omega^k(M).
\]
By definition of the $\L^p$-norms on forms, this decomposition induces by completion a direct sum decomposition
\begin{equation}
\label{decompo-Lp-even-odd}
\L^p(\Omega^\bullet(M))
= \L^p(\Omega^{\even}(M)) \oplus \L^p(\Omega^{\odd}(M)).
\end{equation}
With respect to this decomposition, the operator $D_p$ admits by Proposition \ref{prop-anticommuting-symmetry-unbounded} the block matrix form 
\begin{equation}
\label{Dp-decompo-block}
D_p 
= \begin{bmatrix}
  0   & D_{p,-} \\
 D_{p,+}  &  0 \\
\end{bmatrix},
\end{equation}
where the operators $
D_{p,+} \co \dom D_{p,+} \subset \L^p(\Omega^{\even}(M)) \to \L^p(\Omega^{\odd}(M))$ and $
D_{p,-} \co \dom D_{p,-} \subset \L^p(\Omega^{\odd}(M)) \to \L^p(\Omega^{\even}(M))$ 
are the restrictions of the Hodge--Dirac operator to subspaces of even and odd forms respectively.

\paragraph{$\L^p$-Hodge decomposition} 
For each degree $k \in \{0,\ldots,d\}$, we introduce the space $\cal{H}^k(M)$ of smooth harmonic $k$-forms, that is
\[
\cal{H}^k(M)
\ov{\mathrm{def}}{=}
\{\omega \in \Omega^k(M) \co \Delta_{\HdR} \omega = 0\},
\]
where $\Delta_{\HdR}$ is the Hodge-de Rham Laplacian defined in \eqref{Hodge-deRham-Laplacian}. Recall that the Hodge theorem \cite[Theorem 6 p.~547]{GMS98} (see also \cite[Theorem 6.11 p.~225]{War83} \cite[p.~322]{Dieu82} in the simpler case of orientable manifolds) asserts that every cohomology class in the de Rham cohomology group $\H^k_{\dR}(M,\mathbb{C})$ admits a unique harmonic representative, so
\begin{equation}
\label{harmonic-space-and-de-Rham}
\cal{H}^k(M) 
\cong \H^k_{\dR}(M,\mathbb{C}).
\end{equation}
In particular, $\cal{H}^k(M)$ is finite-dimensional. For any differential form $\omega \in \Omega^\bullet(M)$, it is known \cite[Proposition 7.5.2 p.~539]{AMR88} that $\Delta_{\HdR} \omega = 0$ if and only if $\d\omega = 0$ and $\d^* \omega = 0$.
Suppose that $1 < p < \infty$. For any integer $k \in \{0,\ldots,d\}$, by \cite[Proposition 6.5 p.~2091]{Sco95} (see also \cite[Theorem 1.2 p.~3618]{Li09} and \cite[Theorem~5.7 p.~63]{ISS99}), we have the following $\L^p$-Hodge direct sum decomposition
\begin{equation}
\label{eq:Lp-Hodge-decomp}
\L^p(\Omega^k(M))
=
\d_p \W^{1,p}(\Omega^{k-1}(M))
\oplus
(\d^*)_p \W^{1,p}(\Omega^{k+1}(M))
\oplus
\cal{H}^k(M),
\end{equation}
where $\d_p$ and $(\d^*)_p$ are the closures of the operators $\d$ and $\d^*$. Although this decomposition is stated in \cite[Proposition 6.5 p.~2091]{Sco95} and \cite[Theorem~5.7 p.~63]{ISS99} in the orientable case, it remains valid on compact non-orientable manifolds as well, according to \cite[Theorem 1.2 p.~3618]{Li09}.

We also refer to \cite[Theorem 5 p.~546]{GMS98} for the classical case $p=2$. Following \cite{Sco95} and \cite{ISS99}, we introduce the bounded ``harmonic'' projection $P_p \co \L^p(\Omega^\bullet(M)) \to \L^p(\Omega^\bullet(M))$ onto the subspace $\cal{H}(M) \ov{\mathrm{def}}{=} \oplus_{k=0}^d\cal{H}^k(M)$ and its restriction $P \co \Omega^\bullet(M) \to \Omega^\bullet(M)$.
Moreover, we set
\[
\cal{H}^{\even}(M) 
\ov{\mathrm{def}}{=} \bigoplus_{k \text{ even}} \cal{H}^k(M),
\qquad
\cal{H}^{\odd}(M) 
\ov{\mathrm{def}}{=} \bigoplus_{k \text{ odd}} \cal{H}^k(M).
\]
%
We denote by $\chi(M) \ov{\mathrm{def}}{=} \sum_{k=0}^{\dim M} (-1)^k \dim \H_{\dR}^k(M,\R)$ the Euler characteristic of $M$. Finally, we set
\begin{equation}
\label{def-de-scrD}
\scr{D}
\ov{\mathrm{def}}{=} |D_p| + P_p.
\end{equation}

\paragraph{Green's operator}
Let $E$ be a Hermitian vector bundle of finite rank over a compact manifold $M$. Consider a self-adjoint elliptic differential operator $A \co \C^\infty(M,E) \to \C^\infty(M,E)$ of order $m$. If $P \co \C^\infty(M,E) \to \C^\infty(M,E)$ is the restriction of the orthogonal projection from the Hilbert space $\L^2(M,E)$ on the subspace $\ker A$ then by \cite[Theorem 3.15 p.~15]{BDIP02} (see also \cite[Proposition 18 p.~81]{Mel06}) there exists a unique pseudo-differential operator $G \co \C^\infty(M,E) \to \C^\infty(M,E)$ such that
\begin{equation*}
GA
= \Id_{\C^\infty(M,E)} - P, 
\quad
AG
= \Id_{\C^\infty(M,E)} - P 
\quad \text{and} \quad PG=GP=0.
\end{equation*}
In this case, the order of $G$ is $-m$ and it is called Green's operator of $A$. We warn the reader that, in some references (see for instance \cite[Section 23.31.11]{Dieu82}), the term Green’s operator is used to denote the resolvent $(A-\lambda\, \Id)^{-1}$, where $\lambda$ does not belong to the spectrum of $A$ on $\L^2(M,E)$. Applying this result, we obtain the following proposition. 

\begin{prop}
\label{prop-existence-Green}
Let $(M,g)$ be a smooth compact Riemannian manifold. There exists a unique pseudo-differential operator $G \co \Omega^{\bullet}(M) \to \Omega^{\bullet}(M)$ of order $-2$ such that
\begin{equation}
\label{parametrix}
G\Delta_{\HdR} 
= \Id - P,
\quad
\Delta_{\HdR}G
= \Id - P
\quad \text{and} \quad PG=GP=0.
\end{equation}
\end{prop}
This result is essentially contained in \cite[Section 5]{Sco95} in the orientable case. Using \cite[Proposition 5.1]{Arh26b}, we see that in this situation $G$ induces a bounded operator $G \co \L^p(\Omega^\bullet(M)) \to \W^{2,p}(\Omega^\bullet(M))$ such that 
\begin{equation}
\label{parametrix-bis}
G\Delta_{\HdR,p}   
= \Id_{\L^p(\Omega^\bullet(M))}-P_p, \quad 
\Delta_{\HdR,p} G  
= \Id_{\L^p(\Omega^\bullet(M))}-P_p, 
\end{equation}
where $\Delta_{\HdR,p}$ is the closure of the operator $\Delta_{\HdR}$. 
The next result, in the orientable case, is essentially proved in \cite[Proposition 5.2 p.~2085]{Sco95} with a different argument.

\begin{prop}
\label{prop-harmonic-Lp}
Let $(M,g)$ be a smooth compact Riemannian manifold. Suppose that $1 < p < \infty$. We have
\begin{equation}
\label{ker-Delta-p}
\ker \Delta_{\HdR,p}
=\cal{H}(M).
\end{equation}
\end{prop}

\begin{proof}
First, we prove the inclusion $\ker \Delta_{\HdR,p} \subset \cal{H}(M)$. Let $\omega \in \ker \Delta_{\HdR,p}$. By definition of the closure, we have $\omega \in \dom \Delta_{\HdR,p}$ and the equality $\Delta_{\HdR,p}\omega=0$ in $\L^p(\Omega^{\bullet}(M))$, hence in the sense of distributions. By Proposition \ref{prop-existence-Green}, there exists a pseudo-differential operator $G$ of order $-2$ such that $P = \Id-G\Delta_{\HdR}$. Since $G$ is a pseudo-differential operator, it extends continuously to the space of distributions \cite[Section 28]{Dieu88}. Hence the identity $P=\Id-G\Delta_{\HdR}$ also holds in the distributional sense. Applying this identity to $\omega$ in the distributional sense gives
\[
P(\omega)
=\omega.
\]
According to \cite[p.~321]{Dieu82} the harmonic projection $P$ is a smoothing operator (its Schwartz kernel is smooth), hence $\omega=P(\omega)$ belongs to the space $\Omega^{\bullet}(M)$. The reverse inclusion is immediate since every smooth harmonic form belongs to $\dom \Delta_{\HdR,p}$ and is annihilated by $\Delta_{\HdR,p}$.
\end{proof}

We introduce the restrictions 
\[
P_p^{\even}\co \L^p(\Omega^{\even}(M))\to \L^p(\Omega^{\even}(M)),
\quad
P_p^{\odd}\co \L^p(\Omega^{\odd}(M))\to \L^p(\Omega^{\odd}(M))
\]
of $P_p$ on spaces of even and odd degrees. We have $\Ran P_p^{\odd}=\cal{H}^{\odd}(M)$.

\begin{lemma}
\label{lem:range-Laplacian-harmonic-projection}
Let $(M,g)$ be a smooth compact Riemannian manifold. Suppose that $1 < p < \infty$. We have
\begin{equation}
\label{eq:range-Laplacian-kernel-projection}
\Ran \Delta_{\HdR,p}
=\ker P_p, \quad \Ran D_{p,+} 
= \ker P_p^{\odd}
\quad \text{and} \quad 
\Ran D_{p,-} 
= \ker P_p^{\even}.
\end{equation}
\end{lemma}

\begin{proof}
Let $G$ be the Green operator of Proposition~\ref{prop-existence-Green}. By \cite[Proposition 5.1]{Arh26b} and the continuous embedding $\W^{2,p}_\nabla(\Omega^{\bullet}(M))
\hookrightarrow
\L^p(\Omega^{\bullet}(M))$, 
the operator $G$ extends to a bounded operator on $\L^p(\Omega^{\bullet}(M))$. Similarly, the harmonic projection $P$ extends to the bounded projection $P_p$.

We first prove that $
\ker P_p \subset \Ran \Delta_{\HdR,p}$. Let $\eta \in \ker P_p$. Choose a sequence
$(\eta_j)$ in $\Omega^{\bullet}(M)$ such that $
\eta_j \to \eta$ in $\L^p(\Omega^{\bullet}(M))$. Set $
u_j \ov{\mathrm{def}}{=} G\eta_j$. Since $G$ maps smooth forms to smooth forms, we have
$u_j \in \Omega^{\bullet}(M)$. Moreover, $
u_j \to G\eta$ in $\L^p(\Omega^{\bullet}(M))$. On smooth forms, the Green identity gives $
\Delta_{\HdR}u_j
=\Delta_{\HdR}G\eta_j
=
(\Id-P)\eta_j$. Since $P_p$ is the bounded extension of $P$, we have
\[
(\Id-P)\eta_j
=
(\Id-P_p)\eta_j
\to
(\Id-P_p)\eta
=
\eta
\]
since $\eta \in \ker P_p$. Thus $
u_j \to G\eta$ and $
\Delta_{\HdR}u_j \to \eta$ in $\L^p(\Omega^{\bullet}(M))$. Since
$\Delta_{\HdR,p}$ is the closure of
$\Delta_{\HdR}$ initially defined on smooth forms, it follows that $
G\eta \in \dom \Delta_{\HdR,p}$ 
and $
\Delta_{\HdR,p}G\eta=\eta$. Therefore $\eta \in \Ran \Delta_{\HdR,p}$. So $
\ker P_p \subset \Ran \Delta_{\HdR,p}$.

Conversely, we prove that $
\Ran \Delta_{\HdR,p} \subset \ker P_p$. 
Let $u\in\dom \Delta_{\HdR,p}$. By definition of the closure, there
exists a sequence $(u_j)$ in $\Omega^{\bullet}(M)$ such that $
u_j \to u$ and $\Delta_{\HdR}u_j \to \Delta_{\HdR,p}u$ in $\L^p(\Omega^{\bullet}(M))$. For each $j$, we have $
P\Delta_{\HdR}u_j=0$. Indeed, if $h \in \cal{H}(M)$, then by formal self-adjointness,
\[
\big\la \Delta_{\HdR}u_j,h\big\ra_{\L^2(\Omega^{\bullet}(M))}
=
\big\la u_j,\Delta_{\HdR}h\big\ra_{\L^2(\Omega^{\bullet}(M))}
=
0.
\]
Thus $\Delta_{\HdR}u_j$ is orthogonal to the harmonic forms, and hence $P\Delta_{\HdR}u_j=0$. Since $P_p$ extends $P$ and is bounded on $\L^p(\Omega^{\bullet}(M))$, passing to the limit gives
\[
P_p\Delta_{\HdR,p}u
=
\lim_j P_p\Delta_{\HdR}u_j
=
\lim_j P\Delta_{\HdR}u_j
=
0.
\]
Therefore $
\Delta_{\HdR,p}u\in\ker P_p$. Hence $\Ran \Delta_{\HdR,p} \subset \ker P_p$. Combining both inclusions, we obtain $
\Ran \Delta_{\HdR,p}
=
\ker P_p$. In particular, this space is closed. Finally, we have 
\begin{align}
\MoveEqLeft
\label{infinite}
\ker P_p^\odd \oplus \ker P_p^\even
=\ker P_p
=\Ran \Delta_{\HdR,p}
=\ovl{\Ran \Delta_{\HdR,p}}
\ov{\eqref{eq:Hodge-square-Dirac}}{=} \ovl{\Ran D_{p}^2} \\
&\ov{\eqref{Bisec-Ran-Ker}}{=} \ovl{\Ran D_{p}}
\ov{\eqref{Dp-decompo-block}}{=} \ovl{\Ran D_{p,+}} \oplus \ovl{\Ran \D_{p,-}}. \nonumber        
\end{align}
Let $z \in \ker P_p^{\odd}$. By the first part of the proof, applied on odd forms, we have $
\ker P_p^{\odd}
=
\Ran \Delta_{\HdR,p}^\odd$. Hence there exists $v \in \dom \Delta_{\HdR,p}^\odd$ such that $\Delta_{\HdR,p}^\odd v=z$. Since $D_{p}^2 \ov{\eqref{eq:Hodge-square-Dirac}}{=} \Delta_{\HdR,p}$, we have on odd forms $D_{p,+}D_{p,-}
=
\Delta_{\HdR,p}^\odd$. In particular, we have $v \in \dom D_{p,-}$ and $D_{p,-}v \in \dom D_{p,+}$. Then the element $x \ov{\mathrm{def}}{=} D_{p,-}v$ belongs to $\dom D_{p,+}$ and
\[
D_{+,p}x
=D_{p,+} D_{p,-}v
=\Delta_{\HdR,p}^\odd v
=z.
\]
Thus $
\ker P_p^{\odd}
\subset
\Ran D_{p,+}$. The reverse inclusion follows from $\Ran D_{p,+}
\subset
\ovl{\Ran D_{p,+}}
\ov{\eqref{infinite}}{=}
\ker P_p^{\odd}$. Consequently, we have $\Ran D_{p,+}=\ker P_p^{\odd}$. The proof of the identity $\Ran D_{p,-}
=
\ker P_p^{\even}$ is similar.
\end{proof}

The first step is to prove the next result.

\begin{prop}
\label{prop-isomorphism}
Let $(M,g)$ be a smooth compact Riemannian manifold. Suppose that $1 < p < \infty$. The linear map
\begin{equation}
\label{eq:S-iso}
\scr{D} \co \W^{1,p}(\Omega^\bullet(M)) \xrightarrow{\ \cong\ } \L^p(\Omega^\bullet(M))
\end{equation}
is an isomorphism of Banach spaces.
\end{prop}

\begin{proof}
With \eqref{eq:Lp-Hodge-decomp}, we have the topological direct sum
\[
\L^p(\Omega^\bullet(M))
=
\cal{H}(M)\oplus\ker P_p.
\]
Since the space $\cal{H}(M)$ of harmonic forms is finite-dimensional and contained in the Sobolev space $\W^{1,p}(\Omega^\bullet(M))$, the $\L^p$-norm and the $\W^{1,p}$-norm are equivalent on $\cal{H}(M)$. Consequently, for any $\omega\in\W^{1,p}(\Omega^\bullet(M))$, we have
\[
\norm{P_p(\omega)}_{\W^{1,p}(\Omega^\bullet(M))}
\lesssim_p
\norm{P_p(\omega)}_{\L^p(\Omega^\bullet(M))}
\lesssim_p
\norm{\omega}_{\L^p(\Omega^\bullet(M))}
\leq
\norm{\omega}_{\W^{1,p}(\Omega^\bullet(M))}.
\]
Therefore, the restriction $
P_p^{(1)}
\ov{\mathrm{def}}{=}
P_p\big|_{\W^{1,p}(\Omega^\bullet(M))}
\co
\W^{1,p}(\Omega^\bullet(M))
\to
\W^{1,p}(\Omega^\bullet(M))$ is a bounded projection with range $\cal{H}(M)$. With the notation
\begin{equation}
\label{def-Fp}
F_p
\ov{\mathrm{def}}{=}\ker P_p^{(1)}
=\W^{1,p}(\Omega^\bullet(M))\cap\ker P_p,
\end{equation}
we obtain the topological direct sum
\[
\W^{1,p}(\Omega^\bullet(M))
=
\cal{H}(M)\oplus F_p.
\]
Since $\ker D_p \ov{\eqref{Bisec-Ran-Ker}}{=} \ker \Delta_{\HdR,p} \ov{\eqref{ker-Delta-p}}{=} \cal{H}(M)$, we obtain the action of $\scr{D}$ on the first summand:
\[
\scr{D}\big|_{\cal{H}(M)} 
\ov{\eqref{def-de-scrD}}{=} (|D_p| + P_p )\big|_{\cal{H}(M)} 
= \Id_{\cal{H}(M)}.
\]
According to \cite[Theorem 13.6 a) p.~318]{BlB13}, the Hodge--Dirac operator $D$ is an elliptic differential operator of order 1. By \cite[17.13.5 p.~298]{Dieu72}, 
we deduce that its square $D^2$ is a differential operator of order 2, which is elliptic by \cite[p.~277]{Dieu88}. Finally, by \cite[Theorem 1 p.~124]{Bur68} \cite{See67} \cite{Shu01}, we conclude that $|D|=(D^2)^{\frac{1}{2}}$ is an elliptic pseudo-differential operator of order $1$. Note that this operator coincides with the restriction of the operator $|D_p|$. By \cite[Proposition 5.1]{Arh26b}, we have a continuous map $|D_p| \co \W^{1,p}(\Omega^\bullet(M)) \to \L^p(\Omega^\bullet(M))$. By restriction, we deduce a bounded map
\[
|D_p|\big|_{F_p} \co F_p \to \L^p(\Omega^\bullet(M)),
\]
which is injective since $\ker |D_p| =\ker D_p = \cal{H}(M)$. Now, we will show that its range is $\ker P_p$.  
We see that 
\begin{equation}
\label{inter-end}
\ovl{\Ran |D_p|} = \ovl{\Ran (D_p^2)^{\frac{1}{2}}} \subset \ovl{\Ran \Delta_{\HdR,p}} \ov{\eqref{eq:range-Laplacian-kernel-projection}}{=} \ker P_p,
\end{equation}
where we used \cite[Corollary 3.1.11 p.~66]{Haa06} in the first inclusion. Let $\eta \in \ker P_p$, so that $P_p \eta = 0$. We let $v \ov{\mathrm{def}}{=} G\eta \in \W^{2,p}(\Omega^\bullet(M))$. Then
\[
\Delta_{\HdR,p} v
= \Delta_{\HdR,p} G\eta
\ov{\eqref{parametrix-bis}}{=} (\Id - P_p)\eta
= \eta.
\]
Since $D_p^2 = \Delta_{\HdR,p}$, this can be rewritten as $D_p^2 v = \eta$. Since $|D_p|$ is an elliptic pseudodifferential operator of order $1$, it induces a bounded map $|D_p| \co \W^{2,p}(\Omega^\bullet(M)) \to \W^{1,p}(\Omega^\bullet(M))$. Consequently the element $u \ov{\mathrm{def}}{=} |D_p|v$ belongs to the space $ \W^{1,p}(\Omega^\bullet(M))$. Moreover, we have
\[
|D_p|u 
= |D_p|^2 v 
= D_p^2 v 
= \eta.
\]
Note that \eqref{inter-end} implies that $\Ran |D_p| \subset \ker P_p$, it follows that
$
P_p u 
= P_p |D_p| v 
= 0.
$
So $u$ belongs to $\W^{1,p}(\Omega^\bullet(M)) \cap \ker P_p \ov{\eqref{def-Fp}}{=} F_p$. We conclude that the map $
|D_p| \big|_{F_p} \co F_p \to \ker P_p$ is an isomorphism of Banach spaces. Since $P_p$ vanishes on the subspace $F_p$, we have
 $\scr{D}\big|_{F_p} 
= |D_p|\big|_{F_p}$, 
which is an isomorphism. Combining the two pieces, we deduce that $\scr{D}$ is continuous and bijective from $\W^{1,p}(\Omega^\bullet(M))$ onto $\L^p(\Omega^\bullet(M))$, and the open mapping theorem \cite[Corollary 1.6.6 p.~44]{Meg98} gives \eqref{eq:S-iso}.
\end{proof}

\paragraph{Fredholm operators}
Following \cite[Definition 4.37 p.~156]{AbA02}, we say that a bounded operator $T \co X \to Y$, acting between complex Banach spaces $X$ and $Y$, is a Fredholm operator if the subspaces $\ker T$ and $\coker(T) \ov{\mathrm{def}}{=} Y/\Ran T$ are finite-dimensional. In this case, we introduce the index
\begin{equation}
\label{Fredholm-Index}
\Index T 
\ov{\mathrm{def}}{=} \dim \ker T-\dim Y/\Ran T.
\end{equation}
Every Fredholm operator has a closed range by \cite[Lemma 4.38 p.~156]{AbA02}. 

According to \cite[Theorem 4.48 p.~163]{AbA02}, the set $\Fred(X,Y)$ of all Fredholm operators from $X$ into $Y$ is an open subset of $\B(X,Y)$ and the index function $\Index \co \Fred(X,Y) \to \Z$ is continuous (hence locally constant). 
If $T \co X \to Y$ and $S \co Y \to Z$ are Fredholm operators then the composition $ST$ is also a Fredholm operator and
\begin{equation}
\label{composition-Index}
\Index ST
=\Index S +\Index T,
\end{equation}
see \cite[Theorem 4.43 p.~158]{AbA02}. 


\paragraph{$\K$-theory} 
We refer to the books \cite{Bl98}, \cite{CMR07}, \cite{Eme24}, \cite{GVF01}, \cite{RLL0} and \cite{WeO93} for more information and to \cite{BaS11} for a short introduction. We briefly recall the algebraic construction of the $\K_0$-group, which applies to Banach algebras as well as to more general rings.

\paragraph{$\mathrm{K}_0$-group}
Let $\cal{A}$ be a unital ring. We denote by $\M_{\infty}(\cal{A})$ the algebraic direct limit of the algebras $\M_n(\cal{A})$ under the embeddings $x \mapsto \diag(x,0)$. The direct limit $\M_{\infty}(\cal{A})$ can be thought of as the algebra of infinite-dimensional matrices with entries in $\cal{A}$, all but finitely many of which are zero. Two idempotents $e,f \in \M_{\infty}(\cal{A})$ are said to be (Murray-von Neumann) equivalent if there exist elements $v,w \in \M_{\infty}(\cal{A})$ such that $e=v w$ and $f=w v$. The equivalence class of an idempotent $e$ is denoted by $[e]$ and we introduce the set $\V(\cal{A}) \ov{\mathrm{def}}{=} \{ [e] : e \in \M_{\infty}(\cal{A}), e^2=e \}$ of all equivalence classes of idempotents in $\M_\infty(\cal{A})$.

According to \cite[pp.~3-4]{CMR07}, 
the set $\V(\cal{A})$ has a canonical structure of abelian semigroup, with the addition operation defined by the formula $[e]+[f] \ov{\mathrm{def}}{=} [\diag(e,f)]$. The identity of the semigroup $\V(\cal{A})$ is given by $[0]$, where $0$ is the zero idempotent. The group $\K_0(\cal{A})$ of the unital ring $\cal{A}$ is the Grothendieck group of the abelian semigroup $\V(\cal{A})$. In particular, each element of $\K_0(\cal{A})$ identifies with a formal difference $[e]-[f]$. Two such formal differences $[e_1] - [f_1]$ and $[e_2] - [f_2]$ are equal in $\K_0(\cal{A})$ if there exists $g \in \V(\cal{A})$ such that $[e_1] + [f_2] +  [g] = [f_1] + [e_2] +  [g]$.

\begin{example} \normalfont 
If $\cal{A} = \C(\cal{X})$ is the algebra of continuous functions on a second countable compact Hausdorff topological space $\cal{X}$, then by \cite[Proposition 8.1.7 p.~305]{Eme24} we have an isomorphism $\K_0(\C(\cal{X})) \cong \K^0(\cal{X})$, where $\K^0(\cal{X})$ is the topological $\K$-theory group defined by vector bundles.
\end{example}

\paragraph{Coupling between Banach K-homology with K-theory}
An even Banach Fredholm module $(X,\pi,F,\gamma)$ over an algebra $\cal{A}$ on a Banach space $X$ consists of a representation $\pi \co \cal{A} \to \B(X)$, a bounded operator $F \co X \to X$ and a bounded operator $\gamma \co X \to X$ with $\gamma^2=\Id_X$ such that
\begin{enumerate}
\item $F^2-\Id_X$ is a compact operator on $X$,
\item for any $a \in \cal{A}$ the commutator $[F, \pi(a)]$ is a compact operator on $X$,
\item $\gamma$ is a symmetry, i.e., $\gamma^2=\Id_X$, and for any $a \in \cal{A}$ we have
\begin{equation}
\label{commuting-rules}
F\gamma
=-\gamma F
\quad \text{and} \quad  
[\pi(a),\gamma]=0 .
\end{equation}
\end{enumerate}
We say that $\gamma$ is the grading operator. In this case, $P \ov{\mathrm{def}}{=} \frac{\Id_X+\gamma}{2}$ is a bounded projection and we can write $X=X_+\oplus X_-$, where $X_+ \ov{\mathrm{def}}{=} \Ran P$ and $X_- \ov{\mathrm{def}}{=}\Ran (\Id_X-P)$. With respect to this decomposition, the equations of \eqref{commuting-rules} imply that we can write
\begin{equation}
\label{Fredholm-even}
\pi(a)
=\begin{bmatrix}
 \pi_+(a)    &  0 \\
  0   & \pi_-(a)  \\
\end{bmatrix}
\quad \text{and} \quad 
F=\begin{bmatrix}
  0   & F_-  \\
  F_+   & 0  \\
\end{bmatrix},
\end{equation}
where $\pi_+ \co \cal{A} \to \B(X_+)$ and $\pi_- \co \cal{A} \to \B(X_-)$ are representations of $\cal{A}$. The following coupling result is proved in \cite{Arh26a}. 

\begin{thm}
\label{th-pairing-even}
Let $\cal{A}$ be a unital algebra. Consider a projection $e$ of the matrix algebra $\M_n(\cal{A})$ and an even Banach Fredholm module $\bigg(X_+ \oplus X_-,\pi,\begin{bmatrix}
  0   & F_-  \\
   F_+  &  0 \\
\end{bmatrix},\gamma\bigg)$ over $\cal{A}$ with $\pi(1)=\Id_X$. We introduce the bounded operator $e_n \ov{\mathrm{def}}{=} (\Id \ot \pi)(e) \co X^{\oplus n} \to X^{\oplus n}$. Then the bounded operator
\begin{equation}
\label{eFe-90}
e_n (\Id \ot F_+) e_n \co e_n(X_+^{\oplus n}) \to e_n(X_-^{\oplus n}) 
\end{equation}
is Fredholm. 
\end{thm}

Consequently, we have a pairing $\la \cdot, \cdot \ra \co \K_0(\cal{A}) \times \K^{0}(\cal{A},\scr{B}) \to \Z$ defined by
\begin{equation}
\label{pairing-even-1}
\big\la [e], [(X,\pi,F,\gamma)] \big\ra 
\ov{\mathrm{def}}{=} \Index e_n (\Id \ot F_+) e_n.
\end{equation}
Here $\K_0(\cal{A})$ is the $\K$-theory group of the algebra $\cal{A}$ and $\K^{0}(\cal{A},\scr{B})$ is a Banach $\K$-homology 
group introduced in \cite{Arh26a}. In the present paper, we can use the class $\scr{B}$ of all Banach spaces. The following result is proved in \cite{Arh26a}.
\begin{prop} 
\label{prop-triple-to-Fredholm}
Let $(\cal{A},X, D)$ be a compact Banach spectral triple over an algebra $\cal{A}$ via a homomorphism $\pi \co \cal{A} \to \B(X)$. Then $(X,\pi,\sgn D)$ is a Banach Fredholm module over $\cal{A}$.
\end{prop}

Now, we prove the first main result of this section.

\begin{thm}
\label{thm:Euler-characteristic-Lp}
Let $(M,g)$ be a smooth compact Riemannian manifold. Suppose that $1 < p < \infty$. Let $[1] \in \K_0(\C(M))$ be the $\K$-theory class of the constant function $1$. Then the linear operator $D_{p,+} \co \W^{1,p}(\Omega^{\even}(M)) \to \L^p(\Omega^{\odd}(M))$ is Fredholm and we have
\[
\big\la [1],[\L^p(\Omega^\bullet(M)),\pi,\sgn D_p,\gamma] \big\ra
= \Index D_{p,+}.
\]
\end{thm}

\begin{proof}
By Theorem~\ref{thm:even-Hodge-Dirac}, the quadruple $(\C(M),\L^p(\Omega^\bullet(M)),D_p,\gamma)$ is an even compact Banach spectral triple. By Proposition~\ref{prop-triple-to-Fredholm}, it gives rise to an even Banach Fredholm module $
\big(\L^p(\Omega^\bullet(M)),\pi,\sgn(D_p),\gamma\big)$ over the algebra $\C(M)$, where $\pi$ is the representation by pointwise multiplication. Moreover, we have
\begin{equation}
\label{eq:sign-Dp-factorization}
\sgn D_p 
= D_p |D_p|^{-1},
\end{equation}
where $|D_p| \ov{\mathrm{def}}{=} (D_p^2)^{\frac12}$ is defined by the sectorial functional calculus of the sectorial operator $D_p^2$. On the subspace $\ker D_p$, we have $\sgn D_p = 0$ and $|D_p|^{-1}=0$. Note that by Theorem \ref{thm:even-Hodge-Dirac} we have $\gamma D_{p}=-D_{p}\gamma$. Since the sign function is odd, we obtain by \cite[Proposition 3.2.10]{Ege15} the equality $
\gamma \sgn(D_{p})
=
-\sgn(D_{p})\gamma$. Thus, with respect to the decomposition \eqref{decompo-Lp-even-odd}, we can write with Proposition \ref{prop-anticommuting-symmetry-unbounded}
\[
\sgn D_p
=
\begin{bmatrix}
0 & (\sgn D_p)_-\\
(\sgn D_p)_+ & 0
\end{bmatrix},
\]
where $
(\sgn D_p)_+ \co \L^p(\Omega^{\even}(M)) \to \L^p(\Omega^{\odd}(M))$ 
and $
(\sgn D_p)_- \co \L^p(\Omega^{\odd}(M)) \to \L^p(\Omega^{\even}(M))$ are bounded operators.

We now apply Theorem \ref{th-pairing-even} with $
\cal{A} = \C(M)$, $X_+ \ov{\mathrm{def}}{=} \L^p(\Omega^{\even}(M))$ and
$X_- \ov{\mathrm{def}}{=} \L^p(\Omega^{\odd}(M))$, the bounded operator $F \ov{\mathrm{def}}{=} \sgn D_p$ and the projection $e = 1 \in \M_1(\C(M))$. Since  $e_1 = (\Id \ot \pi)(1) = \Id_{\L^p(\Omega^\bullet(M))}$ the operator
\[
(\sgn D_p)_+
=e_1(\Id \ot F_+)e_1
\co \L^p(\Omega^{\even}(M)) \to \L^p(\Omega^{\odd}(M))
\]
is Fredholm and we have
\begin{equation}
\label{eq:pairing-signDp-plus}
\big\la [1],[\L^p(\Omega^\bullet(M)),\pi,\sgn D_p,\gamma] \big\ra
=
\Index(\sgn D_p)_+.
\end{equation}
Using the notations $|D_p|_+ \ov{\mathrm{def}}{=} (D_{p,-}D_{p,+})^{\frac12}$ and $
|D_p|_- \ov{\mathrm{def}}{=} (D_{p,+}D_{p,-})^{\frac12}$ and using the compatibility \cite[Theorem 3.2.20 p.~135]{Ege15} of the functional calculi of the operators $D_p$ and $D_p^2$, we obtain
\[
|D_p|
=
(D_p^2)^{\frac12}
\ov{\eqref{Dp-decompo-block}}{=}
\begin{bmatrix}
D_{p,-}D_{p,+} & 0\\
0 & D_{p,+}D_{p,-}
\end{bmatrix}^{\frac12}
=
\begin{bmatrix}
|D_p|_+ & 0\\
0 & |D_p|_-
\end{bmatrix}.
\]
The identity $D_p = \sgn(D_p)|D_p|$ which follows from \eqref{eq:sign-Dp-factorization} translates in block form as
\[
\begin{bmatrix}
0 & D_{p,-}\\
D_{p,+} & 0
\end{bmatrix}
=
\begin{bmatrix}
0 & (\sgn D_p)_-\\
(\sgn D_p)_+ & 0
\end{bmatrix}
\begin{bmatrix}
|D_p|_+ & 0\\
0 & |D_p|_-
\end{bmatrix}.
\]
On the even component, we obtain in particular
\begin{equation}
\label{eq:Dp-plus-factorization}
D_{p,+} 
= (\sgn D_p)_+\,|D_p|_+.
\end{equation} 
The kernel of $D_p$ coincides with the finite-dimensional space $\cal{H}(M)$ of smooth harmonic forms. The harmonic projection $P_p \co \L^p(\Omega^\bullet(M)) \to \L^p(\Omega^\bullet(M))$ is a finite-rank operator. It preserves the $\Z_2$-grading, so with respect to the decomposition \eqref{decompo-Lp-even-odd} it has the block form
\[
P_p 
=\begin{bmatrix}
P_+ & 0\\
0 & P_-
\end{bmatrix}.
\]
Since $|D_p|$ is even and $P_p$ preserves the grading, the operator $\scr{D} \ov{\eqref{def-de-scrD}}{=} |D_p| + P_p$ is also even and by Proposition \ref{prop:commuting-symmetry-diagonal} can be written as
\[
\scr{D} =
\begin{bmatrix}
\scr{D}_+ & 0\\
0 & \scr{D}_-
\end{bmatrix}
\]
for appropriate operators $\scr{D}_\pm$. By Proposition \ref{prop-isomorphism}, the map $
\scr{D} \co \W^{1,p}(\Omega^\bullet(M)) \xrightarrow{\ \cong\ } \L^p(\Omega^\bullet(M))$ 
is an isomorphism of Banach spaces. Moreover, the operator $\sgn(D_p)$ vanishes on the subspace $\ker D_p$, hence $
\sgn(D_p) P_p = 0$. This gives
\[
\sgn(D_p)\,\scr{D}
\ov{\eqref{def-de-scrD}}{=} \sgn(D_p)\,(|D_p| + P_p)
= \sgn(D_p)\,|D_p|
= D_p.
\]
Restricting \eqref{eq:S-iso} to the even part, we obtain an isomorphism
\[
\scr{D}_+ 
\ov{\mathrm{def}}{=} \scr{D}\big|_{\W^{1,p}(\Omega^{\even}(M))}
\co \W^{1,p}(\Omega^{\even}(M)) \xrightarrow{\ \cong\ } X_+.
\]
With respect to the splitting $X_+ \oplus X_-$, the factorization \(D_p = \sgn(D_p) \scr{D}\) reads
\[
\begin{bmatrix}
0 & D_{p,-}\\
D_{p,+} & 0
\end{bmatrix}
=
\begin{bmatrix}
0 & (\sgn D_p)_-\\
(\sgn D_p)_+ & 0
\end{bmatrix}
\begin{bmatrix}
\scr{D}_+ & 0\\
0 & \scr{D}_-
\end{bmatrix},
\]
so, on the even component, we obtain
\begin{equation}
\label{eq:Dp-plus-factorization-with-S}
D_{p,+} 
= (\sgn D_p)_+ \scr{D}_+
\co \W^{1,p}(\Omega^{\even}(M)) \to X_-.
\end{equation}
Recall that the operator $(\sgn D_p)_+ \co X_+ \to X_-$ is Fredholm. Since the map $\scr{D}_+$ is an isomorphism by \eqref{eq:S-iso}, it is a Fredholm operator with index 0. So the composition \eqref{eq:Dp-plus-factorization-with-S} is Fredholm as well, and
\[
\Index D_{p,+}
\ov{\eqref{composition-Index}}{=} \Index (\sgn D_p)_+ + \Index \scr{D}_+
= \Index (\sgn D_p)_+.
\]
\end{proof}

\begin{prop}
\label{prop-ker-HD}
Let $(M,g)$ be a smooth compact Riemannian manifold. Suppose that $1 < p < \infty$. We have $\ker D_{p,+} = \cal{H}^{\even}(M)$.
\end{prop}

\begin{proof}
We first determine the kernel of $D_{p,+}$. If $\omega \in \cal{H}^{\even}(M)$, then by \cite[Proposition 7.5.2 p.~539]{AMR88} we have $\d \omega = 0$ and $\d^* \omega = 0$. So $D\omega \ov{\eqref{def-Hodge-Dirac}}{=} \d \omega+\d^* \omega= 0$. Since $\omega$ is smooth, it belongs to the domain of $D_{p,+}$ and we have \(D_{p,+}\omega  = 0\). Hence
$\cal{H}^{\even}(M) \subset
\ker D_{p,+}$.

Conversely, if $\omega \in \ker D_{p,+}$ then in particular $\omega \in \dom D_{p,+} \subset \L^{p}(\Omega^{\even}(M))$. By definition of $D_{p,+}$ as a restriction of $D_p$, we have $D_p \omega
=D_{p,+}\omega
=0$. Applying $D_p$ once more and using Proposition \ref{prop:closure-Hodge-square-dirac} we obtain $
\Delta_{\HdR,p} \omega 
= D_p^2 \omega 
=0$. By Proposition \ref{prop-harmonic-Lp}, we obtain $\omega \in \cal{H}(M)$, this implies that $\omega$ is smooth and harmonic. Since $\omega \in \L^{p}(\Omega^{\even}(M))$, we conclude that $\omega \in \cal{H}^{\even}(M)$. We have obtained the inclusion $\ker D_{p,+} \subset \cal{H}^{\even}(M)$.
\end{proof}

\begin{prop}
\label{prop-coker-HD}
Let $(M,g)$ be a smooth compact Riemannian manifold. Suppose that $1 < p < \infty$. We have an isomorphism
\begin{equation}
\label{description-coker}
\coker D_{p,+} 
\cong \cal{H}^{\odd}(M).
\end{equation}
\end{prop}

\begin{proof}
By Lemma \ref{lem:range-Laplacian-harmonic-projection}, we have $\Ran D_{p,+} = \ker P_p^{\odd}$. Hence 
$$
\coker D_{p,+} 
= \L^p(\Omega^{\odd}(M)) / \Ran D_{p,+}
=\L^p(\Omega^{\odd}(M))/\ker P_p^{\odd}
\cong \cal{H}^{\odd}(M).
$$
\end{proof}

Now, we obtain the first index theorem of this paper.

\begin{cor}
\label{cor-Euler-operator}
Let $(M,g)$ be a smooth compact Riemannian manifold. Suppose that $1 < p < \infty$. The index of the Fredholm operator $D_{p,+} \co \W^{1,p}(\Omega^{\even}(M)) \to \L^p(\Omega^{\odd}(M))$ is 
\begin{equation}
\label{index-Hodge-Dirac-Lp}
\Index D_{p,+}
= \chi(M).
\end{equation}
\end{cor}

\begin{proof}
According to Theorem \ref{thm:Euler-characteristic-Lp}, $D_{p,+} \co \W^{1,p}(\Omega^{\even}(M)) \to \L^p(\Omega^{\odd}(M))$ is a Fredholm operator. By Proposition \ref{prop-ker-HD} and Proposition \ref{prop-coker-HD}, we have
\begin{align*}
\MoveEqLeft
\Index D_{p,+}
\ov{\eqref{Fredholm-Index}}{=}\dim \ker D_{p,+}-\dim \coker D_{p,+}
= \dim \cal{H}^{\even}(M) - \dim \cal{H}^{\odd}(M) \\
&=\sum_{k=0}^{\dim M} (-1)^k \dim \cal{H}^{k}(M) 
\ov{\eqref{harmonic-space-and-de-Rham}}{=} \sum_{k=0}^{\dim M} (-1)^k \dim \H_{\dR}^k(M)
=\chi(M).
\end{align*}
\end{proof}

\section{Index of the signature operator}
\label{Section-Index-signature-operator}

In this section, we turn to the index problem associated with the signature operator. If $M$ is oriented and if the dimension of $M$ is even, the Hodge star operator gives rise to a second natural grading of $\L^p(\Omega^\bullet(M))$, and we show that it defines an even compact Banach spectral triple. We then identify the Fredholm index of the corresponding graded component of the $\L^p$-Hodge--Dirac operator and prove that, in dimension divisible by four, this index coincides with the signature of the manifold.

\begin{thm}
\label{thm:signature-Lp}
Let $(M,g)$ be a smooth compact oriented Riemannian manifold of even dimension $d=2\ell$. Suppose that $1 < p < \infty$.  Then the maps $\Omega^k(M) \to \Omega^{d-k}(M)$, $\omega \mapsto \i^{k(k-1)+\ell} * \omega$, where $k \in \{0,\ldots,d\}$, induce an isometric isomorphism $\tau \co \L^p(\Omega^\bullet(M)) \to \L^p(\Omega^\bullet(M))$ satisfying $\tau^2=\Id_{\L^p(\Omega^\bullet(M))}$. Moreover, the quadruple $(\C(M),\L^p(\Omega^\bullet(M)),D_p,\tau)$ is an even compact Banach spectral triple.
\end{thm}

\begin{proof}
By Theorem~\ref{thm:compact-Hodge-Dirac}, the triple $(\C(M),\L^p(\Omega^\bullet(M)),D_p)$ is a compact Banach spectral triple. We define an operator $\tau \co \Omega^\bullet(M) \to \Omega^\bullet(M)$ by setting, for each $k$-form $\omega \in \Omega^k(M)$,
\begin{equation}
\label{eq:def-tau}
\tau\omega \ov{\mathrm{def}}{=} \i^{k(k-1)+\ell} \ast \omega.
\end{equation}
We first compute $\tau^2$. Let $\omega \in \Omega^k(M)$. Then $\ast \omega$ has degree $d-k$, so applying \eqref{eq:def-tau} to $\ast \omega$ gives
\begin{equation}
\label{def-invol-4l}
\tau(\ast\omega)
=\i^{(d-k)(d-k-1)+\ell} \ast^2\omega.
\end{equation}
Using \eqref{eq:star-square}, we obtain
\begin{align*}
\MoveEqLeft
\tau^2\omega
=\tau(\tau\omega) 
\ov{\eqref{eq:def-tau}}{=}\i^{k(k-1)+\ell} \tau(\ast\omega) 
\ov{\eqref{def-invol-4l}}{=} \i^{k(k-1)+\ell}\, \i^{(d-k)(d-k-1)+\ell} \ast^2\omega.
\end{align*}
A straightforward computation gives
\begin{align*}
\MoveEqLeft
k(k-1) + \ell + (d-k)(d-k-1) + \ell \\
&= k(k-1) + (d-k)(d-k-1) + d 
= 2k^2 - 2dk + d^2.
\end{align*}
Thus by \eqref{eq:star-square} and since $\i^2=-1$ we have
\[
\tau^2\omega
=\i^{2k^2 - 2dk + d^2} (-1)^{k(d-k)} \omega
=\i^{2k^2 - 2dk + d^2 + 2k(d-k)} \omega
=\i^{d^2} \omega.
\]
Since $d=2\ell$, we have $d^2 = 4\ell^2$. Hence $
\i^{d^2} 
= \i^{4\ell^2} 
= (\i^4)^{\ell^2} = 1$ and consequently $\tau^2 = \Id_{\Omega^\bullet(M)}$. If $\omega \in \Omega^\bullet(M)$, we can write $\omega = \sum_{k=0}^d \omega_k$ with $\omega_k \in \Omega^k(M)$. Observe that
\begin{align*}
\norm{\tau \omega}_{\L^p(\Omega^\bullet(M))}
& \ov{\eqref{eq:def-tau}}{=} \bigg(\int_M \bigg(\sum_{k=0}^d |\i^{k(k-1)+\ell} \ast \omega_k(x)|_{\Lambda_{\mathbb{C}}^{d-k} \mathrm{T}_x^*M}^2\bigg)^{\frac{p}{2}} \d\mu_g(x)\bigg)^{\frac1p}
\\
&\ov{\eqref{Hodge-norm}}{=}  \bigg(\int_M \bigg(\sum_{k=0}^d |\omega_k(x)|_{\Lambda_{\mathbb{C}}^k \mathrm{T}_x^*M}^2\bigg)^{\frac{p}{2}} \d\mu_g(x)\bigg)^{\frac1p}
=\norm{\omega}_{\L^p(\Omega^\bullet(M))}.
\end{align*}
Since the subspace $\Omega^\bullet(M)$ is dense in the Banach space $\L^p(\Omega^\bullet(M))$, the operator $\tau \co \Omega^\bullet(M) \to \Omega^\bullet(M)$ clearly extends by continuity to an isometric isomorphism $\tau \co \L^p(\Omega^\bullet(M)) \to \L^p(\Omega^\bullet(M))$ satisfying $\tau^2=\Id_{\L^p(\Omega^\bullet(M))}$. We will prove that on the space $\Omega^\bullet(M)$ we have the anticommutation relations
\begin{equation}
\label{eq:tau-d-dstar}
\tau \d + \d^* \tau = 0
\quad \text{and} \quad
\tau \d^* + \d \tau = 0.
\end{equation}
We prove the first equality. Let $\omega \in \Omega^k(M)$ with $0 \leq k \leq d$. On the one hand, we have
\begin{equation}
\label{inter-AAZRR}
\tau(\d\omega)
\ov{\eqref{eq:def-tau}}{=} 
\i^{(k+1)k+\ell}\ast(\d\omega).
\end{equation}
On the other hand, $\tau\omega$ has degree $d-k$, so by the definition of the codifferential on $(d-k)$-forms we obtain
\begin{align*}
\d^*(\tau\omega)
&\ov{\eqref{d-star-Omega-k}}{=} (-1)^{d(d-k+1)+1}\ast\big(\d(\ast(\tau\omega))\big).
\end{align*}
Using again \eqref{eq:def-tau} and \eqref{eq:star-square}, we compute
\[
\ast(\tau\omega)
\ov{\eqref{eq:def-tau}}{=} 
\ast\bigl(\i^{k(k-1)+\ell}\ast\omega\bigr)
=\i^{k(k-1)+\ell}\ast^2\omega
\ov{\eqref{eq:star-square}}{=} \i^{k(k-1)+\ell}(-1)^{k(d-k)}\omega.
\]
Hence
\begin{align}
\label{inter-56}
\d^*(\tau\omega)
&=
(-1)^{d(d-k+1)+1}\ast\Big(\d\big(\i^{k(k-1)+\ell}(-1)^{k(d-k)}\omega\big)\Big)\\
&=
(-1)^{d(d-k+1)+1}\, \i^{k(k-1)+\ell}(-1)^{k(d-k)}\ast(\d\omega). \nonumber
\end{align}
Putting these formulas together, we obtain
\begin{align*}
(\tau\d + \d^*\tau)\omega
&\ov{\eqref{inter-AAZRR} \eqref{inter-56}}{=} 
\Bigl(\i^{(k+1)k+\ell}
+
(-1)^{d(d-k+1)+1+k(d-k)}\, \i^{k(k-1)+\ell}\Bigr)\ast(\d\omega)\\
&=
\i^{k(k-1)+\ell}\Bigl(\i^{2k}
+
(-1)^{d(d-k+1)+1+k(d-k)}\Bigr)\ast(\d\omega).
\end{align*}
Set
\[
A_k
\ov{\mathrm{def}}{=}
d(d-k+1)+1+k(d-k)
=
d^2+d+1-k^2.
\]
Since $d=2\ell$ is even, the parity of $A_k+k$ is the same as that of $1+k-k^2$. Modulo $2$ we have $k^2\equiv k$, so
\[
A_k+k\equiv 1+k-k^2\equiv 1 \pmod 2.
\]
Therefore $
(-1)^{A_k}
=(-1)^{A_k+k}(-1)^{-k}
=-(-1)^k$. Using $\i^{2k} = (-1)^k$, we deduce
\[
\i^{2k}
+
(-1)^{A_k}
=
(-1)^k-(-1)^k
=
0.
\]
It follows that $(\tau\d + \d^*\tau)\omega = 0$. Hence $\tau\d + \d^*\tau = 0$ on $\Omega^\bullet(M)$. From this identity, we see that $\tau\d = -\d^*\tau$. Multiplying on the left by $\tau$ and using $\tau^2=\Id$ from the first part of the proof, we obtain $\d = -\tau\d^*\tau$. Multiplying this identity on the right by $\tau$ and using again $\tau^2=\Id$ yields $
\d\tau = -\tau\d^*$ which is exactly the second relation in \eqref{eq:tau-d-dstar}.

Adding the two identities in \eqref{eq:tau-d-dstar}, we obtain $\tau D + D \tau = 0$ on the space $\Omega^\bullet(M)$. Since $\tau$ is bounded on $\L^p(\Omega^\bullet(M))$, the anticommutation relation extends from $\Omega^\bullet(M)$ to the domain of the closed operator $D_p$, and we obtain $\tau D_p + D_p\tau = 0$ on the subspace $\dom D_p$.

Let $f \in \C^\infty(M)$. Fix $k \in \{0,\dots,d\}$ and $\omega \in \Omega^k(M)$. We have
\[
\tau(\pi(f)\omega)
\ov{\eqref{def-de-pi}}{=} \tau(f\omega)
\ov{\eqref{eq:def-tau}}{=} \i^{k(k-1)+\ell}\ast(f\omega)
\ov{\eqref{*-et-Mf}}{=} \i^{k(k-1)+\ell}f \ast \omega
\ov{\eqref{eq:def-tau}}{=} f\tau\omega
\ov{\eqref{def-de-pi}}{=} \pi(f)(\tau\omega).
\]
By linearity, this proves $\tau\pi(f) = \pi(f)\tau$ on $\Omega^\bullet(M)$. Since the subspace $\Omega^\bullet(M)$ is dense in $\L^p(\Omega^\bullet(M))$ and since $\tau$ extends to a bounded operator on $\L^p(\Omega^\bullet(M))$, it follows by continuity that $\tau\pi(f) = \pi(f)\tau$ holds on $\L^p(\Omega^\bullet(M))$ for all $f \in \C^\infty(M)$. Now, let $f \in \C(M)$. There exists a sequence $(f_n)$ in $\C^\infty(M)$ such that $\norm{f_n-f}_{\L^\infty(M)}\to 0$. Since $\pi$ is the representation by multiplication operators, we have
$
\norm{\pi(f_n)-\pi(f)}_{\B(\L^p(\Omega^\bullet(M)))}
= \norm{f_n-f}_{\L^\infty(M)} \to 0$. Hence, using $\tau\pi(f_n)=\pi(f_n)\tau$, we obtain
\begin{align*}
\norm{\tau\pi(f)-\pi(f)\tau}
\leq
\norm{\tau(\pi(f)-\pi(f_n))}
+
\norm{(\pi(f_n)-\pi(f))\tau}
\leq
2\norm{\tau}\norm{f_n-f}_{\L^\infty(M)} \to 0.
\end{align*}
Consequently, we have $
\tau\pi(f)=\pi(f)\tau$ for any $f \in \C(M)$.
\end{proof}

The involution $\tau$ yields a $\Z_2$-grading of $\L^p(\Omega^\bullet(M))$ by its $\pm 1$ eigenspaces:
\begin{equation}
\label{decompo-OBBG}
\L^p(\Omega^\bullet(M))
= \L^p(\Omega^+(M)) \oplus \L^p(\Omega^-(M)),
\end{equation}
where
\[
\L^p(\Omega^+(M))
\ov{\mathrm{def}}{=} \{\omega \in \L^p(\Omega^\bullet(M)) \co \tau\omega =  \omega\}
\quad \text{and} \quad
\L^p(\Omega^-(M))
\ov{\mathrm{def}}{=} \{\omega \in \L^p(\Omega^\bullet(M)) \co \tau\omega = - \omega\}.
\]
Since $\tau D_p =- D_p\tau$, with respect to this decomposition, the operator $D_p$ admits by Proposition \ref{prop-anticommuting-symmetry-unbounded} the block matrix form
\begin{equation}
\label{final-39}
D_p 
= \begin{bmatrix}
  0   & \cal{D}_{p,-} \\
 \cal{D}_{p,+}  &  0 \\
\end{bmatrix},
\end{equation} 
where the operators $
\cal{D}_{p,+} \co \dom \cal{D}_{p,+} \subset \L^p(\Omega^{+}(M)) \to \L^p(\Omega^{-}(M))$ and $
\cal{D}_{p,-} \co \dom \cal{D}_{p,-} \subset \L^p(\Omega^{-}(M)) \to \L^p(\Omega^{+}(M))$ are the restrictions of the Hodge--Dirac operator to subspaces. The operator $\cal{D}_{2,+}$ is sometimes called the signature operator. 
Recall that we denote by $P_p \co \L^p(\Omega^\bullet(M)) \to \L^p(\Omega^\bullet(M))$ the harmonic projection onto the subspace $\cal{H}(M)$ of harmonic forms.

\begin{lemma}
\label{lem:tau-commutes-S}
Let $(M,g)$ be a smooth compact oriented Riemannian manifold of even dimension $d=2\ell$. Suppose that $1 < p < \infty$.  Then the operator $\scr{D}\ov{\mathrm{def}}{=}
|D_p| + P_p$ commutes with $\tau$, i.e., $\tau \scr{D} = \scr{D} \tau$. In particular, $\scr{D}$ is even with respect to the $\Z_2$-grading defined by $\tau$.
\end{lemma}

\begin{proof}
On the space $\Omega^\bullet(M)$, we compute
\begin{align*}
\MoveEqLeft
\tau \Delta_{\HdR}
\ov{\eqref{Hodge-deRham-Laplacian}}{=} (\tau \d)\d^* + (\tau \d^*)\d
\ov{\eqref{eq:tau-d-dstar}}{=} (-\d^*\tau)\d^* + (-\d\tau)\d
=-\d^*(\tau\d^*) - \d(\tau\d) \\
&\ov{\eqref{eq:tau-d-dstar}}{=} \d^*\d\tau + \d\d^*\tau
\ov{\eqref{Hodge-deRham-Laplacian}}{=} \Delta_{\HdR} \tau.
\end{align*}
Since the subspace $\Omega^\bullet(M)$ is a core of $\Delta_{\HdR,p}$, we obtain $
\tau \Delta_{\HdR,p} = \Delta_{\HdR,p}\tau$. Hence $\tau$ commutes with $D_p^2= \Delta_{\HdR,p}$. We deduce that $\tau$ commutes with $|D_p| = (D_p^2)^{\frac12}$, i.e., $\tau |D_p| = |D_p|\tau$.

Since $\tau$ commutes with the operator $\Delta_{\HdR,p}$, it preserves the subspace $\ker \Delta_{\HdR,p} = \cal{H}(M)$. By \eqref{eq:tau-d-dstar}, it also preserves the complementary subspace $\d \W^{1,p}(\Omega^\bullet(M)) \oplus \d^* \W^{1,p}(\Omega^\bullet(M))$. Consequently, by \cite[Theorem 2.22 p.~81]{AbA02}, $\tau$ commutes with the projection $P_p$ on $\cal{H}(M)$, i.e., we have $\tau P_p = P_p \tau$.  Combining the two commutation relations, we obtain the equality $\tau \scr{D}=\scr{D}\tau$.  
This shows that $\scr{D}$ is even with respect to the grading defined by $\tau$.
\end{proof}

We will use the space $\W^{1,p}(\Omega^+(M)) \ov{\mathrm{def}}{=} \W^{1,p}(\Omega^\bullet(M)) \cap \L^p(\Omega^+(M))$, equipped with the norm induced by the one of $\W^{1,p}(\Omega^\bullet(M))$. Using the same method as in the proof of Theorem \ref{thm:Euler-characteristic-Lp}, we obtain the following result.

\begin{thm}
\label{thm:signature-characteristic-Lp}
Let $(M,g)$ be a smooth compact oriented Riemannian manifold of even dimension. Suppose that $1 < p < \infty$. Let $[1] \in \K_0(\C(M))$ be the $\K$-theory class of the constant function $1$. Then the linear operator $\cal{D}_{p,+} \co \W^{1,p}(\Omega^+(M)) \to \L^p(\Omega^-(M))$ is Fredholm and we have
\[
\big\la [1],[\L^p(\Omega^\bullet(M)),\pi,\sgn D_p,\tau] \big\ra
= \Index \cal{D}_{p,+}.
\]
\end{thm}

\begin{proof}
The proof is identical to that of Theorem \ref{thm:Euler-characteristic-Lp}, replacing the grading operator $\gamma$ by $\tau$ and using Lemma \ref{lem:tau-commutes-S} instead of the fact that $\scr D$ is even with respect to the parity grading.
\end{proof}

Now, assume that the dimension $d$ of the Riemannian oriented compact manifold $M$ is divisible by four, i.e., $d=4j$ and $\ell=2j$ for some integer $j \geq 1$. Observe that in this case, we have $\tau\omega=*\omega$ for any differential form $\omega$ belonging to the space $\Omega^{2j}(M)$ of middle degree forms since
\begin{equation*}
\tau\omega
\ov{\eqref{eq:def-tau}}{=}
\i^{k(k-1)+\ell}\ast\omega
=\i^{2j(2j-1)+2j}\ast\omega
=\i^{4j^2}\ast\omega
=\ast\omega.
\end{equation*}

Consequently, the Hodge star operator $\ast\big|_{\Omega^{2j}(M)} \co \Omega^{2j}(M) \to \Omega^{2j}(M)$ is an involution on the space $\Omega^{2j}(M)$ of $2 j$-forms, i.e.,  
\begin{equation}
\label{ast=Id}
\ast^2\big|_{\Omega^{2j}(M)}=\Id_{\Omega^{2 j}(M)}.
\end{equation}
Since $*\Delta_{\HdR}=\Delta_{\HdR}*$, see \cite[(4) p.~221]{War83}, and since $\tau$ is a scalar multiple of $*$ on each homogeneous degree, for every integer $k \in \{0,\ldots,d\}$ the map $\tau$ induces an isomorphism
\begin{equation}
\label{tau-invol-bis}
\tau|_{\cal{H}^k(M)} \co \cal{H}^k(M) \to \cal{H}^{d-k}(M), 
\end{equation}
satisfying $\tau|_{\cal{H}^{d-k}(M)} \circ \tau|_{\cal{H}^k(M)}= \Id_{\cal{H}^k(M)}$. 
In particular, in middle degree the space $\cal{H}^{2j}(M)$ is invariant under $\tau$, and since $\tau=*$ on $\Omega^{2j}(M)$, we may define
\begin{equation}
\label{def-middle-harmonic-plus-minus}
\cal{H}^{2j}_+(M)
\ov{\mathrm{def}}{=}
\{\omega\in\cal{H}^{2j}(M):*\omega=\omega\}
\quad \text{and} \quad
\cal{H}^{2j}_-(M)
\ov{\mathrm{def}}{=}
\{\omega\in\cal{H}^{2j}(M):*\omega=-\omega\}.
\end{equation}
We have
\begin{equation}
\label{middle-harmonic-decomposition}
\cal{H}^{2j}(M)
=
\cal{H}^{2j}_+(M)\oplus\cal{H}^{2j}_-(M).
\end{equation}
Now, we recall the relation with the signature. Let $\cal{H}^{2j}(M,\R)$ denote the space of real harmonic $2j$-forms and set
\[
\cal{H}^{2j}_+(M,\R)
\ov{\mathrm{def}}{=}
\{\omega \in \cal{H}^{2j}(M,\R) : *\omega = \omega\} 
\text{ and }
\cal{H}^{2j}_-(M,\R)
\ov{\mathrm{def}}{=}
\{\omega \in \cal{H}^{2j}(M,\R) : *\omega = -\omega\}.
\]
Since the Hodge star is the complexification of its restriction to real differential forms, we have
\begin{equation}
\label{complexification-middle-eigenspaces}
\cal{H}^{2j}_+(M)
=
\cal{H}^{2j}_+(M,\R) \ot_{\R} \mathbb{C}
\quad \text{and} \quad
\cal{H}^{2j}_-(M)
=
\cal{H}^{2j}_-(M,\R) \ot_{\R} \mathbb{C}.
\end{equation}
By the Hodge theorem, described in \eqref{harmonic-space-and-de-Rham} in the complex case, we can identify de Rham cohomology group $\H^{2j}_{\dR}(M,\R)$ with the space $\cal{H}^{2j}(M,\R)$ of harmonic forms. According to \cite[p.~36]{Dieu82}, the intersection form
\begin{equation}
\label{intersection-form}
Q_M
\co
\H^{2j}_{\dR}(M,\R)
\times
\H^{2j}_{\dR}(M,\R)
\to \R,
\qquad
([\alpha],[\beta])
\mapsto
\int_M\alpha\wedge\beta
\end{equation}
is a nondegenerate symmetric bilinear form. If $\alpha,\beta \in \cal{H}^{2j}(M,\R)$, then
\begin{equation}
\label{sign-as-scalar}
Q_M([\alpha],[\beta])
\ov{\eqref{intersection-form}}{=}
\int_M\alpha\wedge\beta
\ov{\eqref{eq:def-Hodge-star}\eqref{ast=Id}}{=}
\int_M\la\alpha,*\beta\ra\vol_g
=
\la\alpha,*\beta\ra_{\L^2(\Omega^{2j}(M))}.
\end{equation}
If $\omega \in \cal{H}^{2j}(M,\R)$, by \eqref{middle-harmonic-decomposition}, we can write $\omega=\omega_++\omega_-$ with $\omega_+ \in \cal{H}^{2j}_+(M,\R)$ and $\omega_- \in \cal{H}^{2j}_-(M,\R)$. Then $*\omega=\omega_+-\omega_-$ and the two eigenspaces $\cal{H}^{2j}_+(M,\R)$ and $\cal{H}^{2j}_-(M,\R)$ are orthogonal. Consequently, we have
\begin{equation}
\label{eq:intersection-form-self-dual}
Q_M([\omega],[\omega])
\ov{\eqref{sign-as-scalar}}{=}
\la \omega,*\omega \ra_{\L^2(\Omega^{2j}(M))}
=
\norm{\omega_+}_{\L^2(\Omega^{2j}(M))}^2
-
\norm{\omega_-}_{\L^2(\Omega^{2j}(M))}^2.
\end{equation}
Hence $Q_M$ is positive definite on $\cal{H}^{2j}_+(M,\R)$ and negative definite on $\cal{H}^{2j}_-(M,\R)$. Therefore its signature (the difference of the number of positive and negative numbers in the diagonal of the matrix of $Q_M$ with respect to an orthogonal basis), called the signature of $M$ and denoted by $\sign(M)$, satisfies
\begin{align}
\label{signature}
\MoveEqLeft
\sign(M)
=
\dim_{\R}\cal{H}^{2j}_+(M,\R)
-
\dim_{\R} \cal{H}^{2j}_-(M,\R) 
\ov{\eqref{complexification-middle-eigenspaces}}{=}
\dim_{\mathbb{C}}\cal{H}^{2j}_+(M)
-
\dim_{\mathbb{C}}\cal{H}^{2j}_-(M).
\end{align}
For the remaining degrees, we cannot decompose each space $\cal{H}^k(M)$ into the eigenspaces of $\tau$, since $\tau$ maps $\cal{H}^k(M)$ onto $\cal{H}^{d-k}(M)$. Instead, for every integer $k\in\{0,\ldots,2j-1\}$, we consider the finite-dimensional $\tau$-invariant subspace
\begin{equation}
\label{def-Ek-space}
E^k
\ov{\mathrm{def}}{=}
\cal{H}^k(M)\oplus\cal{H}^{d-k}(M)
\end{equation}
of $\ker D_p$. By \eqref{tau-invol-bis}, we have $\tau(E^k)=E^k$. Therefore, setting
\begin{equation}
\label{def-Ek}
E^k_+
\ov{\mathrm{def}}{=}
\tfrac{\Id+\tau}{2}(E^k),
\qquad
E^k_-
\ov{\mathrm{def}}{=}
\tfrac{\Id-\tau}{2}(E^k),
\end{equation}
we obtain the direct sum decomposition $
E^k=E^k_+\oplus E^k_-$. For any integer $k \in \{0,\ldots,2j-1\}$, it is known \cite[p.~137]{LaM89} that
\begin{equation}
\label{dim-Ek+}
\dim E^k_+
=\dim E^k_-
=\dim \cal{H}^k(M).
\end{equation}
Finally, since
\[
\ker D_p
\ov{\eqref{Bisec-Ran-Ker}}{=} \ker D_p^2
\ov{\eqref{eq:Hodge-square-Dirac}}{=} \ker \Delta_{\HdR,p}
\ov{\eqref{ker-Delta-p}}{=} \cal{H}(M)
=\left(\bigoplus_{k=0}^{2j-1}E^k\right)
\oplus
\cal{H}^{2j}(M),
\]
and since all the summands in this decomposition are invariant under $\tau$, the $(+1)$-eigenspace of $\tau$ in $\ker D_p$ is
\begin{equation}
\label{Ker-D+}
\ker\cal{D}_{p,+}
\ov{\eqref{final-39}}{=} \ker D_p\cap\L^p(\Omega^+(M))
= E^0_+\oplus E^1_+\oplus\cdots\oplus E^{2j-1}_+\oplus\cal{H}^{2j}_+(M).
\end{equation}
Similarly, we have
\begin{equation}
\label{ker-D-}
\ker\cal{D}_{p,-}
\ov{\eqref{final-39}}{=} \ker D_p\cap\L^p(\Omega^-(M))
=E^0_-\oplus E^1_-\oplus\cdots\oplus E^{2j-1}_-\oplus\cal{H}^{2j}_-(M),
\end{equation}
which is the $(-1)$-eigenspace of $\tau$ in $\ker D_p$.

\begin{prop}
\label{prop-coker-signature}
Let $(M,g)$ be a smooth compact oriented Riemannian manifold of dimension divisible by four, say $\dim M=4j$. Suppose that $1 < p < \infty$. Then we have an isomorphism
\begin{equation}
\label{eq:coker-signature}
\coker \cal{D}_{p,+}
\cong
E_-^0 \oplus E_-^1 \oplus \cdots \oplus E_-^{2j-1} \oplus \cal{H}^{2j}_-(M).
\end{equation}
\end{prop}

\begin{proof}
By Theorem \ref{thm:signature-characteristic-Lp}, the operator $
\cal{D}_{p,+} \co \dom \cal{D}_{p,+} \subset \L^p(\Omega^+(M)) \to \L^p(\Omega^-(M))$ is Fredholm. By Lemma \ref{lem:tau-commutes-S}, the operator $
\scr{D}
\ov{\mathrm{def}}{=}
|D_p|+P_p$ commutes with $\tau$, hence, by Proposition \ref{prop:commuting-symmetry-diagonal}, it is even with respect to the decomposition \eqref{decompo-OBBG}. Therefore, it has a block diagonal form
\[
\scr{D}
=
\begin{bmatrix}
\scr{D}_+ & 0\\
0 & \scr{D}_-
\end{bmatrix}.
\]
Moreover, by Proposition \ref{prop-isomorphism}, the operator $
\scr{D} \co \W^{1,p}(\Omega^\bullet(M)) \xrightarrow{\ \cong\ } \L^p(\Omega^\bullet(M))$ is an isomorphism of Banach spaces, so in particular $
\scr{D}_+ \co \W^{1,p}(\Omega^+(M)) \xrightarrow{\ \cong\ } \L^p(\Omega^+(M))$ is an isomorphism.

By Theorem \ref{thm:signature-Lp}, we have $\tau D_{p}=-D_{p}\tau$. Since the sign function is odd, we obtain by \cite[Proposition 3.2.10]{Ege15} the equality that the operator $\sgn D_p$ anticommutes with $\tau$. So it has by Proposition \ref{prop-anticommuting-symmetry-unbounded} an off-diagonal form
\[
\sgn D_p
=\begin{bmatrix}
0 & (\sgn D_p)_-\\
(\sgn D_p)_+ & 0
\end{bmatrix}.
\]
As in the proof of Theorem \ref{thm:Euler-characteristic-Lp}, we have
\[
D_p
=
\sgn(D_p)\scr{D}.
\]
Passing to the block decomposition associated with the decomposition $\L^p(\Omega^+(M)) \oplus \L^p(\Omega^-(M))$, we obtain
\begin{equation}
\label{eq:factorization-calDp-plus}
\cal{D}_{p,+}
=
(\sgn D_p)_+ \scr{D}_+.
\end{equation}
Since $\scr{D}_+$ is an isomorphism, it follows that
\begin{equation}
\label{eq:coker-reduction-sign}
\coker \cal{D}_{p,+}
\cong
\coker (\sgn D_p)_+.
\end{equation}
Now, by \eqref{decomposition-de-X}, applied to the bisectorial operator $D_p$, we have
\[
\L^p(\Omega^\bullet(M))
=
\ker D_p \oplus \ovl{\Ran D_p}.
\]
Since $\tau D_p = -D_p\tau$, the operator $\tau$ preserves both $\ker D_p$ and $\ovl{\Ran D_p}$. Hence
\begin{equation}
\label{eq:decomp-signature-plus}
\L^p(\Omega^+(M))
=
(\ker D_p \cap \L^p(\Omega^+(M))) \oplus (\ovl{\Ran D_p} \cap \L^p(\Omega^+(M)))
\end{equation}
and
\begin{equation}
\label{eq:decomp-signature-minus}
\L^p(\Omega^-(M))
=
(\ker D_p \cap \L^p(\Omega^-(M))) \oplus (\ovl{\Ran D_p} \cap \L^p(\Omega^-(M))).
\end{equation}

By construction, the operator $\sgn D_p$ vanishes on the subspace $\ker D_p$. Moreover, on the subspace $\ovl{\Ran D_p}$, the operator $\sgn D_p$ is defined from the injective part of $D_p$, so it is an involution there. Since it anticommutes with $\tau$, its off-diagonal part
\[
(\sgn D_p)_+ \co \ovl{\Ran D_p} \cap \L^p(\Omega^+(M)) \to \ovl{\Ran D_p} \cap \L^p(\Omega^-(M))
\]
is an isomorphism with inverse the restriction of $(\sgn D_p)_-$.

Therefore, with respect to the decompositions \eqref{eq:decomp-signature-plus} and \eqref{eq:decomp-signature-minus}, the operator $(\sgn D_p)_+$ has the form
\[
(\sgn D_p)_+
=
\begin{bmatrix}
0 & 0\\
0 & U
\end{bmatrix},
\]
where $
U \co \ovl{\Ran D_p} \cap \L^p(\Omega^+(M)) \xrightarrow{\ \cong\ } \ovl{\Ran D_p} \cap \L^p(\Omega^-(M))$ is an isomorphism. It follows that
\begin{equation}
\label{inter-245}
\Ran (\sgn D_p)_+
=
\ovl{\Ran D_p} \cap \L^p(\Omega^-(M)).
\end{equation}
Consequently, we have
\begin{align*}
\MoveEqLeft
\coker \cal{D}_{p,+}
\ov{\eqref{eq:coker-reduction-sign}}{\cong} \coker (\sgn D_p)_+
=\L^p(\Omega^-(M)) / \Ran (\sgn D_p)_+ \\
&\ov{\eqref{inter-245}}{\cong} 
\L^p(\Omega^-(M)) / (\ovl{\Ran D_p} \cap \L^p(\Omega^-(M)))
\ov{\eqref{eq:decomp-signature-minus}}{\cong} 
\ker D_p \cap \L^p(\Omega^-(M)).         
\end{align*}
Finally, we obtain that
\[
\ker D_p \cap \L^p(\Omega^-(M))
=
\ker \cal{D}_{p,-}
\ov{\eqref{ker-D-}}{=} 
E_-^0 \oplus E_-^1 \oplus \cdots \oplus E_-^{2j-1} \oplus \cal{H}^{2j}_-(M).
\]
This proves \eqref{eq:coker-signature}.
\end{proof}

Now, we can state the second index theorem of this paper.

\begin{thm}
\label{thm-signature}
Let $(M,g)$ be a smooth compact oriented Riemannian manifold of dimension $d=4j$ divisible by four. Suppose that $1 < p < \infty$.  The index of the Fredholm operator $\cal{D}_{p,+}$ is 
\begin{equation}
\Index \cal{D}_{p,+}
=\dim \cal{H}^{2j}_+(M) - \dim \cal{H}^{2j}_-(M)
=\sign(M).
\end{equation}
\end{thm}

\begin{proof}
By Theorem \ref{thm:signature-characteristic-Lp}, the operator $\cal{D}_{p,+} \co \W^{1,p}(\Omega^+(M)) \to \L^p(\Omega^-(M))$ is Fredholm. Therefore,
\begin{equation}
\label{eq:index-signature-coker}
\Index \cal{D}_{p,+}
\ov{\eqref{Fredholm-Index}}{=}
\dim \ker \cal{D}_{p,+}
-
\dim \coker \cal{D}_{p,+}.
\end{equation}
By \eqref{Ker-D+}, we have $
\ker \cal{D}_{p,+}
=
E_+^0 \oplus E_+^1 \oplus \cdots \oplus E_+^{2j-1} \oplus \cal{H}^{2j}_+(M)$. Moreover, by Proposition \ref{prop-coker-signature}, we have $\coker \cal{D}_{p,+}
\cong
E_-^0 \oplus E_-^1 \oplus \cdots \oplus E_-^{2j-1} \oplus \cal{H}^{2j}_-(M)$. Substituting these two descriptions into \eqref{eq:index-signature-coker}, we obtain
\begin{align*}
\MoveEqLeft
\Index \cal{D}_{p,+} 
=\dim E_+^0 + \dim E_+^1 + \cdots + \dim E_+^{2j-1} + \dim \cal{H}^{2j}_+(M) \\
&\quad
-\dim E_-^0 - \dim E_-^1 - \cdots - \dim E_-^{2j-1} - \dim \cal{H}^{2j}_-(M).
\end{align*}
Using \eqref{dim-Ek+}, we get the equality $\dim E_+^k = \dim E_-^k$ for any integer $k \in \{0,\dots,2j-1\}$. Hence all the terms corresponding to the spaces $E_\pm^k$ cancel, and we are left with
\[
\Index \cal{D}_{p,+}
=
\dim \cal{H}^{2j}_+(M) - \dim \cal{H}^{2j}_-(M)
\ov{\eqref{signature}}{=} \sign(M).
\]
\end{proof}

\section{Appendix: structure of operators anticommuting with a bounded symmetry}
\label{Appendix}

The following elementary result describes the structure of operators that anticommute with a bounded symmetry.

\begin{prop}
\label{prop-anticommuting-symmetry-unbounded}
Let $X$ be a Banach space and let $\gamma \co X \to X$ be a bounded involution, i.e., $\gamma^2 = \Id_X$. Define the bounded projections $P \ov{\mathrm{def}}{=} \frac{\Id_X + \gamma}{2}$ and $Q \ov{\mathrm{def}}{=} \frac{\Id_X - \gamma}{2}$. 
Set $X_+ \ov{\mathrm{def}}{=} \Ran P = \{x \in X : \gamma(x) = x\}$ and $X_- \ov{\mathrm{def}}{=} \Ran Q = \{x \in X : \gamma(x) = -x\}$. Then $X = X_+ \oplus X_-$. Let $T$ be an unbounded closed linear operator on $X$ with dense domain $\dom T$. Assume that $\gamma(\dom T) \subset \dom T$ and that $T\gamma x = - \gamma Tx$ for any $x \in \dom T$. Then
\[
T(\dom T \cap X_+) \subset X_-,
\qquad
T(\dom T \cap X_-) \subset X_+.
\]
If we define the operators
\[
T_+ \ov{\mathrm{def}}{=} T\big|_{\dom T \cap X_+} \co \dom T \cap X_+ \to X_-,
\qquad
T_- \ov{\mathrm{def}}{=} T\big|_{\dom T \cap X_-} \co \dom T \cap X_- \to X_+,
\]
then for every $x = x_+ + x_- \in \dom T$ with $x_\pm \in \dom T \cap X_\pm$ we have
\[
Tx 
= T_+ x_+ + T_- x_-,
\]
so that, relative to the decomposition $X = X_+ \oplus X_-$, the operator $T$ has off-diagonal form $
T
=
\begin{bmatrix}
0 & T_-\\
T_+ & 0
\end{bmatrix}$ 
in the previous sense.
\end{prop}

\begin{proof}
Since $\gamma^2 = \Id_X$, a direct computation \cite[Exercise 18 p.~22]{Car05} shows that $P$ and $Q$ are bounded projections on $X$ with
\[
P^2 = P, \quad Q^2 = Q, \quad PQ = QP = 0, \quad P + Q = \Id_X.
\]
Moreover $X_+ = \Ran P$ and $X_- = \Ran Q$, and any $x \in X$ decomposes uniquely as
\[
x = Px + Qx,
\qquad
Px \in X_+, Qx \in X_-.
\]
Thus $X = X_+ \oplus X_-$. Since $\gamma(\dom T) \subset \dom T$ and $\dom T$ is a linear subspace, it follows that
\[
P(\dom T)
=\tfrac12(\Id_X + \gamma)(\dom T) \subset \dom T,
\qquad
Q(\dom T) 
=\tfrac12(\Id_X - \gamma)(\dom T)
\subset \dom T.
\]
Hence for every $x \in \dom T$ we have $
x = Px + Qx$ with $Px, Qx \in \dom T$, $Px \in X_+$ and $Qx \in X_-$. It follows that
\[
\dom T
=
(\dom T \cap X_+)
+
(\dom T \cap X_-).
\]
The sum is direct: if $x \in \dom T \cap X_+ \cap X_-$ then $\gamma x = x$ and $\gamma x = -x$, hence $x=0$. Thus
\begin{equation}
\label{decompo-directe}
\dom T
=
(\dom T \cap X_+)
\oplus
(\dom T \cap X_-).
\end{equation}
Assume now that $T \gamma x = -\gamma Tx$ for all $x \in \dom T$. Let $x_+ \in \dom T \cap X_+$. Then $\gamma x_+ = x_+$. Applying the relation to $x_+$ yields
\[
T\gamma x_+ = -\gamma Tx_+.
\]
Again, $T \gamma x_+ = Tx_+$, so $
Tx_+ = -\gamma Tx_+$, 
or equivalently $\gamma Tx_+ = -Tx_+$. Thus $T x_+$ is a $(-1)$-eigenvector of $\gamma$, and therefore $Tx_+ \in X_-$. Hence 
$T(\dom T \cap X_+) \subset X_-$. 
Similarly, if $x_- \in \dom T \cap X_-$ then $\gamma x_- = -x_-$. Applying the relation to $x_-$ gives $
T\gamma x_- = -\gamma Tx_-$. The left-hand side is $T \gamma x_- = T(-x_-) = -Tx_-$, so $
- Tx_- = -\gamma T x_-$. Hence $\gamma Tx_- = Tx_-$. Thus $Tx_-$ is a $(+1)$-eigenvector of $\gamma$. So $Tx_- \in X_+$. Hence $T(\dom T \cap X_-) \subset X_+$. Now, we introduce the operators
\[
T_- 
\ov{\mathrm{def}}{=} T\big|_{\dom T \cap X_-} \co \dom T \cap X_- \to X_+
\quad \text{and} \quad
T_+ 
\ov{\mathrm{def}}{=} T\big|_{\dom T \cap X_+} \co \dom T \cap X_+ \to X_-.
\]
The closedness of $T$ implies that $T_+$ and $T_-$ are closed operators by \cite[Proposition A.2.8 p.~275]{Haa06} between the corresponding Banach spaces. For any $x \in \dom T$, write with \eqref{decompo-directe} the decomposition $x = x_+ + x_-$ with $x_\pm \in \dom T \cap X_\pm$. Then we have
\[
Tx 
= T x_+ + T x_- 
= T_+ x_+ + T_- x_-,
\]
which is the claimed off-diagonal form relative to the decomposition $X = X_+ \oplus X_-$.
\end{proof}

The next result on the structure of operators commuting with a bounded symmetry is elementary.

\begin{prop}
\label{prop:commuting-symmetry-diagonal}
Let $X$ be a Banach space and let $\gamma \co X \to X$ be a bounded symmetry, that is, $\gamma^2=\Id_X$. Consider the closed subspaces $
X_+ \ov{\mathrm{def}}{=} \ker(\gamma-\Id_X)$ and $
X_- \ov{\mathrm{def}}{=} \ker(\gamma+\Id_X)$. Let $T \co \dom T \subset X \to X$ be a closed operator such that $
\gamma(\dom T) \subset \dom T$. Suppose that
\[
T\gamma x=\gamma Tx, \quad x\in\dom T.
\]
Then $
\dom T
=
(\dom T \cap X_+)
\oplus
(\dom T \cap X_-)$, $
T(\dom T \cap X_+)\subset X_+$, $T(\dom T \cap X_-)\subset X_-$. In particular, the restrictions
\[
T_+
\ov{\mathrm{def}}{=}
T\big|_{\dom T\cap X_+}
\co
\dom T\cap X_+
\to X_+
\quad \text{and} \quad
T_-
\ov{\mathrm{def}}{=}
T\big|_{\dom T\cap X_-}
\co
\dom T\cap X_-
\to X_-
\]
are closed operators and, with respect to the decomposition $X=X_+\oplus X_-$, the operator $T$ has the diagonal form 
$T=
\begin{bmatrix}
T_+ & 0\\
0 & T_-
\end{bmatrix}$.
\end{prop}

\begin{proof}
Since $\gamma^2=\Id_X$, the bounded operators $
P_+
\ov{\mathrm{def}}{=}
\frac12(\Id_X+\gamma)
$ and $
P_-
\ov{\mathrm{def}}{=}
\frac12(\Id_X-\gamma)
$ are complementary bounded projections with ranges $X_+$ and $X_-$, respectively. Hence
\begin{equation}
\label{decompo-directe-commuting}
X=X_+\oplus X_-.
\end{equation}
Moreover, since $\gamma(\dom T)\subset\dom T$, both $P_+$ and $P_-$ leave $\dom T$ invariant. Consequently,
\[
\dom T
=
(\dom T\cap X_+)
\oplus
(\dom T\cap X_-).
\]

Assume now that $T\gamma x=\gamma Tx$ for all $x\in\dom T$. Let $x_+\in\dom T\cap X_+$. Then $\gamma x_+=x_+$. Applying the commutation relation to $x_+$ yields $
T\gamma x_+=\gamma Tx_+$. Since $T\gamma x_+=Tx_+$, we obtain $
\gamma Tx_+=Tx_+$. Thus $Tx_+$ is a $(+1)$-eigenvector of $\gamma$, and therefore $Tx_+\in X_+$. Hence $
T(\dom T\cap X_+)\subset X_+$. Similarly, if $x_-\in\dom T\cap X_-$, then $\gamma x_-=-x_-$. Applying the commutation relation to $x_-$ gives $T\gamma x_-=\gamma Tx_-$. The left-hand side is $T\gamma x_-=T(-x_-)=-Tx_-$, so $-Tx_-=\gamma Tx_-$. Thus $Tx_-$ is a $(-1)$-eigenvector of $\gamma$, and therefore $Tx_-\in X_-$. Hence $
T(\dom T\cap X_-)\subset X_-$. Now, we introduce the operators
\[
T_+
\ov{\mathrm{def}}{=}
T\big|_{\dom T\cap X_+}
\co
\dom T\cap X_+
\to X_+
\quad \text{and} \quad
T_-
\ov{\mathrm{def}}{=}
T\big|_{\dom T\cap X_-}
\co
\dom T\cap X_-
\to X_-.
\]
The closedness of $T$ implies that $T_+$ and $T_-$ are closed operators by \cite[Proposition A.2.8 p.~275]{Haa06} between the corresponding Banach spaces. For any $x\in\dom T$, write with \eqref{decompo-directe-commuting} the decomposition $x=x_++x_-$ with $x_\pm\in\dom T\cap X_\pm$. Then
\[
Tx
=
Tx_++Tx_-
=
T_+x_++T_-x_-,
\]
which is the claimed diagonal form relative to the decomposition $X=X_+\oplus X_-$.
\end{proof}


\paragraph{Declaration of interest} None.

\paragraph{Competing interests} The author declares that he has no competing interests. 

\paragraph{Data availability} No data sets were generated during this study.

\paragraph{Acknowledgments} The author would like to thank Jan van Neerven and Dmitriy Zanin for short discussions.

\paragraph{AI disclosure statement} The author acknowledges the use of ChatGPT (OpenAI) for proofreading and improving the clarity of the manuscript. All mathematical content, arguments, and conclusions were independently developed and verified by the author.



{\footnotesize

\vspace{0.2cm}

\noindent C\'edric Arhancet\\ 
\noindent 6 rue Didier Daurat, 81000 Albi, France\\
URL: \href{https://sites.google.com/site/cedricarhancet}{https://sites.google.com/site/cedricarhancet}\\
cedric.arhancet@protonmail.com\\
ORCID: 0000-0002-5179-6972 
}

\end{document}